\documentclass[final,onefignum,onetabnum]{siamart220329}

\usepackage{lmodern}
\usepackage{braket,amsfonts}
\usepackage{stmaryrd}
\usepackage{amsmath}

\usepackage{mathtools}
\usepackage{amssymb}
\usepackage{booktabs}
\usepackage{algorithmic}
\usepackage{algorithm}
\usepackage{array}
\usepackage{subcaption}
\usepackage{graphicx,epstopdf}
\usepackage{amsopn}
\usepackage{xspace}
\usepackage{bold-extra}
\usepackage[most]{tcolorbox}
\usepackage{amsmath}
\usepackage{enumitem}
\usepackage{amsfonts}
\usepackage{multirow} 
\usepackage{cleveref} 

\makeatletter
\def\refstepcounter@noarg#1{%
	\cref@old@refstepcounter{#1}%
	\cref@constructprefix{#1}{\cref@result}%
	\@ifundefined{cref@#1@alias}%
	{\def\@tempa{#1}}%
	{\def\@tempa{\csname cref@#1@alias\endcsname}}%
	\protected@edef\cref@currentlabel{%
		[\@tempa][\arabic{#1}][\cref@result]%
		\csname p@#1\endcsname\csname the#1\endcsname}%
}
\def\refstepcounter@optarg[#1]#2{%
	\cref@old@refstepcounter{#2}%
	\cref@constructprefix{#2}{\cref@result}%
	\@ifundefined{cref@#1@alias}%
	{\def\@tempa{#1}}%
	{\def\@tempa{\csname cref@#1@alias\endcsname}}%
	\protected@edef\cref@currentlabel{%
		[\@tempa][\arabic{#2}][\cref@result]%
		\csname p@#2\endcsname\csname the#2\endcsname}%
}
\makeatother

\numberwithin{equation}{section}
\newtheorem{remark}[theorem]{Remark}
\makeatletter
\@ifundefined{proposition}{}{}
\makeatother

\newcommand{\eqm}{\begin{eqnarray}}
	\newcommand{\enm}{\end{eqnarray}}
\crefname{hypothesis}{Hypothesis}{Hypotheses}
\usepackage[a4paper, margin=1in]{geometry} 

\Crefname{ALC@unique}{Line}{Lines}

\newcommand{\erro}{\mathrm{error}}
\newtheorem{assumption}[theorem]{Assumption}
\colorlet{texcscolor}{blue!50!black}
\colorlet{texemcolor}{red!70!black}
\colorlet{texpreamble}{red!70!black}
\colorlet{codebackground}{black!25!white!25}

\lstdefinestyle{siamlatex}{%
	style=tcblatex,
	texcsstyle=*\color{texcscolor},
	texcsstyle=[2]\color{texemcolor},
	keywordstyle=[2]\color{texemcolor},
	moretexcs={cref,Cref,maketitle,mathcal,text,headers,email,url},
}

\tcbset{%
	colframe=black!75!white!75,
	coltitle=white,
	colback=codebackground, 
	colbacklower=white, 
	fonttitle=\bfseries,
	arc=0pt,outer arc=0pt,
	top=1pt,bottom=1pt,left=1mm,right=1mm,middle=1mm,boxsep=1mm,
	leftrule=0.3mm,rightrule=0.3mm,toprule=0.3mm,bottomrule=0.3mm,
	listing options={style=siamlatex}
}

\newtcblisting[use counter=example]{example}[2][]{%
	title={Example~\thetcbcounter: #2},#1}

\newtcbinputlisting[use counter=example]{\examplefile}[3][]{%
	title={Example~\thetcbcounter: #2},listing file={#3},#1}
\allowdisplaybreaks
\DeclareTotalTCBox{\code}{ v O{} }
{ 
	fontupper=\ttfamily\color{black},
	nobeforeafter,
	tcbox raise base,
	colback=codebackground,colframe=white,
	top=0pt,bottom=0pt,left=0mm,right=0mm,
	leftrule=0pt,rightrule=0pt,toprule=0mm,bottomrule=0mm,
	boxsep=0.5mm,
	#2}{#1}

\patchcmd\newpage{\vfil}{}{}{}
\begin{tcbverbatimwrite}{tmp_\jobname_header.tex}
	\title{A Perturbation-Correction Method Based on Local Randomized Neural Networks for Quasi-Linear Interface Problems\thanks{\textbf{Corresponding Author:} School of Mathematical Sciences, Ministry of Education Key Laboratory of NSLSCS, Nanjing Normal University, Nanjing, 210023, China (\email{zhangzhiyue@njnu.edu.cn}).\funding{This work is supported by the National Natural Science Foundation of China (No.~12571431).}}}
	
	\author{Siyuan Lang \thanks{School of Mathematical Sciences, Ministry of Education Key Laboratory of NSLSCS, Nanjing Normal University, Nanjing, 210023, China, (\email{2679040392@qq.com}).}
		\and Zhiyue Zhang*}
	
\end{tcbverbatimwrite}
\input{tmp_\jobname_header.tex}

\ifpdf
\hypersetup{ pdftitle={A Perturbation-Correction Method Based on Local Randomized Neural Networks for Quasi-Linear Interface Problems} }
\fi

\begin{document}
	\setlength{\abovedisplayskip}{5pt}
	\setlength{\belowdisplayskip}{5pt}
	
	\maketitle
	
	\begin{tcbverbatimwrite}{tmp_\jobname_abstract.tex}
		\begin{abstract}
~\par
For quasi-linear elliptic interface problems with discontinuous diffusion coefficients, randomized neural network approximations may exhibit stagnation because the associated objective functional is generally nonconvex. This paper proposes a Local Randomized Neural Network (LRaNN) perturbation-correction method, denoted by LRaNN-PC, to alleviate this stagnation. The method represents the solution using an LRaNN on each subdomain, coupled through the interface conditions in a domain-decomposed framework. It consists of a primary stage and a perturbation-correction (PC) stage. The primary stage computes the primary approximation by minimizing the original nonconvex objective functional. The PC stage constructs a residual-driven correction by performing a local expansion of the residual around the primary approximation and representing the correction in an independently generated randomized trial space. The correction coefficients are obtained by solving a least-squares residual-correction subproblem in this trial space. For the solution-dependent quasi-linear elliptic model under the stated sufficient assumptions, we derive a residual-controlled upper bound for the broken $H^1$ seminorm error. The bound involves the discrete residual, the quadrature error, the residual of the perturbation-correction subproblem, and the truncation remainder of the perturbation expansion. Numerical experiments further test irregular interfaces and high-contrast coefficients within this setting, and also examine gradient-dependent diffusivities and moving-interface extensions. In the tested benchmarks, LRaNN-PC reduces the relative $L^2$ error by up to $4$--$7$ orders of magnitude compared with the primary LRaNN stage.
\end{abstract}

		
	\end{tcbverbatimwrite}
	\input{tmp_\jobname_abstract.tex}

\section{Introduction}\label{sec1}
Interface problems with nonlinear diffusion coefficients~\cite{ying2004decomposition} arise in many physical and engineering models. Examples include composite material manufacturing~\cite{material_mechanics,donato2015homogenization}, battery electrochemistry~\cite{kirk2023nonlinear,zeng2013efficient}, and biological transport processes~\cite{brenner2013interfacial,teigen2009diffuse}. In these problems, discontinuities across material interfaces are intrinsic to the model. Achieving high accuracy therefore requires both reliable treatment of the interface conditions and stable handling of the nonlinearity.

\par Classical numerical methods have provided effective tools for interface problems, particularly in the linear setting.  Examples include finite difference methods (FDM)~\cite{wang2018convergence},
finite volume methods (FVM)~\cite{wang2021new,zhao2021semi},
the immersed boundary method (IBM)~\cite{IBM1977},
immersed interface methods (IIM)~\cite{li2006immersed,leveque1994immersed},
the matched interface and boundary method (MIB)~\cite{yu2007matched,zhou2006high,MIB2014},
ghost fluid methods (GFM)~\cite{fedkiw1999non,liu2000boundary,gfm2003},
and finite element methods (FEM)~\cite{babuvska1970finite,hansbo2002unfitted,burman2015cutfem}. Additional results for linear interface problems can be found in~\cite{kirchhart2016analysis,DG2012,lehrenfeld2013analysis} and the references therein. More recently, mesh-based formulations have been extended to nonlinear cases. For example, immersed methods on interface-independent meshes have been developed for quasi-linear elliptic problems with discontinuous diffusion under homogeneous jump conditions~\cite{adjerid2025immersed}.

\par In recent years, mesh-free neural network methods have provided an alternative approximation framework for partial differential equations (PDE) problems. Physics-informed neural networks (PINNs)~\cite{raissi2019physics}, the deep Galerkin method~\cite{sirignano2018dgm}, and the deep Ritz method~\cite{yu2018deep} formulate PDE approximation through residual or variational optimization. Multi-stage training strategies have also been investigated in the PINN framework~\cite{wang2024multi,aldirany2024multi}. Randomized neural networks (RaNNs)~\cite{pao1994learning} and extreme learning machines (ELMs)~\cite{huang2006extreme} are particularly relevant here. Their hidden-layer parameters are randomly generated and fixed, while only the output coefficients are trained. These models have approximation guarantees under suitable assumptions~\cite{igelnik1995stochastic,lin2014elm}.

\par Within this mesh-free framework, neural network methods have also been developed for interface problems and can be organized into several directions.
One direction uses PINNs with network designs specialized for interface problems, including discontinuity-capturing shallow networks~\cite{hu2022discontinuity},
cusp-capturing PINNs~\cite{tseng2023cusp}, interface PINNs~\cite{sarma2024interface},
adaptive interface PINNs~\cite{roy2024adaptive}, hard-constraint PINNs~\cite{lai2025hard},
piecewise deep networks~\cite{he2022mesh},
and PINNs for two-phase and multiphase problems with moving interfaces~\cite{zhu2023physics,li2025physics,zhai2026physics}.
Another direction uses randomized neural networks (RaNNs) and extreme learning machines (ELMs).
Examples include the random feature method~\cite{chen2022bridging,chi2024random},
local RaNNs based on subdomain decomposition, which can be interpreted as piecewise randomized neural network representations~\cite{li2023local},
local RaNNs with finite-difference assembly~\cite{li2025local},
ELM with nonoverlapping domain decomposition~\cite{lee2025nonoverlapping},
and PIELM for evolving interfaces and Stefan problems~\cite{zeng2025high,ren2025physics}.
A third direction couples neural networks with classical discretizations, such as hybrid NN--MAC
schemes~\cite{chang2023hybrid} and immersed interface neural networks~\cite{zhang2025immersed}.
Operator learning offers another route to interface problems~\cite{wu2024solving,bi2025xi,fan2024decoupling,fan2026hybrid}, and mesh-free kernel methods extend to nonlinear diffusion problems with jump coefficients~\cite{liu2023meshfree}. Beyond the interface setting, related neural and randomized formulations target nonlinear elliptic and parabolic equations~\cite{dong2021local,lu2021deepxde,kharazmi2021hp}, with rigorous analysis of the approximation and generalization properties of neural PDE approximations~\cite{mishra2022estimates,mishra2023estimates}.
Nevertheless, high-accuracy randomized or neural approximations for quasi-linear interface problems are still challenging in many settings. A central difficulty is that the objective functional induced by the nonlinear operator is generally nonconvex, so the primary approximation may stagnate at a nonzero residual level even when the randomized basis appears sufficiently expressive in practice. This observation motivates the residual-driven correction stage introduced below.

\par In this paper, we propose a two-stage perturbation-correction method within the LRaNN
framework, denoted LRaNN-PC. The local representation assigns an independent randomized neural network to each subdomain, with randomly generated and fixed hidden parameters, and the subdomain networks are coupled through the interface conditions. The primary stage minimizes the nonconvex least-squares objective and computes the primary approximation $u_N$. After setting $\epsilon$ to the scale of the residual norm of $u_N$, the PC stage expands the residual around $u_N$ and computes a scaled correction of the form $\epsilon u_p$, where $u_p$ is represented in an independent randomized trial space with the same subdomain decomposition and interface coupling. The corrected approximation is then defined by $u_h=u_N+\epsilon u_p$. The correction coefficients are computed by solving a least-squares residual-correction subproblem in an independent trial space. Its leading-order Gauss--Newton linearized form is convex with respect to the correction coefficients, while the second-order form is generally nonlinear and is treated as a local residual-correction refinement in the numerical algorithm. For the solution-dependent quasi-linear elliptic interface model under the stated sufficient assumptions, we derive a residual-controlled upper bound on the broken $H^1$ seminorm error. The method is tested on solution-dependent quasi-linear elliptic interface problems with irregular interfaces and high-contrast coefficients, and is further examined through numerical extensions involving gradient-dependent diffusion, moving interfaces, a physically motivated problem without a closed-form solution, and a benchmark comparison with published immersed finite element (IFE) results.

\par The main contributions are as follows:
\begin{enumerate}
\item A two-stage LRaNN-PC framework for quasi-linear elliptic interface problems, combining a primary
LRaNN stage with a residual-driven PC stage on an independent randomized basis. A separate
network is assigned to each subdomain and each stage, coupled through the interface conditions.

\item A least-squares residual-correction subproblem on the independent randomized basis. Its leading-order form is convex, while the second-order form is generally nonlinear and is solved as a local residual-correction problem.

\item A residual-controlled error estimate for the model problem considered in this paper under the stated sufficient assumptions. The
broken $H^1$ seminorm error is bounded by the discrete residual, the quadrature error, the
PC-stage residual, and the truncation remainder; the estimate identifies the PC-stage residual as one computable contribution to the residual-controlled error bound.

\item Numerical evaluation on the elliptic model considered in this paper, together with gradient-dependent and moving-interface extensions and an IFE benchmark, reducing the
relative $L^2$ error by about $4$--$7$ orders of magnitude over the primary LRaNN approximation in the tested examples.
\end{enumerate}

The remainder of this paper is organized as follows.  
Sec~\ref{sec:main} introduces the quasi-linear elliptic interface model and assumptions.  
Sec~\ref{sec:3} describes the LRaNN representation and the two-stage perturbation-correction framework.  
Sec~\ref{sec:error analysis} presents the residual-controlled error upper bound and the residual-correction subproblem for the corrected approximation.
Sec~\ref{sec:4} reports numerical experiments, including elliptic tests and numerical extensions, and Sec~\ref{sec:conclusion} concludes the paper.

\section{Problem Formulation}
\label{sec:main}
We consider a class of quasi-linear elliptic interface problems with 
solution-dependent diffusion coefficients, source terms, 
and interface jump conditions:
\begin{equation}
\label{model of the problem}
\begin{cases}
-\nabla\cdot\!\left(\beta^+\!\left(x,u^+\right)\nabla u^+\right) = f^+\!\left(x,u^+\right), & x \in \Omega^+, \\[2mm]
-\nabla\cdot\!\left(\beta^-\!\left(x,u^-\right)\nabla u^-\right) = f^-\!\left(x,u^-\right), & x \in \Omega^-, \\[2mm]
\llbracket u(x) \rrbracket = w(x), & x \in \Gamma, \\[2mm]
\llbracket \beta(x,u)\,\partial_{\mathbf n}u \rrbracket = v(x), & x \in \Gamma, \\[2mm]
u(x) = g(x), & x \in \partial\Omega,
\end{cases}
\end{equation}
where $\beta(x,u)$ denotes the piecewise nonlinear coefficient
\begin{displaymath}
\beta(x,u) =
\begin{cases}
\beta^+(x,u^+), & x \in \Omega^+,\\
\beta^-(x,u^-), & x \in \Omega^-.
\end{cases}
\end{displaymath}

Let $\Omega \subset \mathbb{R}^d$ be a bounded domain with Lipschitz boundary 
$\partial\Omega$. The interface $\Gamma$ partitions $\Omega$ into two disjoint 
open Lipschitz subdomains $\Omega^+$ and $\Omega^-$, as depicted in 
Fig~\ref{fig:interface_types}, with
\begin{displaymath}
\overline{\Omega} = \overline{\Omega^+} \cup \overline{\Omega^-}, \quad 
\Omega^+ \cap \Omega^- = \emptyset, \quad 
\Gamma = \partial\Omega^+ \cap \partial\Omega^-.
\end{displaymath}
The unit normal $\mathbf{n}$ on $\Gamma$ points from $\Omega^+$ towards 
$\Omega^-$, and $\partial_{\mathbf n}u := \nabla u\cdot\mathbf{n}$. Specifically, for $x\in\Gamma$,
\begin{equation}
\label{eq:jump-definitions}
\begin{aligned}
\llbracket u\rrbracket\big|_{\Gamma}
&= \lim_{x\to\Gamma}u^+(x) - \lim_{x\to\Gamma}u^-(x) = w(x),\\
\llbracket \beta(x,u)\,\partial_{\mathbf n}u\rrbracket\big|_{\Gamma}
&= \lim_{x\to\Gamma}\beta^+(x,u^+)\nabla u^+\!\cdot\mathbf{n}
 - \lim_{x\to\Gamma}\beta^-(x,u^-)\nabla u^-\!\cdot\mathbf{n}
 = v(x).
\end{aligned}
\end{equation}
The interface is assumed to be sufficiently regular so that $\mathbf{n}$ and the 
traces in the jump conditions are well defined. The proposed framework applies 
to both cut (open) and embedded (closed) interface configurations shown in 
Fig~\ref{fig:interface_types}.

Let $V:=H^1(\Omega^+)\oplus H^1(\Omega^-)$ be the broken Sobolev space associated 
with the partition, equipped with the norm 
$$\|\phi\|_V^2:=\|\phi^+\|_{H^1(\Omega^+)}^2+\|\phi^-\|_{H^1(\Omega^-)}^2.$$ 
The coefficients $\beta^\pm$ and the sources $f^\pm$ are assumed to be sufficiently smooth in $(x,s)$ on 
$\overline{\Omega^\pm}\times\mathbb{R}$, and the data $w$, $v$, and $g$ are assumed sufficiently smooth on 
$\Gamma$ and $\partial\Omega$. These assumptions are used to make the strong-form residuals and their local expansions well defined when evaluated on the smooth LRaNN approximations. The reference solution \(u_g\) is assumed to be a piecewise \(H^1\) weak solution satisfying the flux regularity stated in Assumption~\ref{ass:weak-existence}. The unknown field is 
written $u=(u^+,u^-)$, and $u_g=(u_g^+,u_g^-)\in V$ denotes the weak solution of 
\eqref{model of the problem} introduced in Assumption~\ref{ass:weak-existence}; it 
serves as the reference solution in the error analysis of 
Sec~\ref{sec:error analysis} and the numerical comparisons of Sec~\ref{sec:4}.

\par The diffusion coefficient $\beta^\pm(x,\cdot):\mathbb{R}\to\mathbb{R}$ 
depends on $u^\pm$ in a possibly nonlinear way. We assume uniform ellipticity: 
there exists $\delta > 0$ such that
\begin{equation}
\label{eq:uniform ellipticity}
\beta^\pm(x,s) \geq \delta, \quad \forall x \in \Omega^\pm, \quad \forall s \in \mathbb{R}.
\end{equation}
We further assume that the associated quasi-linear elliptic operator is 
strongly monotone: there exists $\lambda^\pm > 0$ such that for all 
$\xi,\zeta \in H^1(\Omega^\pm)$,
\begin{equation}
\label{eq:strong-monotonicity}
\int_{\Omega^\pm} 
\big(\beta^\pm(x,\xi)\nabla \xi - \beta^\pm(x,\zeta)\nabla \zeta\big)\cdot(\nabla \xi-\nabla \zeta)\,dx 
\geq \lambda^\pm \|\nabla \xi - \nabla \zeta\|_{L^2(\Omega^\pm)}^2.
\end{equation}

\par The source term $f^\pm(x,s)$ is assumed to be one-sided monotone in $s$ 
on an interval $\mathcal{I}^\pm \subseteq \mathbb{R}$: for all 
$x \in \Omega^\pm$ and all $s,t \in \mathcal{I}^\pm$,
\begin{equation}
\label{eq:one-sided-monotonicity}
\big(f^\pm(x,s) - f^\pm(x,t)\big)(s-t) \leq 0.
\end{equation}
In the residual estimate below, this condition is used for approximations $u_*$
whose range, together with the range of $u_g$, is contained in $\mathcal I^\pm$.
If $f^\pm$ is differentiable in $s$, this condition is implied by
\begin{equation}
\label{eq:source derivative}
\partial_s f^\pm(x,s) \leq 0, \quad x \in \Omega^\pm,\quad s \in \mathcal{I}^\pm.
\end{equation}

\begin{assumption}[Existence and flux regularity of a weak solution]\label{ass:weak-existence}
Problem \eqref{model of the problem} admits a weak solution $u_g=(u_g^+,u_g^-)\in V$
that attains the boundary datum $g$ on $\partial\Omega$ and the interface data $w$
and $v$ on $\Gamma$ in the sense of traces. Moreover, the subdomain fluxes
$\beta^\pm(x,u_g^\pm)\nabla u_g^\pm$ belong to $H(\mathrm{div},\Omega^\pm)$ and admit
$L^2$ normal traces on $\partial\Omega$ and $\Gamma$.
\end{assumption}

For transmission problems of this type, existence and uniqueness can be obtained by standard monotone-operator arguments under additional structural assumptions on the operator and the data~\cite{showalter2013monotone,zeidler2013nonlinear}. Since the precise conditions depend on the particular nonlinearity and on the interface data, we impose Assumption~\ref{ass:weak-existence} directly and use it only as the assumption of reference solution in the residual-controlled estimate.

\begin{remark}\label{rmk:monotonicity}
Some coefficients used in Sec~\ref{sec:4} (for example $\beta=1+u$ or 
$\beta=1+0.1\,u^3$) do not satisfy the uniform ellipticity 
\eqref{eq:uniform ellipticity} on all of $\mathbb{R}$, but only on the range of 
the corresponding solution. For these examples, uniform ellipticity is understood on the solution range used in the computation. We do not verify the global strong monotonicity condition for all numerical examples; those cases should be interpreted as numerical tests beyond the sufficient assumptions used in Theorem~\ref{thm:residual-bound}. The LRaNN-PC algorithm does not require strong monotonicity for its implementation, although the residual-controlled analysis in Section~\ref{sec:error analysis} relies on this assumption.
\end{remark}

\begin{remark}\label{rmk:source}
The monotonicity condition on $f^\pm$ is a sufficient assumption used in the 
residual-controlled estimate. It can be relaxed to cases in which the positive 
one-sided Lipschitz part of $f^\pm$ is small relative to the monotonicity 
constant of the diffusion operator. In the numerical experiments, the 
perturbation-correction stage may reduce the residual beyond this sufficient 
theoretical regime.
\end{remark}

\begin{figure}[t]
\centering
\begin{subfigure}[t]{0.4\textwidth}
\centering
\includegraphics[width=\textwidth]{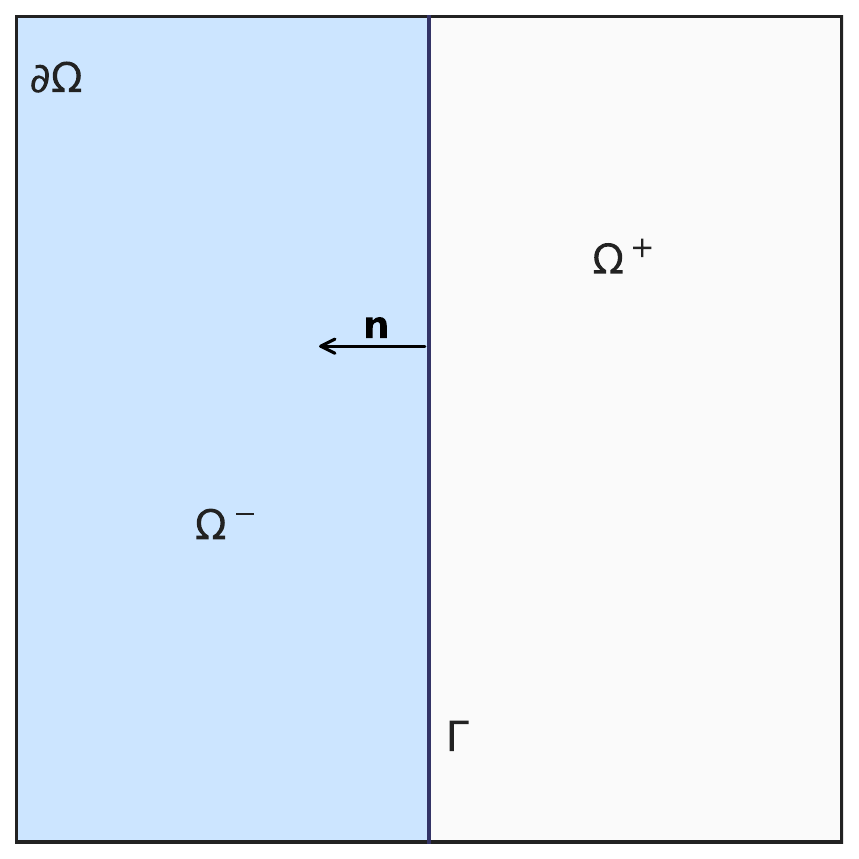}
\caption{Cut (open) interface.}
\label{fig:cut_interface}
\end{subfigure}
\hspace{0.04\textwidth}
\begin{subfigure}[t]{0.4\textwidth}
\centering
\includegraphics[width=\textwidth]{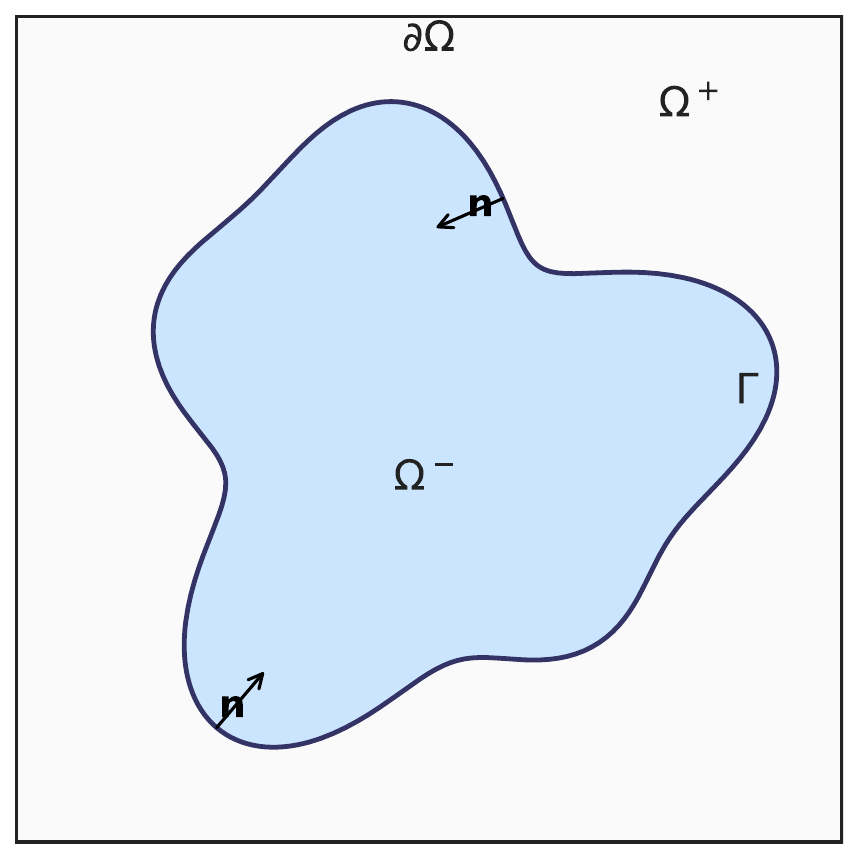}
\caption{Embedded (closed) interface.}
\label{fig:embedded_interface}
\end{subfigure}
\caption{
Two typical geometric configurations for interface problems. 
\textbf{(a)}~A cut (open) interface that intersects the exterior boundary, 
partitioning the domain into disconnected subdomains.
\textbf{(b)}~An embedded (closed) interface located entirely within the domain, 
separating it into interior and exterior regions.
}
\label{fig:interface_types}
\end{figure}
\section{Methodology}
\label{sec:3}
This section presents the numerical framework used to solve the quasi-linear interface problem. We first introduce the local randomized neural network (LRaNN) representation, then formulate the strong-form least-squares residual used for training. We then describe the two consecutive stages in the proposed framework: the primary LRaNN stage and the subsequent perturbation-correction stage.
\subsection{Local Randomized Neural Network (LRaNN) Architecture}
\label{subsec:LRaNNs}
Randomized neural networks (RaNNs) approximate functions by fixing the hidden-layer parameters after random initialization and optimizing only the output-layer coefficients. We use a single-hidden-layer feedforward architecture, as illustrated in Fig~\ref{fig:ELM}. For an input $\mathbf{x} \in \mathbb{R}^d$, the network output is a linear combination of $m$ nonlinear activation functions:
\begin{equation}
\label{eq:RaNN}
u(\mathbf{x}) = \sum_{j=1}^m \alpha_j \phi(\mathbf{w}_j \cdot \mathbf{x} + b_j) = \Phi(\mathbf{x})\boldsymbol{\alpha},
\end{equation}
where $\mathbf{w}_j\in\mathbb{R}^d$ and $b_j\in\mathbb{R}$ are fixed random weights and biases, $\phi:\mathbb{R}\to\mathbb{R}$ is a prescribed activation function, and $\alpha_j\in\mathbb{R}$ are the trainable output coefficients collected in $\boldsymbol{\alpha}$.
\par As illustrated in~Fig~\ref{fig:interface_types}, the interface $\Gamma$ partitions the computational domain into subdomains. Across $\Gamma$, the solution may be discontinuous or have reduced regularity. 
Following domain-decomposed neural-network formulations such as XPINNs and 
hp-VPINNs~\cite{jagtap2020extended,kharazmi2021hp}, as well as local 
randomized-network methods for interface problems~\cite{li2023local,li2025local}, 
we use a local RaNN representation, assigning one randomized network to each 
subdomain.
\par We describe the two-subdomain case for clarity. For \eqref{model of the problem}, we use two LRaNNs, $u_N^+$ in $\Omega^+$ and $u_N^-$ in $\Omega^-$. Both networks use $m$ hidden neurons. Their outputs are
\begin{equation}
\label{local RaNNs}
u_N^+ = \sum_{j=1}^{m}\alpha_j^+\phi^+\left(\mathbf{w}_j^+\cdot \mathbf{x}+b_j^+\right), \quad u_N^- = \sum_{j=1}^{m}\alpha_j^-\phi^-\left(\mathbf{w}_j^-\cdot \mathbf{x}+b_j^-\right).
\end{equation}
Here, $\boldsymbol{\alpha}^+=(\alpha_1^+,\dots,\alpha_m^+)$ and $\boldsymbol{\alpha}^-=(\alpha_1^-,\dots,\alpha_m^-)$ are the trainable output coefficients. The hidden-layer weights $\mathbf{w}_j^\pm$ and biases $b_j^\pm$ are fixed after random initialization, so each activation acts as a predetermined basis function on its subdomain. The unknowns in the optimization are therefore only the output coefficients $\boldsymbol{\alpha}^\pm$. The representation \eqref{local RaNNs} does not enforce the boundary or
interface conditions by construction; these conditions are imposed weakly
through the least-squares residuals introduced in Sec~\ref{subsec:LS}, so the
approximation satisfies them only approximately, at the level of the
corresponding residual.
\begin{figure}[t]
\centering
\includegraphics[width=0.5\textwidth]{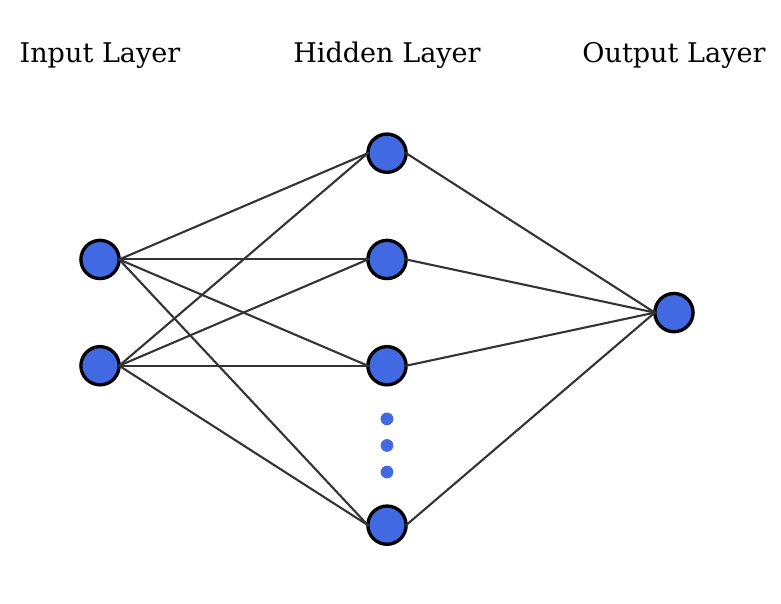}
\caption{Structure of the randomized neural network used in each subdomain.}
\label{fig:ELM}
\end{figure}

\subsection{Discrete Least-Squares Functional}
\label{subsec:LS}
We solve the LRaNN approximation by minimizing strong-form residuals of the PDE, interface conditions, and boundary condition. This gives a least-squares (LS) formulation in which all constraints in \eqref{model of the problem} are imposed through residual penalties.

The strong-form residuals are evaluated only on the randomized-network approximations, which are smooth functions for the activation functions used here. Therefore, pointwise residual evaluation requires the coefficient, source, boundary, and interface data to be pointwise defined and sufficiently smooth.

\par Using the LRaNN approximations $u_N^+$ and $u_N^-$ defined in \eqref{local RaNNs}, we define the following residual functions:
\par For the interior of each subdomain, the residual corresponds to the violation of the governing PDE. Specifically, for $x \in \Omega^+$ we define
\begin{displaymath}
\mathcal{R}_{\Omega^{+}}(x) := \nabla \cdot \left( \beta^{+}\big(x, u_{N}^{+}(x)\big) \nabla u_{N}^{+}(x) \right) + f^{+}\big(x, u_{N}^{+}(x)\big), \quad x \in \Omega^+,
\end{displaymath}
whereas for $x \in \Omega^{-}$ the corresponding expression is
\begin{displaymath}
\mathcal{R}_{\Omega^{-}}(x):= \nabla \cdot \left( \beta^{-}\big(x, u_{N}^{-}(x)\big) \nabla u_{N}^{-}(x) \right) + f^{-}\big(x, u_{N}^{-}(x)\big), \quad x \in \Omega^{-}.
\end{displaymath}
\par Across the interface $\Gamma$, we define the residual for the jump condition as
\begin{displaymath}
\mathcal{R}_{\Gamma_{d}}(x):= u_{N}^{+}\big(x\big) - u_{N}^{-}\big(x\big) - w\big(x\big), \quad x \in \Gamma, 
\end{displaymath}
and the residual for the flux jump as
\begin{displaymath}
\mathcal{R}_{\Gamma_{n}}(x) := \left[ \beta^{+}\big(x, u_{N}^{+}(x)\big) \partial_{\mathbf{n}}u_N^+(x) - \beta^{-}\big(x, u_{N}^{-}(x)\big) \partial_{\mathbf{n}}u_N^-(x) \right] - v\big(x\big), \quad x \in \Gamma.
\end{displaymath}
\par On the outer boundary $\partial\Omega$, we impose Dirichlet conditions, with the residual
\begin{displaymath}
\mathcal{R}_{\partial\Omega}(x) := u_{N}^{\pm}\big(x\big) - g\big(x\big), \quad x \in \partial\Omega.
\end{displaymath}

\par Collecting these terms, we define the continuous least-squares functional
\begin{equation}
\label{eq:LS-formulation}
\begin{aligned}
	\min_{\boldsymbol{\alpha}}\frac 12\Vert\mathcal{F}(\boldsymbol{\alpha})\Vert_{L^2}^2
	=\min_{\boldsymbol{\alpha}}\bigg\{&
	\frac{\omega_{\Omega^+}}{2}\Vert\mathcal{R}_{\Omega^+}\Vert^2_{L^2} + \frac{\omega_{\Omega^-}}{2}\Vert\mathcal{R}_{\Omega^-}\Vert^2_{L^2} +
	\\
	&\frac{\omega_{\Gamma_d}}{2}\Vert\mathcal{R}_{\Gamma_d}\Vert_{L^2}^2 + 
	\frac{\omega_{\Gamma_n}}{2}\Vert\mathcal{R}_{\Gamma_n}\Vert^2_{L^2} + 
	\frac{\omega_{\partial\Omega}}{2}\Vert\mathcal{R}_{\partial\Omega}\Vert^2_{L^2}\bigg\},
\end{aligned}
\end{equation}
where $\omega_{\Omega^+}$, $\omega_{\Omega^-}$, $\omega_{\Gamma_d}$, $\omega_{\Gamma_n}$, and $\omega_{\partial\Omega}$ are positive weights used to balance the residual terms. 

\par The $L^2$ norms in \eqref{eq:LS-formulation} are approximated from collocation points sampled in each subdomain, on the interface, and on the boundary:
\begin{equation}
\label{eq:collocation points}
\begin{aligned}
	&\mathcal{S}_{\Omega^+}=\{x_i^{\Omega^+}\}_{i=1}^{N_{\Omega^+}}\subseteq\Omega^+,
	\\
	&\mathcal{S}_{\Omega^{-}}=\{x_i^{\Omega^-}\}_{i=1}^{N_{\Omega^-}}\subseteq\Omega^-,
	\\
	&\mathcal{S}_{\Gamma_d}=\{x_i^{\Gamma_d}\}_{i=1}^{N_{\Gamma_d}}\subseteq\Gamma,
	\qquad
	\mathcal{S}_{\Gamma_n}=\{x_i^{\Gamma_n}\}_{i=1}^{N_{\Gamma_n}}\subseteq\Gamma,
	\\
	&\mathcal{S}_{\partial\Omega}=\{x_i^{\partial\Omega}\}_{i=1}^{N_{\partial\Omega}}\subseteq\partial\Omega.
\end{aligned}
\end{equation}
Here, $N_{\Omega^+}$, $N_{\Omega^-}$, $N_{\Gamma_d}$, $N_{\Gamma_n}$, and $N_{\partial\Omega}$ denote the numbers of sample points used for the two interior residuals, the two interface residuals, and the boundary residual, respectively. The two interface sets $\mathcal S_{\Gamma_d}$ and $\mathcal S_{\Gamma_n}$ may coincide in the implementation.

Substituting these sample sets into the residual definitions gives the weighted discrete residual vector
\begin{equation}
\label{eq:weighted_discrete_residual}
\mathcal F(\boldsymbol{\alpha})
=
\begin{pmatrix}
	\sqrt{\frac{\omega_{\Omega^+}}{N_{\Omega^+}}}
	\mathcal R_{\Omega^+}(x_i^{\Omega^+};\boldsymbol{\alpha}^+)
	\\[1mm]
	\sqrt{\frac{\omega_{\Omega^-}}{N_{\Omega^-}}}
	\mathcal R_{\Omega^-}(x_i^{\Omega^-};\boldsymbol{\alpha}^-)
	\\[1mm]
	\sqrt{\frac{\omega_{\Gamma_d}}{N_{\Gamma_d}}}
	\mathcal R_{\Gamma_d}(x_i^{\Gamma_d};\boldsymbol{\alpha})
	\\[1mm]
	\sqrt{\frac{\omega_{\Gamma_n}}{N_{\Gamma_n}}}
	\mathcal R_{\Gamma_n}(x_i^{\Gamma_n};\boldsymbol{\alpha})
	\\[1mm]
	\sqrt{\frac{\omega_{\partial\Omega}}{N_{\partial\Omega}}}
	\mathcal R_{\partial\Omega}(x_i^{\partial\Omega};\boldsymbol{\alpha})
\end{pmatrix}.
\end{equation}
In \eqref{eq:weighted_discrete_residual}, each block is understood as the vector obtained by evaluating the corresponding residual over all points in the associated collocation set.

The discrete least-squares problem is then
\begin{equation}
\label{eq:discrete form}
\begin{aligned}
	\frac{1}{2}\min_{\boldsymbol{\alpha}}\Vert\mathcal{F}(\boldsymbol{\alpha})\Vert^2_{\ell^2}
	=
	\frac 12\min_{\boldsymbol{\alpha}}\bigg\{&
	\frac{\omega_{\Omega^+}}{N_{\Omega^+}}\sum_{i=1}^{N_{\Omega^+}}|\mathcal{R}_{\Omega^+}(x_i^{\Omega^+};\boldsymbol{\alpha}^+)|^2+
	\frac{\omega_{\Omega^-}}{N_{\Omega^-}}\sum_{i=1}^{N_{\Omega^-}}|\mathcal{R}_{\Omega^-}(x_i^{\Omega^-};\boldsymbol{\alpha}^-)|^2
	\\
	&+
	\frac{\omega_{\Gamma_d}}{N_{\Gamma_d}}\sum_{i=1}^{N_{\Gamma_d}}|\mathcal{R}_{\Gamma_d}(x_i^{\Gamma_d};\boldsymbol{\alpha})|^2
	+\frac{\omega_{\Gamma_n}}{N_{\Gamma_n}}\sum_{i=1}^{N_{\Gamma_n}}|\mathcal{R}_{\Gamma_n}(x_i^{\Gamma_n};\boldsymbol{\alpha})|^2
	\\
	&+
	\frac{\omega_{\partial\Omega}}{N_{\partial\Omega}}\sum_{i=1}^{N_{\partial\Omega}}|\mathcal{R}_{\partial\Omega}(x_i^{\partial\Omega};\boldsymbol{\alpha})|^2\bigg\}.
\end{aligned}
\end{equation}
The output coefficients $\boldsymbol{\alpha}$ are obtained by solving \eqref{eq:discrete form}.

\subsection{Primary LRaNN Stage}
\label{subsec:Newton iterative}
We first solve the discrete nonlinear least-squares problem using the LRaNN ansatz introduced above. The primary LRaNN stage seeks
\begin{equation}
\label{eq:min}
\min_{\boldsymbol{\alpha}} G(\boldsymbol{\alpha})
= \min_{\boldsymbol{\alpha}} \frac{1}{2} \|\mathcal{F}(\boldsymbol{\alpha})\|_{\ell^2}^2,
\end{equation}
where $\boldsymbol{\alpha}=(\boldsymbol{\alpha}^+,\boldsymbol{\alpha}^-)$ is the concatenated vector of output coefficients. The weighted global residual vector $\mathcal{F}\in\mathbb{R}^{N_{\text{total}}}$, where
\[
N_{\text{total}}=N_{\Omega^+}+N_{\Omega^-}+N_{\Gamma_d}+N_{\Gamma_n}+N_{\partial\Omega},
\]
is formed by stacking all row-scaled residuals evaluated at the collocation points, as defined in \eqref{eq:weighted_discrete_residual}. Equivalently, before row scaling, the residual components are arranged as
\begin{equation}
\label{eq:F}
\begin{pmatrix}
	\mathcal{R}_{\Omega^+}
	\\
	\mathcal{R}_{\Omega^-}
	\\
	\mathcal{R}_{\Gamma_d}
	\\
	\mathcal{R}_{\Gamma_n}
	\\
	\mathcal{R}_{\partial\Omega}
\end{pmatrix}.
\end{equation}
The same row scaling is applied to the corresponding rows of the Jacobian. For clarity, the following formulas give the unweighted pointwise Jacobian blocks.

We solve \eqref{eq:min} with a regularized Gauss--Newton iteration, 
following the standard nonlinear least-squares framework~\cite{nocedal2006numerical,bjorck2024numerical}. At iteration $k$, the Hessian of $G$ is approximated by ${\mathcal{J}^{\{k\}}}^T\mathcal{J}^{\{k\}}$, where $\mathcal{J}^{\{k\}}=\partial\mathcal{F}^{\{k\}}/\partial\boldsymbol{\alpha}$ is the (row-scaled) Jacobian used in the stage; for clarity we display its unweighted block structure,
\begin{equation}
\label{eq:Jacobi}
\mathcal{J}=
\begin{pmatrix}
	\mathcal{J}_{\Omega^+} & \mathbf{0}_{N_{\Omega^+}\times m}
	\\
	\mathbf{0}_{N_{\Omega^-}\times m} & \mathcal{J}_{\Omega^-}
	\\
	\mathcal{J}_{\Gamma_d^+} &
	\mathcal{J}_{\Gamma_d^-}
	\\
	\mathcal{J}_{\Gamma_n^+} &
	\mathcal{J}_{\Gamma_n^-}
	\\ 
	\mathcal{J}_{\partial\Omega^+} & \mathbf{0}_{N^+_{\partial\Omega}\times m}
	\\
	\mathbf{0}_{N^-_{\partial\Omega}\times m} & \mathcal{J}_{\partial\Omega^-}
\end{pmatrix}
.
\end{equation}
The zero blocks indicate residuals that are independent of the coefficients from the other subdomain. Here $N_{\partial\Omega}=N^+_{\partial\Omega}+N^-_{\partial\Omega}$ splits the boundary points by the subdomain on whose outer boundary they lie. For brevity, we write
\begin{displaymath}
\beta_N^\pm=\beta^\pm(x,u_N^\pm), \quad 
\beta_{u,N}^\pm=\partial_u\beta^\pm(x,u_N^\pm), \quad
\beta_{uu,N}^\pm=\partial_{uu}\beta^\pm(x,u_N^\pm),
\end{displaymath}
and define $f_N^\pm$, $f_{u,N}^\pm$, and $f_{uu,N}^\pm$ analogously, all evaluated at $(x, u_N^\pm(x))$ unless otherwise stated. The same shorthand is reused in the perturbation--correction stage of Sec~\ref{subsec:pcm}. 

The submatrices of the Jacobian are evaluated pointwise at the collocation points. For $x_i\in\mathcal{S}_{\Omega^+}$ and activation function $\phi_j^+$, the interior block is
\begin{equation}
\begin{aligned}
\left(\mathcal{J}_{\Omega^+}\right)_{i,j}
&:=\frac{\partial\mathcal{R}_{\Omega^+}}{\partial\alpha_j^+}\Bigg|_{x=x_i}
\\
&=
\left[
\nabla\cdot\left(
\beta_N^+\nabla\phi_j^+
+
\beta_{u,N}^+\phi_j^+\nabla u_N^+
\right)
+
f_{u,N}^+\phi_j^+
\right]_{x=x_i}.
\end{aligned}
\end{equation}
The block $\mathcal{J}_{\Omega^-}$ is defined analogously on $\mathcal{S}_{\Omega^-}$. 

For $x_i\in\mathcal{S}_{\Gamma_d}$, the blocks associated with the solution jump are
\begin{equation}
\left(\mathcal{J}_{\Gamma_d^+}\right)_{i,j}
:=
\frac{\partial\mathcal{R}_{\Gamma_d}}{\partial \alpha^+_j}\Bigg|_{x=x_i}
=
\phi^+_j(x_i),
\end{equation}
and
\begin{equation}
\left(\mathcal{J}_{\Gamma_d^-}\right)_{i,j}
:=
\frac{\partial\mathcal{R}_{\Gamma_d}}{\partial \alpha^-_j}\Bigg|_{x=x_i}
=
-\phi^-_j(x_i).
\end{equation}
For $x_i\in\mathcal{S}_{\Gamma_n}$, the blocks associated with the flux jump are
\begin{equation}
\left(\mathcal{J}_{\Gamma_n^+}\right)_{i,j}
:=
\frac{\partial\mathcal{R}_{\Gamma_n}}{\partial \alpha^+_j}\Bigg|_{x=x_i}
=
\left[
\beta_{u,N}^+\phi_j^+\partial_{\mathbf{n}}u_N^+
+\beta_N^+\partial_{\mathbf{n}}\phi_j^+
\right]_{x=x_i},
\end{equation}
and
\begin{equation}
\left(\mathcal{J}_{\Gamma_n^-}\right)_{i,j}
:=
\frac{\partial\mathcal{R}_{\Gamma_n}}{\partial \alpha^-_j}\Bigg|_{x=x_i}
=
-\left[
\beta_{u,N}^-\phi_j^-\partial_{\mathbf{n}}u_N^-
+\beta_N^-\partial_{\mathbf{n}}\phi_j^-
\right]_{x=x_i}.
\end{equation}
Finally, for $x_i\in\mathcal{S}_{\partial\Omega}\cap\partial\Omega^\pm$, the boundary block is
\begin{equation}
\left(\mathcal{J}_{\partial\Omega^\pm}\right)_{i,j}
:=
\frac{\partial\mathcal{R}_{\partial\Omega}}{\partial \alpha^\pm_j}\Bigg|_{x=x_i}
=
\phi_j^\pm(x_i).
\end{equation}

The update $\delta\boldsymbol{\alpha}=(\delta\boldsymbol{\alpha}^+,\delta\boldsymbol{\alpha}^-)$ is computed from the normal equations
\begin{equation}
\label{eq:Netwon-Gauss}
{\mathcal{J}^{\{k\}}}^T\mathcal{J}^{\{k\}}\delta\boldsymbol{\alpha}=-{\mathcal{J}^{\{k\}}}^T\mathcal{F}^{\{k\}},
\end{equation}
where $k$ denotes the iteration index.
For numerical stability, we compute $\delta\boldsymbol{\alpha}$ using the singular value decomposition and a truncated Moore--Penrose pseudoinverse, 
a standard regularization strategy for ill-conditioned least-squares systems~\cite{bjorck2024numerical,hansen1987truncated}. Let the SVD of $\mathcal{J}^{\{k\}}\in\mathbb{R}^{N_{total}\times2m}$ be
\begin{equation}
\mathcal{J}^{\{k\}} = U\Sigma V^T,
\end{equation}
where $U\in\mathbb{R}^{N_{total}\times r}$ and $V\in\mathbb{R}^{2m\times r}$ have orthonormal columns, $\Sigma=\text{diag}\{\sigma_1,\sigma_2,\cdots,\sigma_r,0,\cdots\}$ contains the singular values $\sigma_1\geq\sigma_2\geq\cdots\geq\sigma_r>0$, and $r=\text{rank}(\mathcal{J}^{\{k\}})$.
We truncate all singular values below a prescribed threshold $\tau>0$ and retain the indices $\mathcal{I}_r=\{i:\sigma_i\geq \tau\}$. The truncated pseudoinverse of $\Sigma$ is
\begin{equation}
\left(\Sigma^\dagger_r  \right)_{ii}
=
\begin{cases}
	1/\sigma_i,&\sigma_i\geq\tau,
	\\
	0,&\sigma_i<\tau.
\end{cases}
\end{equation}
Accordingly, the regularized update is obtained as
\begin{equation}
\delta\boldsymbol{\alpha}=-V\Sigma_r^\dagger U^T\mathcal{F}^{\{k\}}.
\end{equation}
The output coefficients are then updated by
\begin{equation}
\label{eq:Netwon-Gauss-update}
\boldsymbol{\alpha}^{\{k+1\}}=\boldsymbol{\alpha}^{\{k\}}+\delta\boldsymbol{\alpha}.
\end{equation}
The iteration is terminated once the relative change in the residual norm satisfies 
\begin{equation}
\label{eq:Netwon_converge}
\frac{|\Vert\mathcal{F}(\boldsymbol{\alpha}^{\{k+1\}})\Vert_{\ell^2}-\Vert\mathcal{F}(\boldsymbol{\alpha}^{\{k\}})\Vert_{\ell^2}|}{\Vert\mathcal{F}(\boldsymbol{\alpha}^{\{k+1\}})\Vert_{\ell^2}}\leq\delta_0.
\end{equation}
The initial value $\boldsymbol{\alpha}_0$ is prescribed, for example $\boldsymbol{\alpha}_0=\boldsymbol{0}$. 

The Gauss--Newton system follows from the gradient identity $\nabla_{\alpha}G=\mathcal{J}^T\mathcal{F}$, with $\mathcal{J}^T\mathcal{J}$ used as the standard Gauss--Newton approximation of the Hessian. Thus, \eqref{eq:Netwon-Gauss} gives the usual least-squares Gauss--Newton update for the overdetermined collocation system. Since $G(\boldsymbol{\alpha})$ is generally nonconvex for the quasi-linear problem, the iteration may converge to a stationary point with nonzero residual. The resulting primary approximation is
\begin{equation}
\label{eq:corresponding solution}
u_N^\pm=\sum_{j=1}^m\alpha_{j}^\pm\phi^\pm(\mathbf{w}_j^\pm\cdot x+b_j^\pm).
\end{equation}

\subsection{Perturbation--Correction Stage}\label{subsec:pcm}
To reduce the residual left by the primary LRaNN stage, we introduce a perturbation-correction stage based on a local expansion of the residual operator. The primary approximation $u_N^\pm$ may retain nonzero residuals because the original least-squares problem is nonconvex. We seek a corrected approximation of the form
\begin{equation}\label{eq:perturbation expansion}
u_h^\pm = u_N^\pm + \epsilon\,u_p^\pm,\qquad \epsilon\,u_p^\pm \approx u_g^\pm - u_N^\pm=e_N^\pm,
\end{equation}
where $u_N^\pm$ is the primary approximation, $u_p^\pm$ is the scaled perturbation-correction function, and $u_h^\pm$ is the corrected approximation. The actual correction added to the primary approximation is
\[
\delta u^\pm:=\epsilon u_p^\pm .
\]
The primary error $e_N^\pm=u_g^\pm-u_N^\pm$ is the quantity introduced in \eqref{eq:perturbation expansion}; we likewise denote the corrected error by
\begin{equation}\label{eq:error-fields}
e_h^\pm:=u_g^\pm-u_h^\pm=e_N^\pm-\epsilon\,u_p^\pm .
\end{equation}
The correction is most effective when the actual correction $\epsilon u_p^\pm$ reproduces the primary error $e_N^\pm$, the quantity compared in the interface-trace figures of Sec~\ref{sec:4}. The parameter $\epsilon>0$ is used as a normalization factor in the correction equation. In the computations, we set $\epsilon=\|\mathcal{F}(u_N)\|_{\ell^2}$, but changing $\epsilon$ mainly rescales the correction coefficients; the quantity added to the primary approximation is always $\epsilon u_p^\pm$.

\par The perturbation-correction function is obtained from a residual expansion around $u_N^\pm$ and is represented in an independent randomized trial space with the same domain decomposition and jump structure as the original model.

\par More explicitly, we expand the residual in the actual correction $\delta u^\pm=\epsilon u_p^\pm$ and then write the resulting correction equation in terms of the scaled variable $u_p^\pm$. Retaining the first-order terms gives the linear interface problem
\begin{equation}
\label{eq:up-first-order}
\begin{cases}
\nabla\cdot\!\big(\beta_N^\pm\nabla u_p^\pm+\beta_{u,N}^\pm u_p^\pm\nabla u_N^\pm\big)+f_{u,N}^\pm u_p^\pm=-\epsilon^{-1}\mathcal R_{\Omega^\pm}, & x\in\Omega^\pm,\\[1mm]
\llbracket u_p\rrbracket=-\epsilon^{-1}\mathcal R_{\Gamma_d}, & x\in\Gamma,\\[1mm]
\llbracket \beta_N\partial_{\mathbf n}u_p+\beta_{u,N}u_p\partial_{\mathbf n}u_N\rrbracket=-\epsilon^{-1}\mathcal R_{\Gamma_n}, & x\in\Gamma,\\[1mm]
u_p^\pm=-\epsilon^{-1}\mathcal R_{\partial\Omega}, & x\in\partial\Omega,
\end{cases}
\end{equation}
whose interior and flux operators are the linearizations of $\beta(x,u)\nabla u$ about $u_N$, and whose right-hand sides are the primary residuals $\mathcal R_{\Omega^\pm}$, $\mathcal R_{\Gamma_d}$, $\mathcal R_{\Gamma_n}$, $\mathcal R_{\partial\Omega}$ of Sec~\ref{subsec:LS} evaluated at $u_N$. Retaining the second-order (two-term) expansion adds the quadratic terms in $u_p$, so that the interior and flux-jump conditions become
\begin{equation}
\label{eq:up-second-order-interior}
\begin{aligned}
\nabla\cdot\!\big(\beta_N^\pm\nabla u_p^\pm+\beta_{u,N}^\pm u_p^\pm\nabla u_N^\pm\big)+f_{u,N}^\pm u_p^\pm
&+\epsilon\Big[\nabla\cdot\!\big(\beta_{u,N}^\pm u_p^\pm\nabla u_p^\pm+\tfrac12\beta_{uu,N}^\pm (u_p^\pm)^2\nabla u_N^\pm\big)\\
&\quad+\tfrac12 f_{uu,N}^\pm (u_p^\pm)^2\Big]=-\epsilon^{-1}\mathcal R_{\Omega^\pm}, \qquad x\in\Omega^\pm,
\end{aligned}
\end{equation}
\begin{equation}
\label{eq:up-second-order-flux}
\Big\llbracket \beta_N\partial_{\mathbf n}u_p+\beta_{u,N}u_p\partial_{\mathbf n}u_N+\epsilon\big(\beta_{u,N}u_p\partial_{\mathbf n}u_p+\tfrac12\beta_{uu,N}u_p^2\partial_{\mathbf n}u_N\big)\Big\rrbracket=-\epsilon^{-1}\mathcal R_{\Gamma_n}, \qquad x\in\Gamma,
\end{equation}
while the solution-jump condition $\llbracket u_p\rrbracket=-\epsilon^{-1}\mathcal R_{\Gamma_d}$ and the boundary condition $u_p^\pm=-\epsilon^{-1}\mathcal R_{\partial\Omega}$ stay linear. Since \eqref{eq:up-second-order-interior}--\eqref{eq:up-second-order-flux} are quadratic in $u_p$, the two-term correction solves a nonlinear interface problem with the same domain decomposition, interface, and boundary structure as the original model. The perturbation residual $\mathcal F_p$ introduced next is the discrete least-squares residual of this interface problem.

\par Substituting \(u_N^\pm+\delta u^\pm\), with \(\delta u^\pm=\epsilon u_p^\pm\), into the nonlinear residual operator $\mathcal{F}(\cdot)$ and expanding the residual in the actual correction $\delta u$ gives
\begin{displaymath}
\mathcal{F}(u_N+\delta u)
=
\mathcal{F}(u_N)
+
D\mathcal{F}(u_N)[\delta u]
+
\tfrac{1}{2}D^2\mathcal{F}(u_N)[\delta u,\delta u]
+
R_3(\delta u),
\end{displaymath}
where the Taylor remainder is controlled by the size of the actual correction $\delta u$, rather than by $\epsilon$ alone. Writing $\delta u=\epsilon u_p$ gives
\begin{displaymath}
\mathcal{F}(u_N+\epsilon u_p)
=
\mathcal{F}(u_N)
+
\epsilon\,D\mathcal{F}(u_N)[u_p]
+
\tfrac{1}{2}\epsilon^2D^2\mathcal{F}(u_N)[u_p,u_p]
+
R_3(\epsilon u_p).
\end{displaymath}
Neglecting the remainder and normalizing by $\epsilon$ gives the truncated perturbation-correction residual used in the computation,
\begin{equation}
\label{eq:normalized_pc_residual}
\mathcal{F}_p(\boldsymbol{\gamma})
=
\epsilon^{-1}
\left(
\mathcal{F}(u_N)
+
\epsilon D\mathcal{F}(u_N)[u_p(\boldsymbol{\gamma})]
+
\tfrac{1}{2}\epsilon^2D^2\mathcal{F}(u_N)[u_p(\boldsymbol{\gamma}),u_p(\boldsymbol{\gamma})]
\right).
\end{equation}
The perturbation-correction coefficients are obtained by minimizing the normalized residual least-squares functional
\begin{equation}\label{eq:perturbation optimize}
\min_{\boldsymbol{\gamma}}G_p(\boldsymbol{\gamma})
=
\min_{\boldsymbol{\gamma}}
\frac12
\left\|
\mathcal{F}_p(\boldsymbol{\gamma})
\right\|_{\ell^2}^2 .
\end{equation}

\par We approximate $u_p^\pm$ with another local randomized network whose hidden-layer weights and biases are fixed. This network uses $m_p$ hidden neurons, not necessarily equal to $m$:
\begin{equation}
u_p^+=\sum_{j=1}^{m_p}\gamma_j^+\psi^+(\mathbf{w}_{pj}^+\cdot \mathbf{x}+b_{pj}^+), \quad
u_p^-=\sum_{j=1}^{m_p}\gamma_j^-\psi^-(\mathbf{w}_{pj}^-\cdot \mathbf{x}+b_{pj}^-),
\end{equation}
where $\psi^\pm$ are activation functions for the perturbation-correction network, and $\mathbf{w}_{pj}^\pm,b_{pj}^\pm$ are fixed random parameters. In the primary LRaNN stage we use $\tanh$ activations, while in the perturbation-correction stage we use $\sin$ activations to enrich the representation of residual error modes.

\par The discretized residual vector for the perturbation equation is written as
\begin{equation}
\label{eq:residual_p}
\mathcal{F}_p=
\frac 1\epsilon
\begin{pmatrix}
	\mathcal{R}_p^+
	\\
	\mathcal{R}_p^-
	\\
	\mathcal{R}_{\Gamma_d,p}
	\\
	\mathcal{R}_{\Gamma_n,p}
	\\
	\mathcal{R}_{\partial\Omega,p}
\end{pmatrix}
,
\end{equation}
evaluated at the collocation points $\mathcal{S}_{\Omega^\pm}$, $\mathcal{S}_{\Gamma_d}$, $\mathcal{S}_{\Gamma_n}$, and $\mathcal{S}_{\partial\Omega}$. As in the primary stage, the weights in the discrete least-squares functional are incorporated by row scaling.

We reuse the shorthand $\beta_N^\pm$, $\beta_{u,N}^\pm$, $\beta_{uu,N}^\pm$ and $f_N^\pm$, $f_{u,N}^\pm$, $f_{uu,N}^\pm$ introduced in Sec~\ref{subsec:Newton iterative}, all evaluated at the corresponding collocation point unless otherwise stated.

For each interior collocation point $x_i \in \mathcal S_{\Omega^\pm}$, the pointwise perturbation residual is
\begin{equation}
\begin{aligned}
\mathcal{R}_p^\pm
=&
\epsilon^0
\left[
\nabla\cdot\left(\beta_N^\pm\nabla u_N^\pm\right)
+
f_N^\pm
\right]
\\
&+
\epsilon^1
\left[
\nabla\cdot\left(
\beta_N^\pm\nabla u_p^\pm
+
\beta_{u,N}^\pm u_p^\pm\nabla u_N^\pm
\right)
+
f_{u,N}^\pm u_p^\pm
\right]
\\
&+
\epsilon^2
\left[
\nabla\cdot\left(
\beta_{u,N}^\pm u_p^\pm\nabla u_p^\pm
+
\frac12\beta_{uu,N}^\pm (u_p^\pm)^2\nabla u_N^\pm
\right)
+
\frac12 f_{uu,N}^\pm (u_p^\pm)^2
\right].
\end{aligned}
\end{equation}
Here $\epsilon^0$, $\epsilon^1$, and $\epsilon^2$ indicate the zeroth-, first-, and second-order contributions from the Taylor expansion.

For the solution jump condition at $x_i\in\mathcal S_{\Gamma_d}$, we define
\begin{equation}
\mathcal{R}_{\Gamma_d,p}
=
u_N^+-u_N^--w(x_i)
+
\epsilon\left(u_p^+-u_p^-\right).
\end{equation}
For the flux jump condition at $x_i\in\mathcal S_{\Gamma_n}$, the truncated perturbation residual is
\begin{equation}
\begin{aligned}
\mathcal{R}_{\Gamma_n,p}
=&
\epsilon^0
\left[
\beta_N^+\partial_{\mathbf{n}}u_N^+
-
\beta_N^-\partial_{\mathbf{n}}u_N^-
-
v(x_i)
\right]
\\
&+
\epsilon^1
\left[
\beta_N^+\partial_{\mathbf{n}}u_p^+
+
\beta_{u,N}^+u_p^+\partial_{\mathbf{n}}u_N^+
-
\beta_N^-\partial_{\mathbf{n}}u_p^-
-
\beta_{u,N}^-u_p^-\partial_{\mathbf{n}}u_N^-
\right]
\\
&+
\epsilon^2
\left[
\beta_{u,N}^+u_p^+\partial_{\mathbf{n}}u_p^+
+
\frac12\beta_{uu,N}^+(u_p^+)^2\partial_{\mathbf{n}}u_N^+
-
\beta_{u,N}^-u_p^-\partial_{\mathbf{n}}u_p^-
-
\frac12\beta_{uu,N}^-(u_p^-)^2\partial_{\mathbf{n}}u_N^-
\right].
\end{aligned}
\end{equation}
On the boundary $x_i\in\mathcal S_{\partial\Omega}\cap\partial\Omega^\pm$, we define
\begin{equation}
\mathcal{R}_{\partial\Omega,p}
=
u_N^\pm-g(x_i)+\epsilon u_p^\pm.
\end{equation}

The Gauss--Newton Jacobian is the block matrix
\begin{equation}
\mathcal{J}_p
=
\frac{\partial \mathcal{F}_p}{\partial\boldsymbol{\gamma}}
=
\frac 1\epsilon
\begin{pmatrix}
	\widehat{\mathcal{J}}_{\Omega^+,p}&0\\
	0&\widehat{\mathcal{J}}_{\Omega^-,p}\\
	\widehat{\mathcal{J}}_{\Gamma_d,p}^+&\widehat{\mathcal{J}}_{\Gamma_d,p}^-\\
	\widehat{\mathcal{J}}_{\Gamma_n,p}^+&\widehat{\mathcal{J}}_{\Gamma_n,p}^-\\
	\widehat{\mathcal{J}}_{\partial\Omega,p}^+&0\\
	0&\widehat{\mathcal{J}}_{\partial\Omega,p}^-
\end{pmatrix},
\end{equation}
where the hatted blocks denote derivatives of the unnormalized residual components. The actual Jacobian used in the weighted least-squares problem is obtained by applying the same row scaling as in the residual vector.

For $x_i\in\mathcal S_{\Omega^+}$, the $i,j$ entry of $\widehat{\mathcal J}_{\Omega^+,p}\in\mathbb R^{N_{\Omega^+}\times m_p}$ is
\begin{equation}
\begin{aligned}
(\widehat{\mathcal{J}}_{\Omega^+,p})_{i,j}
&=
\frac{\partial \mathcal R_p^+(x_i)}{\partial \gamma_j^+}
\\
&=
\epsilon
\left[
\nabla\cdot\left(
\beta_N^+\nabla\psi_j^+
+
\beta_{u,N}^+\psi_j^+\nabla u_N^+
\right)
+
f_{u,N}^+\psi_j^+
\right]_{x=x_i}
\\
&\quad+
\epsilon^2
\left[
\nabla\cdot\left(
\beta_{u,N}^+\psi_j^+\nabla u_p^+
+
\beta_{u,N}^+u_p^+\nabla\psi_j^+
+
\beta_{uu,N}^+u_p^+\psi_j^+\nabla u_N^+
\right)
+
f_{uu,N}^+u_p^+\psi_j^+
\right]_{x=x_i}.
\end{aligned}
\end{equation}
The block $\widehat{\mathcal J}_{\Omega^-,p}$ is obtained by replacing $(+)$ with $(-)$ and $\psi_j^+$ with $\psi_j^-$.

For $x_i\in\mathcal S_{\Gamma_d}$, the solution-jump blocks are
\begin{equation}
(\widehat{\mathcal{J}}_{\Gamma_d,p}^+)_{i,j}
=
\epsilon\,\psi_j^+(x_i),
\qquad
(\widehat{\mathcal{J}}_{\Gamma_d,p}^-)_{i,j}
=
-\epsilon\,\psi_j^-(x_i).
\end{equation}
For $x_i\in\mathcal S_{\Gamma_n}$, the flux-jump blocks are
\begin{equation}
\begin{aligned}
(\widehat{\mathcal{J}}_{\Gamma_n,p}^+)_{i,j}
=&
\epsilon
\left[
\beta_N^+\partial_{\mathbf n}\psi_j^+
+
\beta_{u,N}^+\psi_j^+\partial_{\mathbf n}u_N^+
\right]_{x=x_i}
\\
&+
\epsilon^2
\left[
\beta_{u,N}^+
\left(
\psi_j^+\partial_{\mathbf n}u_p^+
+
u_p^+\partial_{\mathbf n}\psi_j^+
\right)
+
\beta_{uu,N}^+u_p^+\psi_j^+\partial_{\mathbf n}u_N^+
\right]_{x=x_i},
\end{aligned}
\end{equation}
and
\begin{equation}
\begin{aligned}
(\widehat{\mathcal{J}}_{\Gamma_n,p}^-)_{i,j}
=&
-\epsilon
\left[
\beta_N^-\partial_{\mathbf n}\psi_j^-
+
\beta_{u,N}^-\psi_j^-\partial_{\mathbf n}u_N^-
\right]_{x=x_i}
\\
&-
\epsilon^2
\left[
\beta_{u,N}^-
\left(
\psi_j^-\partial_{\mathbf n}u_p^-
+
u_p^-\partial_{\mathbf n}\psi_j^-
\right)
+
\beta_{uu,N}^-u_p^-\psi_j^-\partial_{\mathbf n}u_N^-
\right]_{x=x_i}.
\end{aligned}
\end{equation}
For $x_i\in\mathcal S_{\partial\Omega}\cap\partial\Omega^\pm$,
\begin{equation}
(\widehat{\mathcal J}_{\partial\Omega,p}^\pm)_{i,j}
=
\frac{\partial \mathcal R_{\partial\Omega,p}(x_i)}{\partial \gamma_j^\pm}
=
\epsilon\,\psi_j^\pm(x_i).
\end{equation}

\par With the perturbation-correction residual vector $\mathcal{F}_p$ defined, the subproblem \eqref{eq:perturbation optimize} is solved by the same Gauss--Newton procedure used in the primary LRaNN stage. The complete two-stage framework is summarized below.
\begin{algorithm}[h]
\caption{LRaNN--PC Framework}\label{alg:two-stage}
\begin{algorithmic}[1]
	\STATE {Initialize the primary LRaNNs: $u_N^\pm=\sum_{j=1}^m\alpha_j^\pm\phi^\pm\left(\mathbf{w}_j^\pm \cdot \mathbf{x} + b_j^\pm\right)$ for $\mathbf{x}\subseteq \bar{\Omega}$.}
	\FOR {$k=1$ to $M$}
	\STATE {Evaluate the residual vector $\mathcal{F}(\boldsymbol{\alpha})$ and the Jacobian matrix $\mathcal{J}(\boldsymbol{\alpha})$ for \eqref{model of the problem}.}
	\STATE {Solve the normal equations and update the output coefficients $\boldsymbol{\alpha}^{\{k+1\}}=\boldsymbol{\alpha}^{\{k\}}+\delta\boldsymbol{\alpha}$.}
	\STATE {Stop if the residual criterion \eqref{eq:Netwon_converge} is satisfied.}
	\ENDFOR
	\STATE{Set the primary approximation $u_N^\pm$.}
	\STATE {Initialize the perturbation-correction networks: $u_p^\pm=\sum_{j=1}^{m_p}\gamma_j^\pm\psi^\pm\left(\mathbf{w}_{pj}^{\pm} \cdot \mathbf{x} + b_{pj}^{\pm}\right)$ for $\mathbf{x}\subseteq \bar{\Omega}$.}
	\FOR {$k=1$ to $M$}
	\STATE {Evaluate the perturbation residual $\mathcal{F}_p(\boldsymbol{\gamma})$ and Jacobian $\mathcal{J}_p(\boldsymbol{\gamma})$.}
	\STATE {Solve the perturbation-correction normal equations and update $\boldsymbol{\gamma}^{\{k+1\}}=\boldsymbol{\gamma}^{\{k\}}+\delta\boldsymbol{\gamma}$.}
	\STATE {Stop when the perturbation residual stagnates or reaches the prescribed tolerance.}
	\ENDFOR
	\STATE {Return the corrected approximation $u_h^\pm=u_N^\pm(\boldsymbol{\alpha})+\epsilon u_p^\pm(\boldsymbol{\gamma})$.}
\end{algorithmic}
\end{algorithm}
\noindent
This two-stage framework uses the primary LRaNN stage to capture the dominant solution structure and the perturbation-correction stage to reduce the residual components left by the primary approximation.
\begin{remark}
The residual vector and Jacobian can be assembled either from explicit analytic derivatives of the randomized basis functions or by automatic differentiation. In our setting, the fixed single-hidden-layer representation makes analytic differentiation straightforward and avoids the overhead associated with repeated automatic-differentiation passes.
\end{remark}

\section{Error Analysis}\label{sec:error analysis}

This section derives a residual-controlled estimate for smooth LRaNN approximations of the model problem considered in this paper under the stated sufficient assumptions. The analysis is carried out in a broken $H^1$ seminorm, which is natural for interface problems with prescribed jumps across $\Gamma$. The estimate links the error to the continuous least-squares residuals and then to their collocation-based discrete counterparts.

\subsection{Residual-controlled error estimate}
\label{subsec:generalization error}

Let $u_*=(u_*^+,u_*^-)$ be a smooth network approximation to the weak solution $u_g$ to which the residual estimate is applied; in particular, $u_*$ may be the primary approximation $u_N$ or the corrected approximation $u_h=u_N+\epsilon u_p$. We write $e_*^\pm:=u_g^\pm-u_*^\pm$ and measure the error in the broken $H^1$ seminorm
\begin{equation}
	\label{eq:broken-seminorm}
	\|\nabla_h e_*\|_{L^2(\Omega)}^2
	:=\|\nabla e_*^+\|_{L^2(\Omega^+)}^2+\|\nabla e_*^-\|_{L^2(\Omega^-)}^2 .
\end{equation}
For brevity we set $A^\pm(x,\xi):=\beta^\pm(x,\xi)\nabla \xi$. Using the residual operators introduced in
Sec~\ref{subsec:LS} with $u_N$ replaced by the generic approximation $u_*$,
\begin{equation}
	\label{eq:residuals-generic}
	\begin{aligned}
		&\mathcal R_{\Omega^\pm}(u_*)=\nabla\cdot A^\pm(x,u_*^\pm)+f^\pm(x,u_*^\pm),
		\qquad
		\mathcal R_{\partial\Omega}(u_*)=u_*^\pm-g,
		\\
		&\mathcal R_{\Gamma_d}(u_*)=u_*^+-u_*^--w,
		\qquad
		\mathcal R_{\Gamma_n}(u_*)=A^+(x,u_*^+)\cdot\mathbf n-A^-(x,u_*^-)\cdot\mathbf n-v,
	\end{aligned}
\end{equation}
we define the continuous residual indicator
\begin{equation}
	\label{eq:eta-cont}
	\eta(u_*):=
	\|\mathcal R_{\Omega^+}\|_{L^2(\Omega^+)}^2
	+\|\mathcal R_{\Omega^-}\|_{L^2(\Omega^-)}^2
	+\|\mathcal R_{\partial\Omega}\|_{L^2(\partial\Omega)}^2
	+\|\mathcal R_{\Gamma_d}\|_{L^2(\Gamma)}^2
	+\|\mathcal R_{\Gamma_n}\|_{L^2(\Gamma)}^2 .
\end{equation}
The terms in \eqref{eq:eta-cont} coincide, up to the positive weights $\omega_\bullet$, with those of
the least-squares functional \eqref{eq:LS-formulation}. Because these weights are positive constants, the
weighted and unweighted residual indicators are equivalent up to the ratio of the weights, so the
weights affect the constants in the estimates below but not their form.

Throughout this subsection we use the following standing assumptions. The diffusion operator is strongly monotone with constant $\lambda:=\min\{\lambda^+,\lambda^-\}>0$, and the source term
$f^\pm$ is one-sided monotone in $u$, as in \eqref{eq:strong-monotonicity} and \eqref{eq:one-sided-monotonicity}. The estimate is applied to smooth network approximations $u_*$ whose ranges, together with the range of $u_g$, are contained in the corresponding monotonicity intervals $\mathcal I^\pm$.
For piecewise $H^1$ functions $z=(z^+,z^-)$ the broken Friedrichs inequality
\begin{equation}
	\label{eq:broken-friedrichs-cont}
	\|z\|_{L^2(\Omega)}
	\le C_F\big(\|\nabla_h z\|_{L^2(\Omega)}+\|z\|_{L^2(\partial\Omega)}+\|\llbracket z\rrbracket\|_{L^2(\Gamma)}\big)
\end{equation}
holds~\cite{brenner2003poincare}, with $C_F$ depending only on $\Omega^\pm$, $\Gamma$, and the Dirichlet boundary. The exact solution $u_g$ is the weak solution of Assumption~\ref{ass:weak-existence}, a piecewise $H^1$ field, and the approximation $u_*$ considered in the estimate is a smooth network. The flux regularity required in Assumption~\ref{ass:weak-existence} is consistent with standard subdomain elliptic regularity away from corners~\cite{li2000gradient,li2003estimates,dong2019gradient}. The normal flux differences
\begin{equation}
	\label{eq:flux-diff}
	B_{\partial\Omega}:=\big(A(u_g)-A(u_*)\big)\cdot\mathbf n \ \text{ on }\partial\Omega,
	\qquad
	B_\Gamma^\pm:=\big(A^\pm(u_g)-A^\pm(u_*)\big)\cdot\mathbf n \ \text{ on }\Gamma,
\end{equation}
are then well defined, and we assume them uniformly bounded in $L^2$, i.e.\ there exists $M>0$ such that
$\|B_{\partial\Omega}\|_{L^2(\partial\Omega)}+\|B_\Gamma^+\|_{L^2(\Gamma)}+\|B_\Gamma^-\|_{L^2(\Gamma)}\le M$.

\begin{theorem}
	\label{thm:residual-bound}
	Under the standing assumptions above, there exists a constant $C>0$, depending only on
$\Omega^\pm$, $\Gamma$, $\lambda$, $C_F$, and $M$, such that any such smooth network approximation $u_*$
satisfies
	\begin{equation}
		\label{eq:continuous-error-bound}
		\|\nabla_h(u_g-u_*)\|_{L^2(\Omega)}^2
		\le C\big(\eta(u_*)+\eta(u_*)^{1/2}\big).
	\end{equation}
\end{theorem}

\begin{proof}
	Since $u_g$ is the weak solution of Assumption~\ref{ass:weak-existence}, its 
	flux satisfies $A^\pm(x,u_g^\pm)\in H(\mathrm{div},\Omega^\pm)$ with 
	$-\nabla\cdot A^\pm(x,u_g^\pm)=f^\pm(x,u_g^\pm)$ in $L^2(\Omega^\pm)$. Subtracting 
	the corresponding identity for the smooth approximation $u_*^\pm$ gives, in 
	$L^2(\Omega^\pm)$,
	\begin{equation}
		\label{eq:error-equation}
		-\nabla\cdot\big(A^\pm(x,u_g^\pm)-A^\pm(x,u_*^\pm)\big)
		=f^\pm(x,u_g^\pm)-f^\pm(x,u_*^\pm)+\mathcal R_{\Omega^\pm}(u_*).
	\end{equation}
	Testing \eqref{eq:error-equation} with $e_*^\pm$ and applying Green's formula in 
	$H(\mathrm{div},\Omega^\pm)$~\cite{girault2012finite}, which the $L^2$ normal flux traces of the standing 
	assumptions license, then summing over $\pm$ with the convention 
	$\mathbf n^+=\mathbf n$ and $\mathbf n^-=-\mathbf n$ on $\Gamma$, yields
	\begin{equation}
		\label{eq:energy-identity}
		\sum_\pm\int_{\Omega^\pm}\!\big(A^\pm(u_g)-A^\pm(u_*)\big)\cdot\nabla e_*^\pm\,dx
		=\sum_\pm\int_{\Omega^\pm}\!\big(f^\pm(u_g)-f^\pm(u_*)\big)e_*^\pm\,dx
		+I_\Omega+I_{\partial\Omega}+I_\Gamma,
	\end{equation}
	where
	\begin{equation}
		\label{eq:I-terms}
		I_\Omega:=\sum_\pm\int_{\Omega^\pm}\!\mathcal R_{\Omega^\pm}(u_*)\,e_*^\pm\,dx,
		\quad
		I_{\partial\Omega}:=\int_{\partial\Omega}\!B_{\partial\Omega}\,e_*\,ds,
		\quad
		I_\Gamma:=\int_{\Gamma}\!\big(B_\Gamma^+e_*^+-B_\Gamma^-e_*^-\big)\,ds .
	\end{equation}
	By the one-sided monotonicity of $f^\pm$, the first term on the right-hand side of
	\eqref{eq:energy-identity} is non-positive, while the strong monotonicity of the diffusion operator
	bounds the left-hand side from below by $\lambda\|\nabla_h e_*\|_{L^2(\Omega)}^2$. Hence
	\begin{equation}
		\label{eq:coercivity}
		\lambda\|\nabla_h e_*\|_{L^2(\Omega)}^2\le I_\Omega+I_{\partial\Omega}+I_\Gamma .
	\end{equation}
	Because $u_g$ matches the boundary and interface data exactly, the error traces satisfy 
	\begin{equation}
		\label{eq:trace-identities}
		e_*|_{\partial\Omega}=-\mathcal R_{\partial\Omega}(u_*),
		\qquad
		\llbracket e_*\rrbracket=e_*^+-e_*^-=-\mathcal R_{\Gamma_d}(u_*),
		\qquad
		B_\Gamma^+-B_\Gamma^-=-\mathcal R_{\Gamma_n}(u_*).
	\end{equation}
	We bound the three terms in \eqref{eq:coercivity} separately; here $\rho>0$ is a free parameter to be
	fixed below, and $C_\rho$ denotes a generic constant depending on $\rho$, $C_F$, and the data.

	\emph{Interior term.} By the Cauchy--Schwarz inequality, the broken Friedrichs inequality
	\eqref{eq:broken-friedrichs-cont}, and the trace identities \eqref{eq:trace-identities},
	\begin{equation}
		\label{eq:I-Omega-bound}
		|I_\Omega|
		\le\big(\|\mathcal R_{\Omega^+}\|+\|\mathcal R_{\Omega^-}\|\big)
		C_F\big(\|\nabla_h e_*\|+\|\mathcal R_{\partial\Omega}\|+\|\mathcal R_{\Gamma_d}\|\big)
		\le\rho\|\nabla_h e_*\|_{L^2(\Omega)}^2+C_\rho\,\eta(u_*).
	\end{equation}
	All residuals enter quadratically, so $I_\Omega$ contributes only to $\eta(u_*)$.

	\emph{Boundary term.} Using $e_*|_{\partial\Omega}=-\mathcal R_{\partial\Omega}$ and the uniform flux
	bound $\|B_{\partial\Omega}\|\le M$,
	\begin{equation}
		\label{eq:I-bdry-bound}
		|I_{\partial\Omega}|
		=\Big|\int_{\partial\Omega}B_{\partial\Omega}\,\mathcal R_{\partial\Omega}\,ds\Big|
		\le M\,\|\mathcal R_{\partial\Omega}\|_{L^2(\partial\Omega)}
		\le M\,\eta(u_*)^{1/2}.
	\end{equation}
	Since $B_{\partial\Omega}$ is a trace of the flux error that is controlled only through the uniform
	bound $M$ and not by the broken $H^1$ seminorm $\|\nabla_h e_*\|$, this pairing is linear in
	the boundary residual and cannot be absorbed into $\|\nabla_h e_*\|^2$; it produces the square-root
	contribution $\eta^{1/2}$.

	\emph{Interface term.} Adding and subtracting and using \eqref{eq:trace-identities},
	\begin{equation}
		\label{eq:I-gamma-split}
		I_\Gamma
		=\int_\Gamma B_\Gamma^+(e_*^+-e_*^-)\,ds+\int_\Gamma(B_\Gamma^+-B_\Gamma^-)e_*^-\,ds
		=-\int_\Gamma B_\Gamma^+\,\mathcal R_{\Gamma_d}\,ds-\int_\Gamma\mathcal R_{\Gamma_n}\,e_*^-\,ds .
	\end{equation}
	The first integral is estimated as in \eqref{eq:I-bdry-bound}, giving
	$M\|\mathcal R_{\Gamma_d}\|_{L^2(\Gamma)}\le M\,\eta(u_*)^{1/2}$. For the second integral, the trace
	inequality combined with \eqref{eq:broken-friedrichs-cont} yields
	$\|e_*^-\|_{L^2(\Gamma)}\le C\big(\|\nabla_h e_*\|+\|\mathcal R_{\partial\Omega}\|+\|\mathcal R_{\Gamma_d}\|\big)$,
	so that, by Young's inequality,
	\begin{equation}
		\label{eq:I-gamma-bound}
		|I_\Gamma|
		\le M\,\eta(u_*)^{1/2}
		+\rho\|\nabla_h e_*\|_{L^2(\Omega)}^2+C_\rho\,\eta(u_*).
	\end{equation}

	Substituting \eqref{eq:I-Omega-bound}, \eqref{eq:I-bdry-bound}, and \eqref{eq:I-gamma-bound} into \eqref{eq:coercivity} and
	choosing $\rho=\lambda/4$ to absorb the gradient terms into the left-hand side gives
	\begin{equation}
		\label{eq:coercivity-final}
		\frac{\lambda}{2}\|\nabla_h e_*\|_{L^2(\Omega)}^2
		\le C_\rho\,\eta(u_*)+2M\,\eta(u_*)^{1/2},
	\end{equation}
	which is \eqref{eq:continuous-error-bound}.
\end{proof}

It remains to connect \eqref{eq:continuous-error-bound} with the discrete loss that is actually
minimized. For a residual component $r_D$ on $D$, the collocation sets in
\eqref{eq:collocation points} define the discrete mean-square norm
\begin{equation}
	\label{eq:discrete-norm}
	\|r_D\|_{N,D}^2:=\frac{1}{N_D}\sum_{i=1}^{N_D}|r_D(x_i^D)|^2 ,
\end{equation}
the empirical average of $|r_D|^2$ over the $N_D$ collocation points. With $|D|$ the measure of $D$,
this average estimates the averaged continuous norm $|D|^{-1}\|r_D\|_{L^2(D)}^2$, and we assume the
quadrature consistency~\cite{mishra2023estimates}
\begin{equation}
	\label{eq:quadrature-consistency}
	\Big|\,|D|^{-1}\|r_D\|_{L^2(D)}^2-\|r_D\|_{N,D}^2\,\Big|\le C_{\mathrm{quad}}^D\,N_D^{-\theta_D},
	\qquad D\in\{\Omega^+,\Omega^-,\partial\Omega,\Gamma_d,\Gamma_n\}.
\end{equation}
Equivalently $\|r_D\|_{L^2(D)}^2\le|D|\big(\|r_D\|_{N,D}^2+C_{\mathrm{quad}}^D N_D^{-\theta_D}\big)$, so each
continuous residual norm is controlled by its discrete counterpart up to a quadrature error and the
fixed measure $|D|$. Let $\eta_N(u_*)$ be the discrete residual indicator obtained from
\eqref{eq:eta-cont} by replacing each $\|r_D\|_{L^2(D)}^2$ with $\|r_D\|_{N,D}^2$, and let
\begin{equation}
	\label{eq:quadrature-error}
	\mathcal E_N:=\sum_{D}C_{\mathrm{quad}}^D\,N_D^{-\theta_D}
\end{equation}
denote the accumulated quadrature error. Absorbing the fixed measures $|D|$ into the generic constant,
$\eta(u_*)\le C\big(\eta_N(u_*)+\mathcal E_N\big)$, and \eqref{eq:continuous-error-bound} yields
\begin{equation}
	\label{eq:final-residual-error-bound}
	\|\nabla_h(u_g-u_*)\|_{L^2(\Omega)}^2
	\le C\Big(\eta_N(u_*)+\mathcal E_N+\big(\eta_N(u_*)+\mathcal E_N\big)^{1/2}\Big).
\end{equation}

\subsection{Residual correction and the corrected residual bound}
\label{subsec:correction-bound}

Recall that the perturbation-correction functions $u_p^\pm$ are represented by the local randomized neural network ansatz introduced in Sec~\ref{subsec:pcm}. Once the hidden weights and biases are fixed, the correction depends linearly on the output coefficients. Hence, for the combined coefficient vector $\boldsymbol{\gamma}$, the scaled correction can be written as
\begin{equation}
	u_p(\boldsymbol{\gamma})=\Psi\boldsymbol{\gamma},
\end{equation}
where $\Psi$ denotes the fixed correction basis assembled from the $\psi^\pm$ activations over the two subdomains. The actual correction added to the primary approximation is
\begin{equation}
	\delta u(\boldsymbol{\gamma}) := \epsilon u_p(\boldsymbol{\gamma}) = \epsilon\Psi\boldsymbol{\gamma}.
\end{equation}
Thus $\epsilon$ is used as a normalization factor in the correction equation, while the quantity entering the corrected approximation is $\delta u$.

Let $\mathcal F(u_N)$ be the weighted discrete residual vector of the primary LRaNN approximation, defined by the same residual components and collocation sets as in \eqref{eq:discrete form}. The corresponding discrete residual indicator is
\begin{equation}
	\eta_N(u_N):=\|\mathcal F(u_N)\|_{\ell^2}^2 .
\end{equation}
Equivalently, $\eta_N(u_N)=2G(\boldsymbol{\alpha})$ is twice the weighted discrete least-squares loss in \eqref{eq:discrete form}. This weighted indicator and the unweighted indicator $\eta_N$ of Sec~\ref{subsec:generalization error} coincide up to the ratio of the extreme weights; we use the symbol $\eta_N$ for both and absorb the weight ratio into the generic constants. In the computations, the normalization parameter is chosen according to the primary residual scale,
\begin{equation}
	\epsilon=\|\mathcal F(u_N)\|_{\ell^2}
	=
	\eta_N(u_N)^{1/2}.
\end{equation}

Let $J_p$ denote the discrete Fr\'echet derivative of the residual vector at $u_N$, restricted to the correction space:
\begin{equation}
	J_p\boldsymbol{\gamma}
	=
	D\mathcal F(u_N)[\Psi\boldsymbol{\gamma}] .
\end{equation}
Keeping only the first-order term in the residual expansion gives the normalized perturbation residual
\begin{equation}
	\mathcal F_p^{(1)}(\boldsymbol{\gamma})
	=
	\epsilon^{-1}\mathcal F(u_N)
	+
	J_p\boldsymbol{\gamma}.
\end{equation}
The associated perturbation-correction objective is
\begin{equation}
	G_p^{(1)}(\boldsymbol{\gamma})
	=
	\frac12
	\|\mathcal F_p^{(1)}(\boldsymbol{\gamma})\|_{\ell^2}^2 .
\end{equation}
Since $\mathcal F_p^{(1)}$ is affine in $\boldsymbol{\gamma}$, the first-order perturbation-correction subproblem is a convex least-squares problem. Its Hessian is
\begin{equation}
	\nabla_{\boldsymbol{\gamma}}^2G_p^{(1)}
	=
	J_p^TJ_p
	\succeq 0 .
\end{equation}
If $J_p$ has full column rank, then $J_p^TJ_p$ is positive definite and the first-order subproblem is strictly convex. This rank condition is a discrete nondegeneracy condition for the correction basis under the chosen collocation set.

Including the second variation of the residual gives the second-order normalized perturbation residual
\begin{equation}
	\mathcal F_p^{(2)}(\boldsymbol{\gamma})
	=
	\epsilon^{-1}\mathcal F(u_N)
	+
	J_p\boldsymbol{\gamma}
	+
	\epsilon Q_p(\boldsymbol{\gamma},\boldsymbol{\gamma}),
\end{equation}
where $Q_p$ denotes the second variation of the discrete residual operator restricted to the correction space. The corresponding objective is
\begin{equation}
	G_p^{(2)}(\boldsymbol{\gamma})
	=
	\frac12
	\|\mathcal F_p^{(2)}(\boldsymbol{\gamma})\|_{\ell^2}^2 .
\end{equation}
Unlike the first-order subproblem, the second-order correction problem is generally nonlinear and nonconvex with respect to the correction coefficients. We therefore do not use convexity for the second-order model. The Gauss--Newton method is used as a local nonlinear least-squares solver; the rapid convergence observed in the experiments should be interpreted as a numerical property of the tested correction problems, not as a global convergence guarantee.

Let $\boldsymbol{\gamma}_*$ be the coefficient vector obtained by solving the perturbation-correction subproblem, and define
\begin{equation}
	u_h=u_N+\epsilon u_p(\boldsymbol{\gamma}_*).
\end{equation}
Consistent with the scaled perturbation residual vector in \eqref{eq:residual_p}, define the normalized perturbation residual indicator
\begin{equation}
	\eta_p(\boldsymbol{\gamma}_*)
	:=
	\|\mathcal F_p^{(k)}(\boldsymbol{\gamma}_*)\|_{\ell^2}^2,
	\qquad k=1 \text{ or } 2,
\end{equation}
depending on whether the first- or second-order perturbation model is used. The corresponding unscaled truncated residual model at the corrected approximation satisfies
\begin{equation}
	\mathcal F^{(k)}(u_h)
	=
	\epsilon \mathcal F_p^{(k)}(\boldsymbol{\gamma}_*).
\end{equation}
Therefore, the discrete residual indicator associated with the truncated perturbation model is
\begin{equation}
	\eta_N^{(k)}(u_h)
	=
	\|\mathcal F^{(k)}(u_h)\|_{\ell^2}^2
	=
	\epsilon^2
	\|\mathcal F_p^{(k)}(\boldsymbol{\gamma}_*)\|_{\ell^2}^2 .
\end{equation}
Using $\epsilon^2=\eta_N(u_N)$, this gives
\begin{equation}
	\label{eq:pc-product-relation}
	\eta_N^{(k)}(u_h)
	=
	\eta_N(u_N)\eta_p(\boldsymbol{\gamma}_*).
\end{equation}

The relation \eqref{eq:pc-product-relation} involves the truncated model residual $\mathcal F^{(k)}(u_h)$, whereas Theorem~\ref{thm:residual-bound} controls the error through the exact residual $\mathcal F(u_h)$. Let
\begin{equation}
	R^{(k)}:=\mathcal F(u_h)-\mathcal F^{(k)}(u_h),
	\qquad
	\tau_k:=\|R^{(k)}\|_{\ell^2}.
\end{equation}
The Taylor remainder is governed by the size of the actual correction 
$\delta u=\epsilon u_p(\boldsymbol{\gamma}_*)$, not by $\epsilon$ alone. Under the smoothness assumptions on $\beta^\pm$ and $f^\pm$, Taylor's theorem gives
\begin{equation}
	\tau_1=O(\|\delta u\|^2)
	\quad\text{for the first-order model}, 
	\qquad
	\tau_2=O(\|\delta u\|^3)
	\quad\text{for the second-order model}.
\end{equation}
No boundedness of $u_p$ independent of $\epsilon$ is required in this statement.

By the triangle inequality,
\begin{equation}
	\label{eq:pc-exact-indicator}
	\eta_N(u_h)^{1/2}
	=
	\|\mathcal F^{(k)}(u_h)+R^{(k)}\|_{\ell^2}
	\le
	\epsilon\,\eta_p(\boldsymbol{\gamma}_*)^{1/2}
	+
	\tau_k .
\end{equation}
Applying Theorem~\ref{thm:residual-bound} to the corrected approximation and using
\eqref{eq:pc-exact-indicator} yields, after absorbing harmless constants,
\begin{equation}
	\label{eq:pc-final-bound}
	\|\nabla_h(u_g-u_h)\|_{L^2(\Omega)}^2
	\le
	C\Big(
	\epsilon^2\,\eta_p(\boldsymbol{\gamma}_*)
	+\tau_k^2
	+\mathcal E_N
	+
	\big(\epsilon^2\,\eta_p(\boldsymbol{\gamma}_*)+\tau_k^2+\mathcal E_N\big)^{1/2}
	\Big).
\end{equation}

Thus the perturbation-correction subproblem reduces the computable truncated-model residual contribution $\epsilon^2\eta_p(\boldsymbol{\gamma}_*)$. The exact residual-controlled bound also contains the quadrature error $\mathcal E_N$ and the truncation contribution $\tau_k$. Therefore, reducing the PC-stage residual improves the computable part of the bound when the quadrature and truncation terms are not dominant. This is the residual-reduction mechanism used to interpret the numerical results below; it is not a global convergence statement for the nonlinear least-squares algorithm.
\begin{remark}[Approximation, optimization, and quadrature contributions]\label{rmk:error-sources}
Estimate~\eqref{eq:final-residual-error-bound} controls the broken $H^1$ error through the discrete residual $\eta_N(u_*)$ and the quadrature error $\mathcal E_N$. Writing $\eta_N^\star$ for the smallest discrete residual attainable on the randomized trial space, we have $\eta_N(u_*)=\eta_N^\star+(\eta_N(u_*)-\eta_N^\star)$. The error therefore contains an approximation part $\eta_N^\star$, set by the trial space; an optimization part $\eta_N(u_*)-\eta_N^\star$, the residual stagnation left by the solver; and a quadrature part $\mathcal E_N$~\cite{mishra2022estimates,mishra2023estimates}. The LRaNN-PC method is designed to reduce the optimization-related residual left by the primary stage. This reduction is observed in the numerical experiments below, especially when the approximation error of the randomized trial space is small. For solutions that are smooth within each subdomain and well represented by the randomized trial space, $\eta_N^\star$ is small and the optimization part may dominate; this is the regime targeted here. For solutions with strong local singularities, $\eta_N^\star$ may be large, and reducing the optimization part cannot lower the error below the approximation floor; such problems require trial-space enrichment and lie outside the present scope.
\end{remark}

\section{Numerical examples}
\label{sec:4}
In this section, we evaluate the performance of the proposed LRaNN-PC framework through a series of quasi-linear interface problems.
In all experiments, LRaNN denotes the primary LRaNN stage: the primary approximation $u_N$ is obtained by minimizing the least-squares functional on a fixed randomized basis via the Gauss--Newton method introduced in Sec~\ref{subsec:Newton iterative}. LRaNN-PC denotes the perturbation-correction framework: after the primary LRaNN stage, the PC stage computes $u_p$ as introduced in Sec~\ref{subsec:pcm}, and returns the corrected approximation $u_h=u_N+\epsilon u_p$. Consequently, the comparison between LRaNN and LRaNN-PC quantifies the specific gain offered by the perturbation-correction stage over the standard randomized neural network.
The selected tests cover steady and unsteady cases, multiple and curved interfaces, different types of nonlinear diffusion, as well as scenarios with high contrast. The gradient-dependent and moving-interface tests are numerical extensions beyond the solution-dependent elliptic model used in the residual-controlled analysis.

To quantify the solution accuracy, let $u_g$ denote the exact or reference solution and let $U$ denote the numerical approximation being evaluated, that is, the primary approximation $u_N$ or the corrected approximation $u_h=u_N+\epsilon u_p$. The pointwise error fields shown in the figures below are the primary error $e_N=u_g-u_N$ and the corrected error $e_h=u_g-u_h$ introduced in Sec~\ref{subsec:pcm}. We report the primary discrete relative errors computed on test sets that are independent of the training collocation sets. 
On a uniformly distributed test set $\{x_i\}_{i=1}^N$, the discrete relative $L^2$ and $L^\infty$ errors are defined as
\begin{equation}
\label{eq:relative_l2_linf_errors}
	\erro_{L^2}(U)=\frac{\sqrt{\sum_{i=1}^N\left(u_g(x_i)-U(x_i)\right)^2}}{\sqrt{\sum_{i=1}^Nu_g(x_i)^2}}, \quad 
	\erro_{L^\infty}(U)=\frac{\max_{x_i}|u_g(x_i)-U(x_i)|}{\max_{x_i}|u_g(x_i)|}.
\end{equation}
For comparisons involving derivatives, we use the discrete relative broken $H^1$-seminorm error
\begin{equation}
\label{eq:relative_h1_error}
\erro_{H^1}(U)
=
\frac{
\left(
\sum_{i=1}^N |(u_g-U)_x(x_i)|^2+\sum_{i=1}^N |(u_g-U)_y(x_i)|^2
\right)^{1/2}
}{
\left(
\sum_{i=1}^N |(u_g)_x(x_i)|^2+\sum_{i=1}^N |(u_g)_y(x_i)|^2
\right)^{1/2}
}.
\end{equation}
When discrete absolute errors are reported, we use
\begin{equation}
\label{eq:absolute_errors}
E_{L^2}(U)=\sqrt{\sum_{i=1}^N\left(u_g(x_i)-U(x_i)\right)^2}, \quad
E_{L^\infty}(U)=\max_{x_i}|u_g(x_i)-U(x_i)|,
\end{equation}
\begin{equation}
\label{eq:absolute_h1_error}
E_{H^1}(U)=
\left(
\sum_{i=1}^N |(u_g-U)_x(x_i)|^2+\sum_{i=1}^N |(u_g-U)_y(x_i)|^2
\right)^{1/2}.
\end{equation}
The same notation is used on subdomain or interface sample sets when local errors are reported.

\subsection{Example 1}

We first consider a two-subdomain quasi-linear elliptic interface problem with homogeneous jump conditions, $w=v=0$. The computational domain is
\[
\Omega=[-1,1]\times[0,1],
\]
and the interface
\[
\Gamma=\{(x,y)\mid x=0\}
\]
splits $\Omega$ into the left subdomain $\Omega^-=[-1,0]\times[0,1]$ and the right subdomain $\Omega^+=[0,1]\times[0,1]$. The nonlinear diffusion coefficient is defined by
\[
\beta(x,u)=
\begin{cases}
1+u, & x\in\Omega^-,\\[1mm]
1+\dfrac12 u^2, & x\in\Omega^+,
\end{cases}
\]
and the source terms, boundary data, and interface data are generated from the exact solution
\[
u(x,y)=
\begin{cases}
x^2y^2, & x\in\Omega^-,\\[1mm]
e^{-x}x^2y^2, & x\in\Omega^+.
\end{cases}
\]
This example is used as a baseline two-subdomain test for the proposed perturbation-correction procedure.

In the primary LRaNN stage, each subdomain is approximated by an LRaNN $u_N^\pm$ with $m=100$ hidden neurons and $\tanh$ activation. The hidden weights and biases are independently sampled from $\mathcal N(0,1)$ and $\mathcal N(0,0.1^2)$, respectively. The Gauss--Newton iteration reaches a residual plateau after about five iterations (Fig~\ref{fig:stage1_comparison}), and the primary approximation stagnates with $\erro_{L^2}(u_N)$ at the $10^{-7}$ level. 

The perturbation-correction stage is then applied using an independent sinusoidal correction basis. Weights and biases of the correction-network $u_p^\pm$ are independently sampled from $\mathcal N(0,(5\pi)^2)$ and $\mathcal N(0,1)$, respectively. The perturbation parameter is set to the primary residual scale, $\epsilon=\|\mathcal F(u_N)\|_{\ell^2}$, unless otherwise stated. Table~\ref{tab:perturbation_basis_effect} reports the effect of the number of correction basis functions $m_p$, together with the effective rank $r$ of the correction Jacobian $\mathcal J_p$. For $m_p\le400$, the effective rank equals $2m_p$; over this range $\erro_{L^2}(u_h)$ decreases from $1.8983\times10^{-7}$ at $m_p=100$ to $2.4364\times10^{-10}$ at $m_p=400$, and $\erro_{L^\infty}(u_h)$ from $1.7860\times10^{-7}$ to $5.3896\times10^{-10}$. For $m_p\ge500$, the effective rank is smaller than $2m_p$ (for example, $968<1000$ at $m_p=500$ and $1773<4000$ at $m_p=2000$), and the error indicators level off: $\erro_{L^2}(u_h)$ decreases to $4.0712\times10^{-12}$ at $m_p=500$ and to $2.9054\times10^{-12}$ at $m_p=600$, then remains at about $2.9\times10^{-12}$ for $m_p\ge600$, while $\erro_{L^\infty}(u_h)$ remains at about $3.6\times10^{-12}$.

\begin{table}[h]
	\centering
	\small
	\setlength{\tabcolsep}{4.5pt}
	\caption{Effect of the correction basis size in Example~1. The primary LRaNN stage uses $m=100$ $\tanh$ neurons. The perturbation-correction function $u_p$ is constructed from $m_p$ $\sin$ basis functions using $35{,}752$ collocation points. In this basis-size study the perturbation parameter is fixed at $\epsilon=10^{-4}$; Table~\ref{tab:epsilon_robustness} shows that the corrected error is insensitive to this choice. Here $r$ is the effective rank of the correction Jacobian $\mathcal J_p$.}
	\label{tab:perturbation_basis_effect}
	\begin{tabular}{cccccc}
		\toprule
		$m_p$ & $\erro_{L^2}(u_h)$ & $\erro_{L^\infty}(u_h)$ & $r$ & Iter. & Residual norm \\
		\midrule
		100  & $1.8983\times 10^{-7}$  & $1.7860\times 10^{-7}$  & 200  & 3 & $1.1855\times 10^{-1}$ \\
		200  & $1.3988\times 10^{-7}$  & $2.2567\times 10^{-7}$  & 400  & 3 & $6.7436\times 10^{-2}$ \\
		400  & $2.4364\times 10^{-10}$ & $5.3896\times 10^{-10}$ & 800  & 3 & $1.9518\times 10^{-4}$ \\
		500  & $4.0712\times 10^{-12}$ & $6.5351\times 10^{-12}$ & 968  & 3 & $2.4224\times 10^{-6}$ \\
		600  & $2.9054\times 10^{-12}$ & $3.5908\times 10^{-12}$ & 1091 & 8 & $1.3174\times 10^{-7}$ \\
		800  & $2.9049\times 10^{-12}$ & $3.5960\times 10^{-12}$ & 1277 & 4 & $1.2992\times 10^{-7}$ \\
		1000 & $2.9051\times 10^{-12}$ & $3.6098\times 10^{-12}$ & 1402 & 4 & $1.2961\times 10^{-7}$ \\
		2000 & $2.9053\times 10^{-12}$ & $3.6441\times 10^{-12}$ & 1773 & 4 & $1.2909\times 10^{-7}$ \\
		\bottomrule
	\end{tabular}
\end{table}

Table~\ref{tab:epsilon_robustness} reports the corrected errors at $m_p=500$ as the normalization parameter $\epsilon$ varies from $10^{-6}$ to $10^{-1}$, including the default $\epsilon=\|\mathcal F(u_N)\|$. The relative errors remain at $\erro_{L^2}(u_h)\approx4.07\times10^{-12}$ and $\erro_{L^\infty}(u_h)\approx6.53\times10^{-12}$ over this range. The residual norm reported in the table is the normalized perturbation residual and therefore changes with this scaling.

\begin{table}[h]
	\centering
	\caption{Effect of the perturbation parameter $\epsilon$ in Example~1, at fixed $m_p=500$. The default choice is $\epsilon=\|\mathcal F(u_N)\|\approx2.4622\times10^{-3}$. The effective rank ($r=968$) and the iteration count ($3$) are identical for all rows.}
	\label{tab:epsilon_robustness}
	\begin{tabular}{cccc}
		\toprule
		$\epsilon$ & $\erro_{L^2}(u_h)$ & $\erro_{L^\infty}(u_h)$ & Residual norm \\
		\midrule
		$10^{-1}$              & $4.0711\times 10^{-12}$ & $6.5350\times 10^{-12}$ & $2.4224\times 10^{-9}$ \\
		$10^{-2}$              & $4.0711\times 10^{-12}$ & $6.5349\times 10^{-12}$ & $2.4224\times 10^{-8}$ \\
		$\|\mathcal F(u_N)\|$  & $4.0711\times 10^{-12}$ & $6.5350\times 10^{-12}$ & $9.8384\times 10^{-8}$ \\
		$10^{-4}$              & $4.0712\times 10^{-12}$ & $6.5351\times 10^{-12}$ & $2.4224\times 10^{-6}$ \\
		$10^{-6}$              & $4.0711\times 10^{-12}$ & $6.5350\times 10^{-12}$ & $2.4224\times 10^{-4}$ \\
		\bottomrule
	\end{tabular}
\end{table}

\begin{figure}[t]
	\centering
	\begin{subfigure}[t]{0.60\textwidth}
		\centering
		\includegraphics[width=\textwidth]{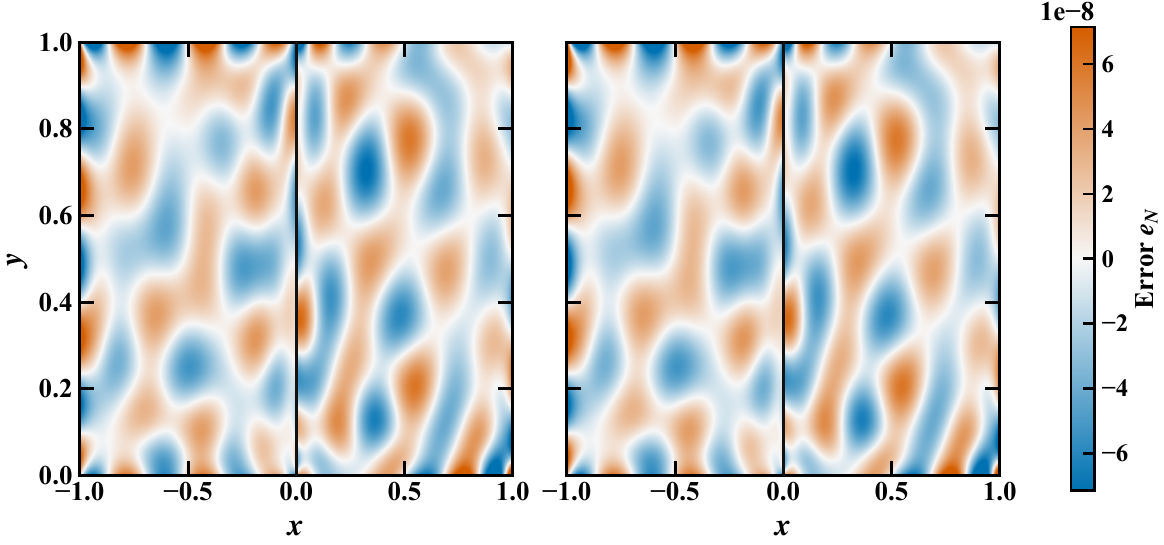}
		\caption{Left: primary error $e_N$. Right: correction $\epsilon u_p$.}
		\label{fig:stage1_error}
	\end{subfigure}\hfill
	\begin{subfigure}[t]{0.36\textwidth}
		\centering
		\includegraphics[width=\textwidth]{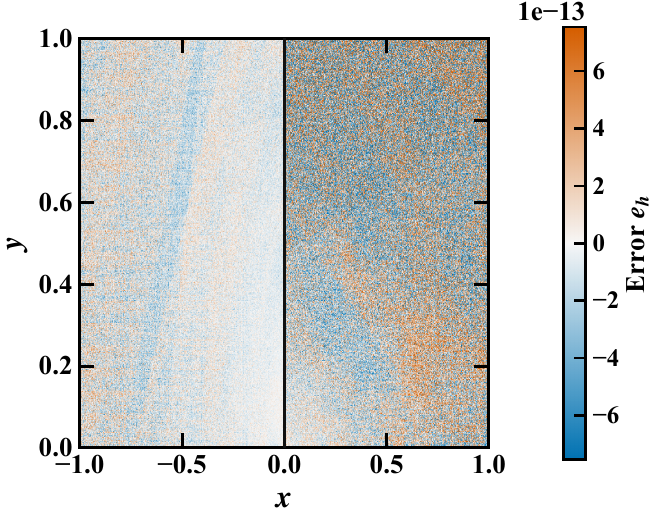}
		\caption{Pointwise error $e_h$.}
		\label{fig:perturbation_solution}
	\end{subfigure}

	\par\medskip

	\begin{subfigure}[t]{0.52\textwidth}
		\centering
		\includegraphics[width=\textwidth]{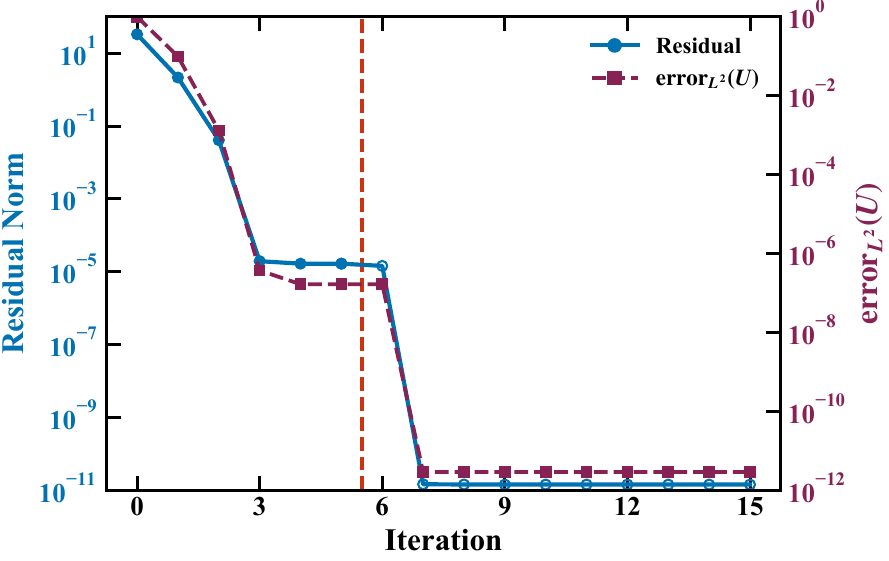}
		\caption{Residual norm and $\erro_{L^2}(U)$ versus iteration.}
		\label{fig:residual_vs_iteration}
	\end{subfigure}\hfill
	\begin{subfigure}[t]{0.435\textwidth}
		\centering
		\includegraphics[width=\textwidth]{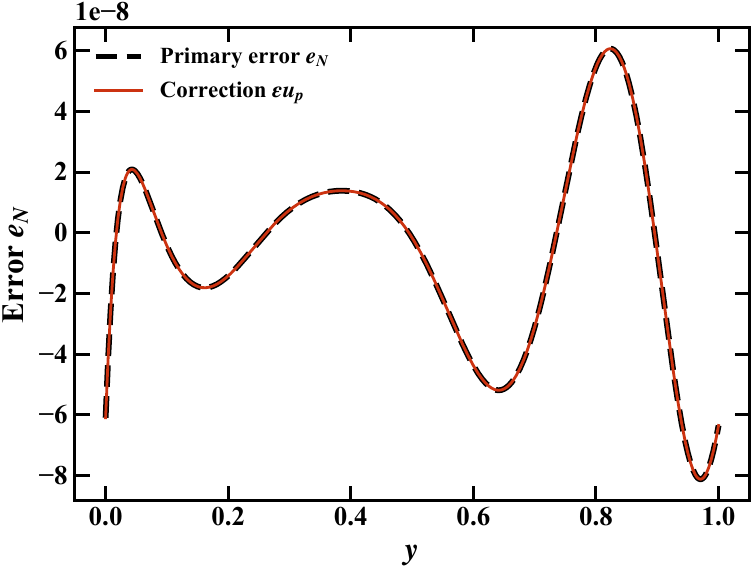}
		\caption{Interface trace of the primary error $e_N$ and the correction. Blue: $e_N$. Red dashed: correction $\epsilon u_p$.}
		\label{fig:interface_error_concise}
	\end{subfigure}
	\caption{Numerical results for Example~1.}
	\label{fig:stage1_comparison}
\end{figure}

Fig~\ref{fig:stage1_comparison} reports the correction and the resulting error. In Fig~\ref{fig:stage1_error}, the left panel shows the primary error $e_N$, and the right panel shows the scaled correction $\epsilon u_p$ computed from the perturbation-correction stage. After adding this correction, the corrected error $e_h$ is reduced to the $10^{-12}$ scale, as shown in Fig~\ref{fig:perturbation_solution}. The convergence history in Fig~\ref{fig:residual_vs_iteration} shows that the residual norm and $\erro_{L^2}(u_N)$ of the primary LRaNN stage stagnate after the initial decrease, whereas the correction stage further reduces the residual norm and $\erro_{L^2}(u_h)$. Finally, Fig~\ref{fig:interface_error_concise} shows the primary error $e_N$ and the correction $\epsilon u_p$ along the interface $\Gamma$, where the two curves overlap. These observations show that, for this baseline two-subdomain test, the perturbation-correction stage captures the dominant residual error left by the primary LRaNN approximation.

\subsection{Example 2}
In the second example, we consider a four-subdomain quasi-linear interface problem with homogeneous jump conditions ($w=v=0$). The computational domain is $\Omega = [-1,1]\times[-1,1]$, and the coordinate axes $x=0$ and $y=0$ partition $\Omega$ into four quadrants
\[
\Omega_1 = [-1,0)\times(0,1],\quad
\Omega_2 = (0,1]\times(0,1],\quad
\Omega_3 = [-1,0)\times[-1,0),\quad
\Omega_4 = (0,1]\times[-1,0),
\]
which meet at the origin. The nonlinear diffusion coefficients are
\[
\beta_1 = 1+u,\quad 
\beta_2 = 1 + 0.5\,u^2,\quad
\beta_3 = 1 + 0.25\,u^2,\quad
\beta_4 = 1 + 0.1\,u^3,
\]
and the source terms are chosen so that the exact solution is
\begin{equation*}
	u(x,y) =
	\begin{cases}
		x^2y^2, & (x,y)\in\Omega_1,\\
		e^{-x}x^2y^2, & (x,y)\in\Omega_2,\\
		e^{-y}x^2y^2, & (x,y)\in\Omega_3,\\
		e^{-(x+y)}x^2y^2, & (x,y)\in\Omega_4.
	\end{cases}
\end{equation*}

We assign to each subdomain a primary $\tanh$ network $u_N^\pm$ and a correction $\sin$ network $u_p^\pm$. Each primary network uses $100$ neurons, with hidden weights and biases sampled from $\mathcal N(0,1)$ and $\mathcal N(0,0.1^2)$; the Gauss--Newton iteration is run for $5$ to $6$ steps, by which point the residual norm has stagnated. Each correction network uses $400$ neurons, with hidden weights and biases sampled from $\mathcal N(0,(5\pi)^2)$ and $\mathcal N(0,1)$, and the perturbation-correction subproblem is solved for another $3$ to $4$ Gauss--Newton steps. The perturbation parameter $\epsilon$ is set to the scale of the primary residual norm $\|\mathcal F(u_N)\|$.

Table~\ref{tab:subdomain_errors} reports $\erro_{L^2}(U)$ and $\erro_{L^\infty}(U)$ in the four subdomains. The LRaNN columns give the primary error from the Gauss--Newton iterations, and the LRaNN-PC columns give the corrected error. The primary error attains $\erro_{L^2}(u_N)$ and $\erro_{L^\infty}(u_N)$ at the $10^{-7}$--$10^{-8}$ level in all four subdomains, and the perturbation-correction stage reduces both errors to the $10^{-12}$--$10^{-13}$ level. Fig~\ref{fig:stage1_comparison_exp1_new} shows primary error and correction error of this example. Panel~(a) shows the surfaces of the primary error $e_N$ (about $O(10^{-7})$) and the correction $\epsilon u_p$, which match in both shape and amplitude. The primary error surface is oscillatory and high-frequency, which is why the perturbation-correction stage uses a sinusoidal correction basis. Panel~(b) shows the corrected error $e_h$, which reaches the $O(10^{-12})$ level. Panels~(c) and~(d) show the primary error $e_N$ and $\epsilon u_p$ along the interfaces $x=0$ and $y=0$, where the two curves coincide.

\begin{figure}[t]
	\centering
	\begin{subfigure}[t]{0.65\textwidth}
		\centering
		\includegraphics[width=\textwidth]{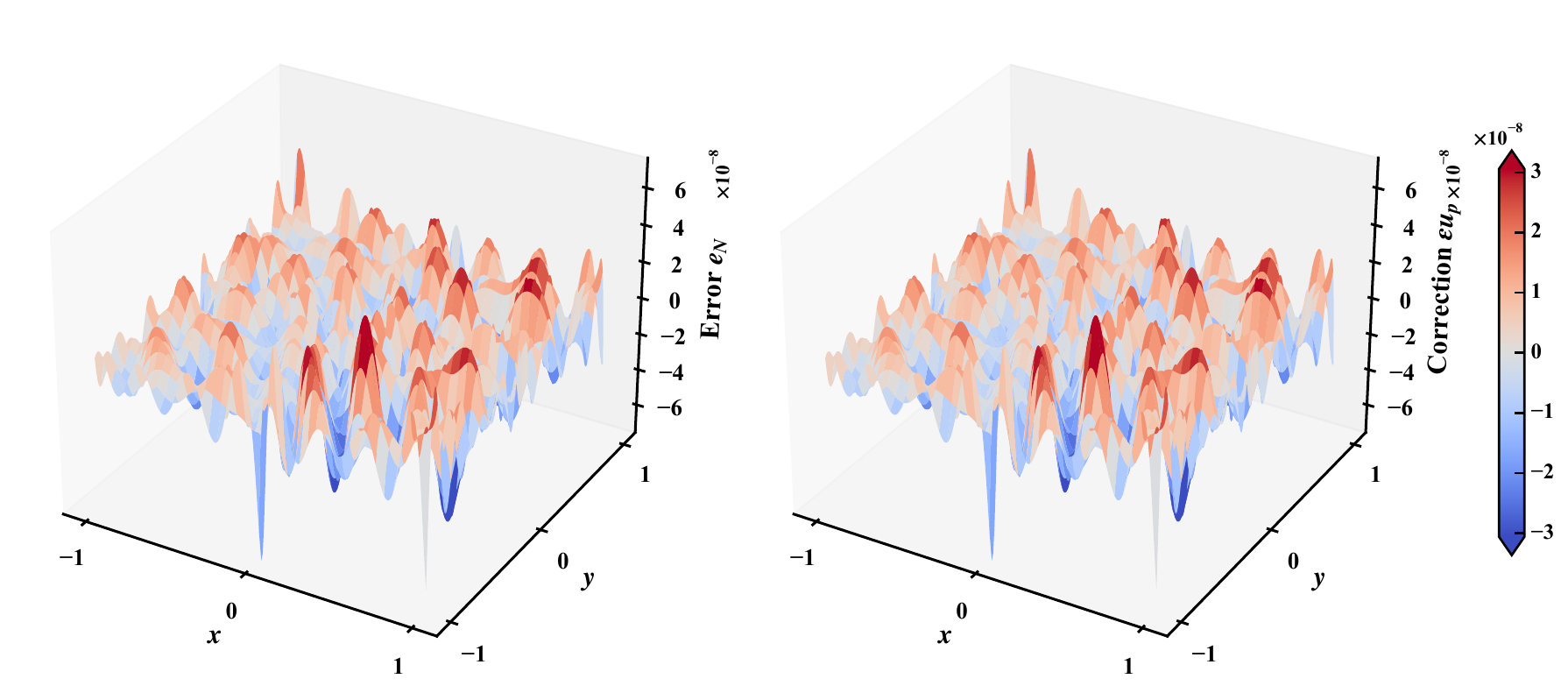}
		\caption{Left: primary error $e_N$. Right: correction $\epsilon u_p$.}
		\label{fig:stage1_error_exp1_new}
	\end{subfigure}\hfill
	\begin{subfigure}[t]{0.33\textwidth}
		\centering
		\includegraphics[width=\textwidth]{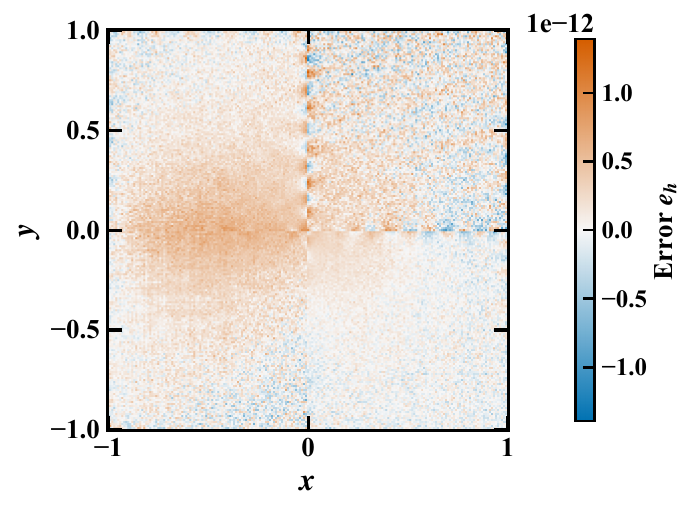} 
		\caption{Pointwise error $e_h$.}
		\label{fig:perturbation_solution_exp1_new}
	\end{subfigure}

	\par\medskip

	\begin{subfigure}[t]{0.48\textwidth}
		\centering
		\includegraphics[width=\textwidth]{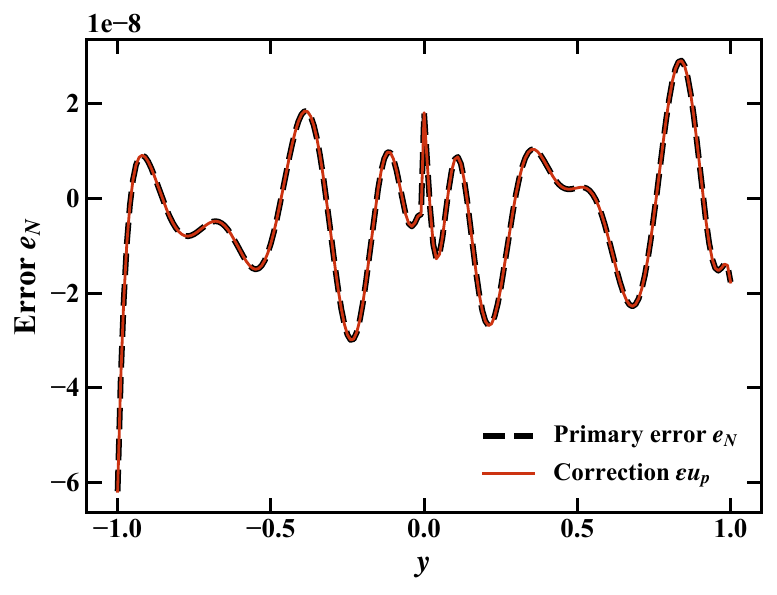} 
    	\caption{Primary error $e_N$ and $\epsilon u_p$ at interface $x = 0$. Black dashed: $e_N$. Red solid: $\epsilon u_p$.}
		\label{fig:interface_error_concise_x0_exp1_new}
	\end{subfigure}\hfill
	\begin{subfigure}[t]{0.48\textwidth}
		\centering
		\includegraphics[width=\textwidth]{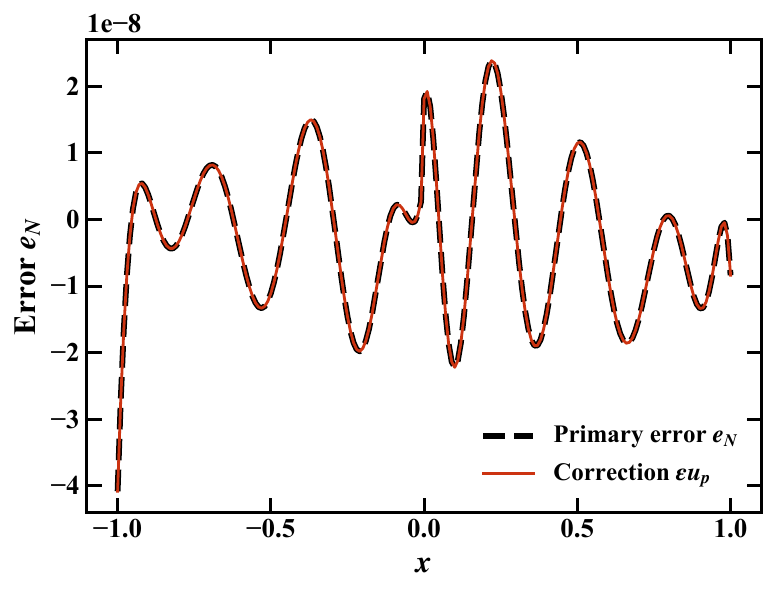} 
		\caption{Primary error $e_N$ and $\epsilon u_p$ at interface $y = 0$. Black dashed: $e_N$. Red solid: $\epsilon u_p$.}
		\label{fig:interface_error_concise_y0_exp1_new}
	\end{subfigure}
	\caption{Numerical results of Example 2.}
	\label{fig:stage1_comparison_exp1_new}
\end{figure}

\begin{table}[h]
	\centering
	\caption{Relative error comparison for the four-subdomain problem (Example 2).}
	\label{tab:subdomain_errors}
	\begin{tabular}{lcccc}
		\toprule
	Subdomain & \multicolumn{2}{c}{LRaNN} & \multicolumn{2}{c}{LRaNN-PC} \\
		\cmidrule(lr){2-3} \cmidrule(lr){4-5}
		& $\erro_{L^2}(u_N)$ & $\erro_{L^\infty}(u_N)$ & $\erro_{L^2}(u_h)$ & $\erro_{L^\infty}(u_h)$ \\
		\midrule
		$\Omega_1$ & $4.5790\times 10^{-8}$ & $5.5010\times 10^{-8}$ & $8.6359\times 10^{-13}$ & $8.8042\times 10^{-13}$ \\
		$\Omega_2$ & $1.8043\times 10^{-7}$ & $8.5307\times 10^{-8}$ & $1.4879\times 10^{-12}$ & $1.1835\times 10^{-12}$ \\
		$\Omega_3$ & $1.2109\times 10^{-7}$ & $6.3436\times 10^{-8}$ & $1.5871\times 10^{-12}$ & $1.0106\times 10^{-12}$ \\
		$\Omega_4$ & $2.3272\times 10^{-7}$ & $3.1417\times 10^{-8}$ & $4.6773\times 10^{-12}$ & $1.6778\times 10^{-12}$ \\
		\bottomrule
	\end{tabular}
\end{table}

To separate the effect of the perturbation-correction formulation from that of iteration count, network width, and basis functions, Table~\ref{tab:ablation_perturbation} reports an ablation study over $10$ random seeds. The proposed perturbation-correction method ($\tanh$ primary for $10$ steps, $\sin$ correction for $10$ steps) reaches a mean relative $L^2$ error of about $4\times10^{-11}$, with individual seeds ranging from about $2\times10^{-12}$ to $1\times10^{-10}$; the representative run in Table~\ref{tab:subdomain_errors} lies at the more accurate end of this range. This error is several orders of magnitude below every alternative considered. Continuing the primary $\tanh$ networks for $20$ steps (Variant~A) or enlarging the primary $\tanh$ networks to a total of $2000$ neurons (Variant~E, i.e.\ $500$ per subdomain, matching the $400$ primary plus $1600$ correction neurons that the two-stage method uses across the four subdomains) leaves $\erro_{L^2}(U)$ at the $10^{-6}$ level, so neither additional $\tanh$ iterations nor a larger $\tanh$ dictionary helps once the primary stage has converged. Using the $\sin$ network alone with $u_0=0$ (Variant~B) yields errors of order $10^{-3}$, so the high-frequency basis is not by itself a suitable approximation space. Placing all basis functions in a single network and training them directly (Variant~C) improves on Variant~B but does not match the staged formulation, so the gain is not explained by the presence of high-frequency functions alone. Replacing the $\sin$ correction by a $\tanh$ correction in the two-stage framework (Variant~D) lowers the error relative to the single-stage $\tanh$ continuation but remains well above the $\sin$-correction result, since the residual is dominated by high-frequency components that the $\tanh$ basis represents less efficiently. These results indicate that the accuracy gain comes from the staged perturbation-correction formulation together with a high-frequency correction basis, rather than from more iterations, a wider $\tanh$ network, or pooling all basis functions in one stage.

To examine the dependence on the collocation resolution, we further refine the interior, interface, and boundary collocation sets for Example~2. The number $N_\Omega$ in Table~\ref{tab:exp2-collocation-refinement} denotes the total number of interior collocation points over the four subdomains; the numbers of interface and boundary points are increased proportionally. The errors are evaluated on an independent test set that is not used in training. Except for the coarsest collocation set, where the correction problem is under-resolved, the perturbation-correction stage consistently reduces both the relative $L^2$ error and the relative broken $H^1$ error as the collocation set is refined.

\begin{table}[h]
	\centering
	\caption{Collocation refinement study for Example~2. The errors are evaluated on an independent test set. Here $N_\Omega$ denotes the total number of interior collocation points over the four subdomains.}
	\label{tab:exp2-collocation-refinement}
	\begin{tabular}{ccccc}
		\toprule
		\multirow{2}{*}{$N_\Omega$} 
		& \multicolumn{2}{c}{LRaNN} 
		& \multicolumn{2}{c}{LRaNN-PC} \\
		\cmidrule(lr){2-3} \cmidrule(lr){4-5}
		& $\erro_{L^2}(u_N)$ & $\erro_{H^1}(u_N)$
		& $\erro_{L^2}(u_h)$ & $\erro_{H^1}(u_h)$ \\
		\midrule
		$1024$  & $5.4008\times 10^{-5}$ & $7.3162\times 10^{-5}$ & $2.4639\times 10^{-3}$ & $3.5852\times 10^{-3}$ \\
		$2304$  & $1.1287\times 10^{-5}$ & $3.7411\times 10^{-5}$ & $3.8769\times 10^{-8}$ & $4.1368\times 10^{-8}$ \\
		$4096$  & $6.2894\times 10^{-6}$ & $3.2503\times 10^{-5}$ & $3.3150\times 10^{-9}$ & $3.1522\times 10^{-9}$ \\
		$9216$  & $1.1736\times 10^{-6}$ & $7.0121\times 10^{-6}$ & $4.4473\times 10^{-11}$ & $5.5110\times 10^{-11}$ \\
		$16384$ & $1.0858\times 10^{-7}$ & $7.6105\times 10^{-7}$ & $3.4828\times 10^{-12}$ & $1.1065\times 10^{-11}$ \\
		\bottomrule
	\end{tabular}
\end{table}

\begin{table}[t]
	\centering
	\caption{Ablation study for the four-subdomain problem (Example~2). Entries are $\erro_{L^2}(U)$ and $\erro_{L^\infty}(U)$ as mean $\pm$ standard deviation over $10$ random seeds.}
	\label{tab:ablation_perturbation}
	\begin{tabular}{lcc}
		\toprule
		Experiment & $\erro_{L^2}(U)$ & $\erro_{L^\infty}(U)$ \\
		\midrule
		Proposed: $\tanh\times10$ + $\sin$-pert $\times10$
		& $4.3436\times 10^{-11} \pm 3.89\times 10^{-11}$
		& $4.4696\times 10^{-11} \pm 4.05\times 10^{-11}$ \\
		Variant~A: $\tanh \times 20$ 
		& $1.2623\times 10^{-6} \pm 1.01\times 10^{-6}$ 
		& $8.1306\times 10^{-7} \pm 6.55\times 10^{-7}$ \\
		Variant~B: $\sin$-pert $\times 20$ ($u_0=0$) 
		& $2.2456\times 10^{-3} \pm 2.50\times 10^{-7}$ 
		& $1.1677\times 10^{-3} \pm 1.09\times 10^{-8}$ \\
		Variant~C: combined $\tanh$+$\sin$ $\times 20$ 
		& $7.2290\times 10^{-7} \pm 5.68\times 10^{-7}$ 
		& $4.3126\times 10^{-7} \pm 3.34\times 10^{-7}$ \\
		Variant~D: $\tanh \times 10$ + $\tanh$-pert $\times 10$ 
		& $1.4899\times 10^{-7} \pm 6.33\times 10^{-8}$ 
		& $1.8392\times 10^{-7} \pm 5.99\times 10^{-8}$ \\
		Variant~E: $\tanh \times 20$ ($2000$ neurons) 
		& $4.9467\times 10^{-6} \pm 1.53\times 10^{-6}$ 
		& $3.0417\times 10^{-6} \pm 8.53\times 10^{-7}$ \\
		\bottomrule
	\end{tabular}
\end{table}

\subsection{Example 3}
 
In this example, we consider a quasi-linear elliptic interface problem with a curved, multi-lobe interface and nonlinear solution-dependent diffusion coefficients. The computational domain is $\Omega = [-1,1] \times [-1,1]$, and the interface $\Gamma$ is given in polar coordinates $(r,\theta)$ by
\begin{equation*}
	r(\theta) = R_0 + A_{\mathrm{plum}} \cos(m_{\mathrm{petal}}\,\theta),
\end{equation*}
where $m_{\mathrm{petal}}$ controls the number of lobes and the fixed constants are $R_0 = 0.5$ and $A_{\mathrm{plum}} = 0.3$. The interface divides $\Omega$ into the interior subdomain $\Omega^{-}$ and the exterior subdomain $\Omega^{+}$.
 
The nonlinear diffusion coefficient is
\begin{equation*}
	\beta(x,u) =
	\begin{cases}
		x^2 + y^2 + \exp\!\left(\tfrac12 u\right), & (x,y) \in \Omega^+, \\[0.4em]
		1 + \sin(u), & (x,y) \in \Omega^-,
	\end{cases}
\end{equation*}
and the source terms $f^\pm$ together with the interface data $w$ and $v$ are chosen so that the exact solution is
\begin{equation*}
	u(x,y) =
	\begin{cases}
		0.5\,\sin(\pi x)\,\sin(\pi y) + 0.25, & (x,y)\in\Omega^+, \\[0.4em]
		0.25 - (x^2+y^2), & (x,y)\in\Omega^-.
	\end{cases}
\end{equation*}
 
We use local randomized neural networks $u_N^\pm$ to approximate the interior and exterior solutions independently, each with $m=500$ neurons and $\tanh$ activation. The hidden weights and biases are sampled from $\mathcal N(0,1)$ and $\mathcal N(0,0.1^2)$, and the Gauss--Newton iteration is run until the residual norm ceases to decay. The perturbation-correction subproblem is then solved with two local randomized networks (one per subdomain) of $m_p=2000$ neurons and $\sin$ activation, whose hidden weights and biases are sampled from $\mathcal N(0,(7\pi)^2)$ and $\mathcal N(0,1)$, respectively.
 
\begin{figure}[t]
	\centering
	\begin{subfigure}[t]{0.32\textwidth}
		\includegraphics[width=\textwidth]{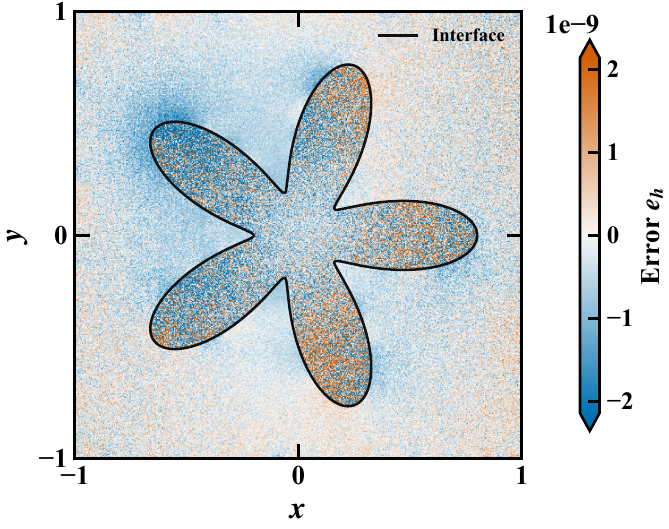}
		\caption{$m_{\mathrm{petal}}=5$: corrected error $e_h$.}
	\end{subfigure}
	\hfill
	\begin{subfigure}[t]{0.32\textwidth}
		\includegraphics[width=\textwidth]{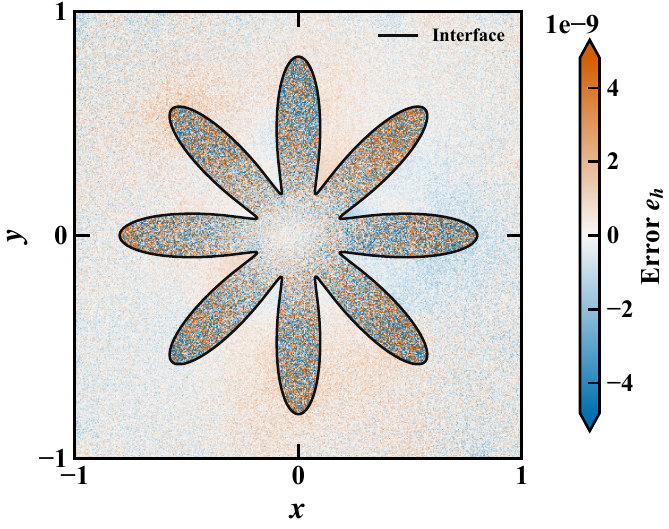}
		\caption{$m_{\mathrm{petal}}=8$: corrected error $e_h$.}
	\end{subfigure}
	\hfill
	\begin{subfigure}[t]{0.32\textwidth}
		\includegraphics[width=\textwidth]{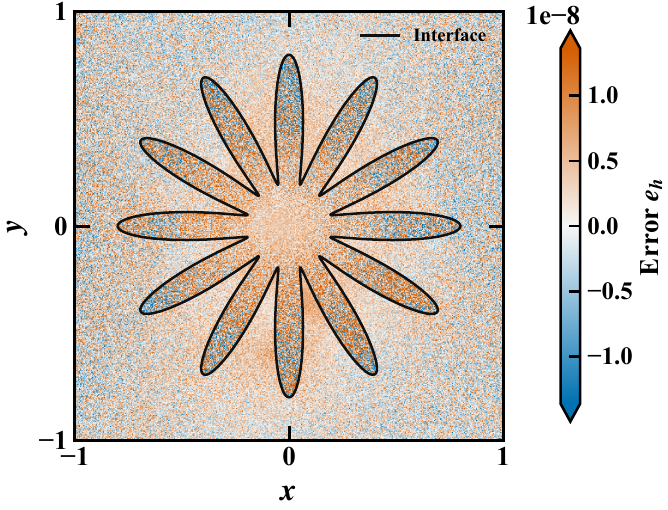}
		\caption{$m_{\mathrm{petal}}=12$: corrected error $e_h$.}
	\end{subfigure}
 
	\par\medskip
 
	\begin{subfigure}[t]{0.32\textwidth}
		\includegraphics[width=\textwidth]{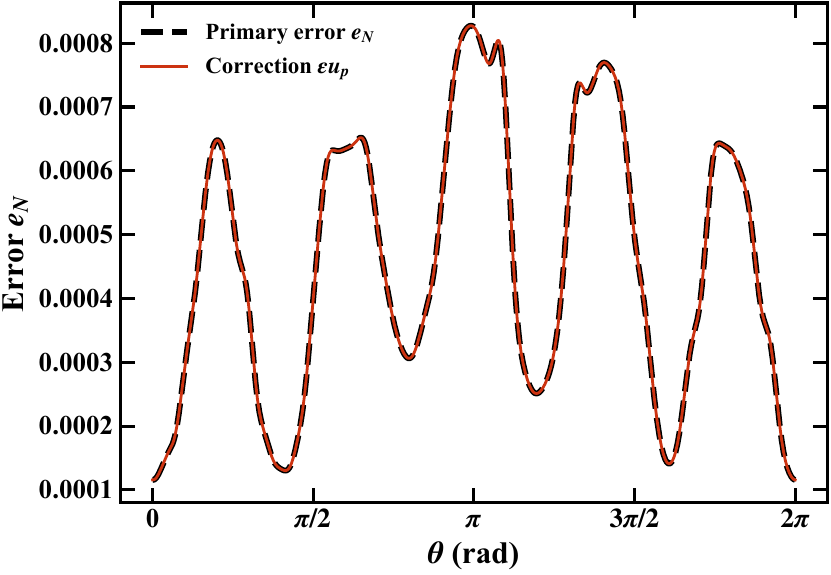}
		\caption{$m_{\mathrm{petal}}=5$: interface trace.}
	\end{subfigure}
	\hfill
	\begin{subfigure}[t]{0.32\textwidth}
		\includegraphics[width=\textwidth]{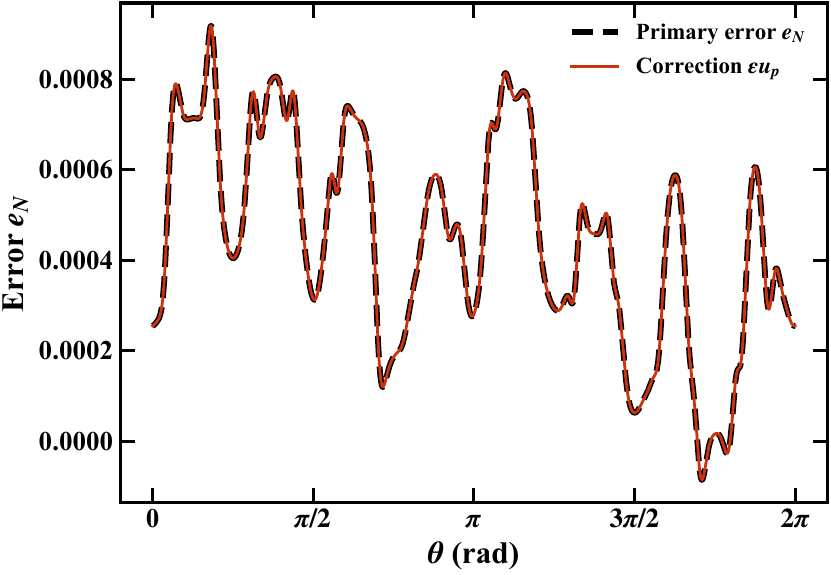}
		\caption{$m_{\mathrm{petal}}=8$: interface trace.}
	\end{subfigure}
	\hfill
	\begin{subfigure}[t]{0.32\textwidth}
		\includegraphics[width=\textwidth]{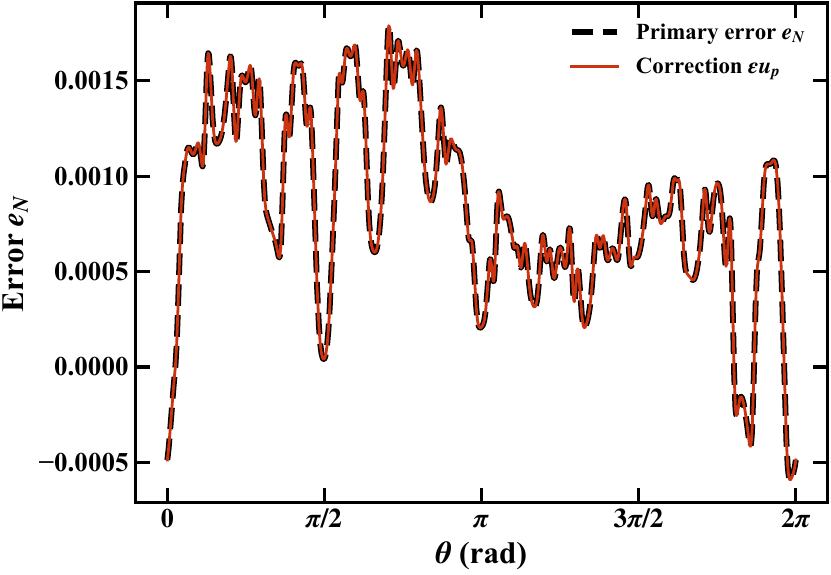}
		\caption{$m_{\mathrm{petal}}=12$: interface trace.}
	\end{subfigure}
	\caption{Example~3, plum-blossom interface with $m_{\mathrm{petal}}=5,8,12$ (columns). Top row (a)--(c): corrected error $e_h$. Bottom row (d)--(f): primary error $e_N$ along the interface $\Gamma$ (blue) and the computed correction $\epsilon u_p$ (red dashed); the two curves overlap.}
	\label{fig:comparison_exp3_newcurve}
\end{figure}
 
Table~\ref{tab:error_analysis_exp3} reports $\erro_{L^2}(U)$ and $\erro_{L^\infty}(U)$ for the primary approximation (LRaNN) and the corrected approximation (LRaNN-PC) at $m_{\mathrm{petal}}=5,8,12$, where LRaNN-PC uses the second-order (two-term) correction. The primary LRaNN stage stagnates at the $10^{-3}$ level for all three geometries, while the perturbation-correction stage reduces both error indicators by about five orders of magnitude, to the $10^{-9}$--$10^{-8}$ range. The primary error grows with $m_{\mathrm{petal}}$, while the corrected error stays in the $10^{-9}$ to $10^{-8}$ range for all three geometries. Fig~\ref{fig:comparison_exp3_newcurve} shows that the corrected error $e_h$ is reduced uniformly across the domain, including the lobed interface region, and that along $\Gamma$ the computed correction $\epsilon u_p$ overlaps the primary error $e_N$, so that the residual is reduced after correction.
\begin{table}[h]
	\centering
	\caption{Relative errors for the plum-blossom interface (Example~3). LRaNN-PC uses the second-order correction.}
	\label{tab:error_analysis_exp3}
	\vspace{0.5em}
	\begin{tabular}{ccccc}
		\toprule
		\multicolumn{1}{c}{$m_{\mathrm{petal}}$} & \multicolumn{2}{c}{LRaNN} & \multicolumn{2}{c}{LRaNN-PC} \\
		\cmidrule(lr){2-3} \cmidrule(lr){4-5}
		& $\erro_{L^2}(u_N)$ & $\erro_{L^\infty}(u_N)$  & $\erro_{L^2}(u_h)$ & $\erro_{L^\infty}(u_h)$\\
		\midrule
		5  & $1.0585 \times 10^{-3}$ & $1.1020 \times 10^{-3}$ & $3.1477 \times 10^{-9}$ & $1.5230 \times 10^{-8}$ \\
		8  & $1.1565 \times 10^{-3}$ & $1.3638 \times 10^{-3}$ & $6.6245 \times 10^{-9}$ & $3.1599 \times 10^{-8}$ \\
		12 & $2.3117 \times 10^{-3}$ & $2.8851 \times 10^{-3}$ & $2.1265 \times 10^{-8}$ & $7.8628 \times 10^{-8}$ \\
		\bottomrule
	\end{tabular}
\end{table}
 
To check that these accuracies are not tied to one random initialization, we repeat the second-order LRaNN-PC stage for each geometry over $10$ random seeds, redrawing the hidden weights and biases each time. Table~\ref{tab:exp3-seed-sweep} reports the resulting mean $\pm$ standard deviation. The corrected $\erro_{L^2}(u_h)$ stays in the $10^{-9}$ to $10^{-8}$ band for every seed, with means of $6.4\times10^{-9}$, $1.2\times10^{-8}$, and $9.2\times10^{-9}$ at $m_{\mathrm{petal}}=5,8,12$. The standard deviation within each geometry is of the same order as the mean and comparable to the spread across geometries, so the corrected error does not depend monotonically on $m_{\mathrm{petal}}$ once seed variability is included. The representative single-run values in Table~\ref{tab:error_analysis_exp3} sit at the favorable end of this distribution for $m_{\mathrm{petal}}=5$ and at the unfavorable end for $m_{\mathrm{petal}}=12$, which overstates the geometry dependence. The $\erro_{L^\infty}(u_h)$ statistics behave the same way.
 
\begin{table}[h]
	\centering
	\caption{Robustness of the corrected (second-order LRaNN-PC) relative errors to the random initialization for the plum-blossom interface (Example~3). Entries are $\erro_{L^2}(u_h)$ and $\erro_{L^\infty}(u_h)$ as mean $\pm$ standard deviation over $10$ random seeds.}
	\label{tab:exp3-seed-sweep}
	\vspace{0.5em}
	\begin{tabular}{ccc}
		\toprule
		$m_{\mathrm{petal}}$ & $\erro_{L^2}(u_h)$ & $\erro_{L^\infty}(u_h)$ \\
		\midrule
		5  & $6.383\times 10^{-9} \pm 2.54\times 10^{-9}$ & $2.677\times 10^{-8} \pm 9.82\times 10^{-9}$ \\
		8  & $1.198\times 10^{-8} \pm 7.43\times 10^{-9}$ & $5.954\times 10^{-8} \pm 4.39\times 10^{-8}$ \\
		12 & $9.180\times 10^{-9} \pm 6.63\times 10^{-9}$ & $3.355\times 10^{-8} \pm 1.23\times 10^{-8}$ \\
		\bottomrule
	\end{tabular}
\end{table}
 
\begin{table}[t]
	\centering
	\caption{Relative errors obtained with one-, two-, and three-term expansions for different $m_{\mathrm{petal}}$ (Example~3). The one-, two-, and three-term models retain the linear, quadratic, and cubic terms, respectively, in the residual expansion with respect to the actual correction $\delta u=\epsilon u_p$.}
	\label{tab:mpetal-epsilon2}
	\begin{tabular}{ccccccc}
		\toprule
		& \multicolumn{2}{c}{One-term (linear)} 
& \multicolumn{2}{c}{Two-term (quadratic)} 
& \multicolumn{2}{c}{Three-term (cubic)} \\
		\cmidrule(lr){2-3}\cmidrule(lr){4-5}\cmidrule(lr){6-7}
		$m_{\mathrm{petal}}$ & $\erro_{L^2}(u_h)$ & $\erro_{L^\infty}(u_h)$ & $\erro_{L^2}(u_h)$ & $\erro_{L^\infty}(u_h)$ & $\erro_{L^2}(u_h)$ & $\erro_{L^\infty}(u_h)$ \\
		\midrule
		5  & $2.3109\times 10^{-7}$ & $3.6438\times 10^{-7}$ & $3.1479\times 10^{-9}$ & $1.5229\times 10^{-8}$ & $3.1402\times 10^{-9}$ & $1.5221\times 10^{-8}$ \\
		8  & $1.8863\times 10^{-7}$ & $4.8269\times 10^{-7}$ & $6.6244\times 10^{-9}$ & $3.1596\times 10^{-8}$ & $6.6254\times 10^{-9}$ & $3.1614\times 10^{-8}$ \\
		12 & $6.8426\times 10^{-7}$ & $2.8138\times 10^{-6}$ & $2.1265\times 10^{-8}$ & $7.8621\times 10^{-8}$ & $2.1314\times 10^{-8}$ & $7.8665\times 10^{-8}$ \\
		\bottomrule
	\end{tabular}
\end{table}
 
Table~\ref{tab:mpetal-epsilon2} compares the corrected accuracy when the residual expansion is truncated after the linear, quadratic, and cubic correction terms. The two-term expansion lowers $\erro_{L^2}(u_h)$ by about two orders of magnitude relative to the one-term expansion. The three-term expansion gives essentially the same accuracy as the two-term one, which suggests that, for this test, the dominant error after the quadratic model is no longer the neglected higher-order correction term. The corresponding correction-stage solve times are about $45$--$64$\,s for the one-term expansion, $86$--$108$\,s for the two-term expansion, and $86$--$105$\,s for the three-term expansion: the two- and three-term solves cost about the same, while the one-term solve is cheaper but roughly two orders of magnitude less accurate. Since the third-order expansion does not improve on the second-order one, we adopt the second-order correction as the default LRaNN-PC correction, and report the three-term results only for completeness.
\subsection{Example 4}
In this example, we consider a high-contrast elliptic interface problem with a large constant diffusion coefficient $\beta^+\gg1$ in $\Omega^+$ and a nonlinear diffusion $\beta^-(u)$ in $\Omega^-$. The problem is posed on $\Omega=[-1,1]\times[-1,1]$ with the circular interface
\begin{equation*}
	\Gamma = \{(x,y)\mid x^2+y^2=1/4\},
\end{equation*}
which partitions $\Omega$ into the inner subdomain $\Omega^-=\{x^2+y^2<1/4\}$ and the outer subdomain $\Omega^+=\Omega\setminus\overline{\Omega^-}$. In $\Omega^+$ the diffusion is linear with a high constant conductivity, and in $\Omega^-$ it is nonlinear,
\begin{equation*}
	\beta^-(u)=1+u^3.
\end{equation*}
The source terms and interface data are chosen so that the exact solution is
\begin{equation*}
	u(x,y)=
	\begin{cases}
		\dfrac{x^3+y^3}{\beta^+}, & (x,y)\in\Omega^+,\\[0.4em]
		\sin(\pi x)\sin(\pi y), & (x,y)\in\Omega^-.
	\end{cases}
\end{equation*}

Each subdomain is approximated by a primary $\tanh$ LRaNN with $500$ neurons and a correction $\sin$ LRaNN with $1000$ neurons. The primary weights and biases are sampled from $\mathcal N(0,1)$ and $\mathcal N(0,0.1^2)$, and the correction weights and biases from $\mathcal N(0,(7\pi)^2)$ and $\mathcal N(0,1)$. The Gauss--Newton iteration is run until the residual norm stagnates in each stage, and the perturbation parameter $\epsilon$ is set to the scale of the primary residual norm $\|\mathcal F(u_N)\|$. Because $\beta^+\gg1$ in $\Omega^+$, the flux-jump residual $\mathcal R_{\Gamma_n}$ and the associated rows of the Jacobian carry a much larger magnitude than the remaining residual blocks, which inflates the condition number of the least-squares system. To remove this artificial scale disparity, we divide the flux-jump residual and its Jacobian rows by $\beta^+$ before assembling the normal equations, so that the optimization is not dominated by the high-contrast flux term.

For the representative case $\beta^+/\beta^-=10^{12}$, Fig~\ref{fig:stage1_comparison_exp3} shows the correction and the resulting error. Fig~\ref{fig:stage1_error_exp3} shows the primary error $e_N$ together with the correction $\epsilon u_p$, which match across the domain; Fig~\ref{fig:perturbation_solution_exp3} shows the reduced corrected error $e_h$; Fig~\ref{fig:residual_vs_iteration_exp3} shows that the residual norm and error stagnate in the primary stage and decrease further in the perturbation-correction stage; and Fig~\ref{fig:interface_error_concise_exp3} shows the primary error $e_N$ and $\epsilon u_p$ along $\Gamma$, where the two curves coincide.

\begin{figure}[t]
	\centering
	\begin{subfigure}[t]{0.6\textwidth}
		\centering
		\includegraphics[width=\textwidth]{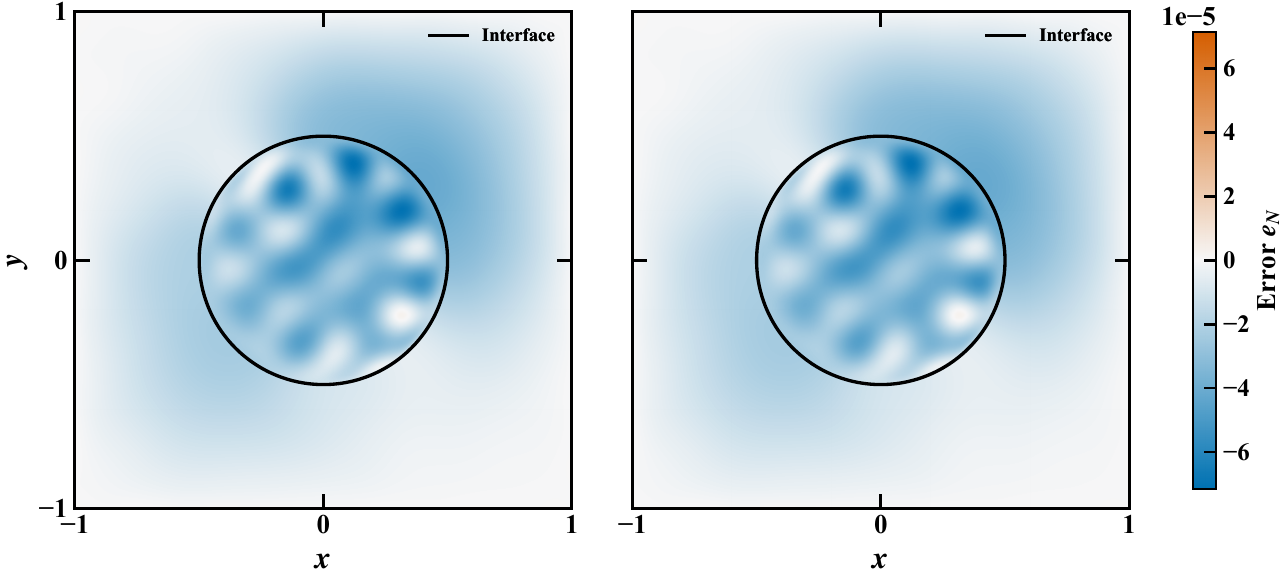}
		\caption{Left: primary error $e_N$. Right: correction $\epsilon u_p$.}
		\label{fig:stage1_error_exp3}
	\end{subfigure}\hfill
	\begin{subfigure}[t]{0.34\textwidth}
		\centering
		\includegraphics[width=\textwidth]{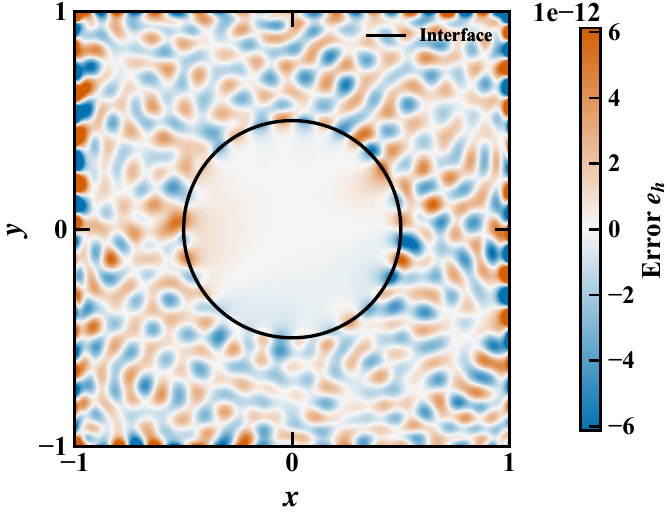} 
		  \caption{Pointwise error $e_h$.}
		\label{fig:perturbation_solution_exp3}
	\end{subfigure}

	\par\medskip

	\begin{subfigure}[t]{0.475\textwidth}
		\centering
		\includegraphics[width=\textwidth]{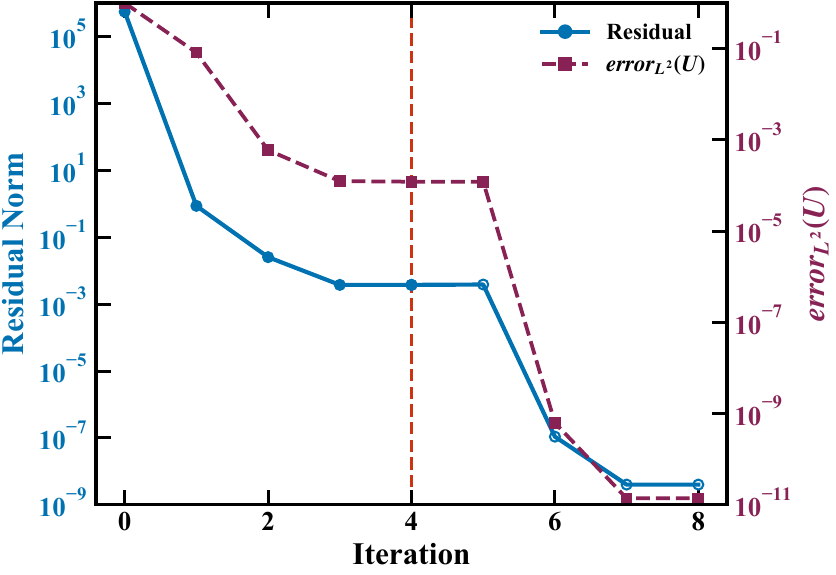} 
		\caption{Residual norm and $\erro_{L^2}(U)$ with iteration.}
		\label{fig:residual_vs_iteration_exp3}
	\end{subfigure}\hfill
	\begin{subfigure}[t]{0.48\textwidth}
		\centering
		\includegraphics[width=\textwidth]{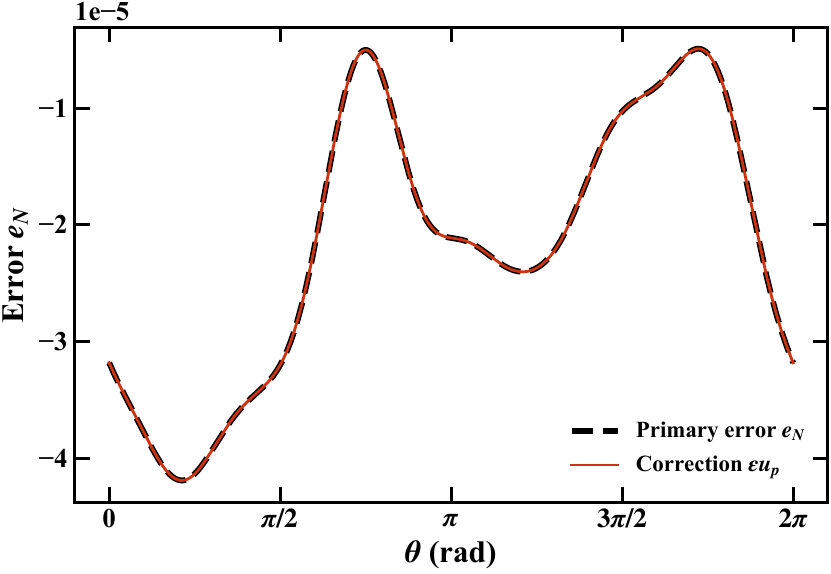} 
		\caption{Primary error $e_N$ and $\epsilon u_p$ at interface $\Gamma$. Black dashed: $e_N$. Red solid: $\epsilon u_p$.}
		\label{fig:interface_error_concise_exp3}
	\end{subfigure}
	\caption{Numerical results for Example~4,high-contrast case $\beta^+/\beta^-=10^{12}$.}
	\label{fig:stage1_comparison_exp3}
\end{figure}

Table~\ref{tab:contrast_comparison} reports the (relative) errors for contrast ratios $\beta^+/\beta^-$ from $10^4$ to $10^{12}$. The primary approximation errors are nearly insensitive to the contrast, with $\erro_{L^2}(u_N)\approx 2.32\times10^{-6}$ and $\erro_{L^\infty}(u_N)\approx 2.92\times10^{-6}$ throughout the tested range. With the flux-jump rescaling described above, the perturbation-correction stage reduces the relative indicators to the $10^{-12}$--$10^{-11}$ level and keeps them nearly constant across the contrast range: the corrected $\erro_{L^2}(u_h)$ stays between $6.31\times10^{-12}$ and $6.97\times10^{-12}$, and $\erro_{L^\infty}(u_h)$ between $1.89\times10^{-11}$ and $2.68\times10^{-11}$. For the representative case $\beta^+/\beta^-=10^{12}$, the corrected relative errors are $\erro_{L^2}(u_h)=6.7670\times10^{-12}$ and $\erro_{L^\infty}(u_h)=2.6583\times10^{-11}$.
\begin{table}[h]
	\centering
	\caption{Relative errors of the two-stage framework for Example~4 under various diffusion contrast ratios $\beta^+/\beta^-$.}
	\label{tab:contrast_comparison}
	\vspace{0.5em}
	\begin{tabular}{ccccc}
		\toprule
		\multicolumn{1}{c}{$\beta^+/\beta^-$} & \multicolumn{2}{c}{LRaNN} & \multicolumn{2}{c}{LRaNN-PC} \\
		\cmidrule(lr){2-3} \cmidrule(lr){4-5}
		& $\erro_{L^2}(u_N)$ & $\erro_{L^\infty}(u_N)$ & $\erro_{L^2}(u_h)$ & $\erro_{L^\infty}(u_h)$ \\
		\midrule
		$10^{4}$  & $2.3164\times 10^{-6}$ & $2.9150\times 10^{-6}$ & $6.3124\times 10^{-12}$ & $1.8925\times 10^{-11}$ \\
		$10^{6}$  & $2.3176\times 10^{-6}$ & $2.9159\times 10^{-6}$ & $6.6070\times 10^{-12}$ & $2.5573\times 10^{-11}$ \\
		$10^{8}$  & $2.3195\times 10^{-6}$ & $2.9171\times 10^{-6}$ & $6.8894\times 10^{-12}$ & $2.6379\times 10^{-11}$ \\
		$10^{10}$ & $2.3197\times 10^{-6}$ & $2.9171\times 10^{-6}$ & $6.9655\times 10^{-12}$ & $2.6831\times 10^{-11}$ \\
		$10^{12}$ & $2.3185\times 10^{-6}$ & $2.9161\times 10^{-6}$ & $6.7670\times 10^{-12}$ & $2.6583\times 10^{-11}$ \\
		\bottomrule
	\end{tabular}
\end{table}

\subsection{Example 5}

The previous examples use manufactured exact solutions. We now consider a physically motivated quasi-linear interface problem that does not admit a closed-form solution, and assess the proposed method against a high-accuracy numerical reference over a range of coefficient contrasts.

The domain is $\Omega=\{(x,y)\mid x^2+y^2<R^2\}$ with $R=1$, and the circular interface
\[
\Gamma=\{(x,y)\mid x^2+y^2=a^2\},\qquad a=0.4,
\]
divides $\Omega$ into an inner subdomain $\Omega^-=\{r<a\}$ and an outer subdomain $\Omega^+=\{a<r<R\}$, with $r=\sqrt{x^2+y^2}$. We consider the quasi-linear elliptic interface problem
\begin{equation*}
\begin{cases}
-\nabla\cdot\!\left(\beta^-(u)\nabla u\right)=f^-(x,u), & x\in\Omega^-,\\[1mm]
-\nabla\cdot\!\left(\beta^+(u)\nabla u\right)=0, & x\in\Omega^+,\\[1mm]
\llbracket u\rrbracket=0,\quad \llbracket \beta(u)\,\partial_{\mathbf n}u\rrbracket=0, & x\in\Gamma,\\[1mm]
u=0, & x\in\partial\Omega,
\end{cases}
\end{equation*}
with solution-dependent diffusion coefficients
\begin{equation*}
\beta^-(u)=k^-\left(1+0.1\,u+0.01\,u^2\right),\qquad
\beta^+(u)=k^+\left(1+0.2\,u+0.05\,u^2\right),
\end{equation*}
where the magnitudes $k^\mp>0$ set the contrast between the two subdomains. The coefficients $\beta^\pm$ are nonlinear in $u$ and discontinuous across $\Gamma$. We test the combinations of $(k^-,k^+)$ reported in Fig~\ref{fig:contrast_error_exp5_heat}, ranging from the balanced case $(1,1)$ to high-contrast cases in which the inner or the outer coefficient is larger by up to eight orders of magnitude. The source is supported in $\Omega^-$,
\begin{equation*}
f^-(x,u)=Q_0\,e^{-\nu r^2}\left(1+\sigma u^2\right),\qquad
Q_0=20,\quad \nu=25,\quad \sigma=0.2,
\end{equation*}
while $\Omega^+$ is source-free; the factor $e^{-\nu r^2}$ localizes $f^-$ near the origin, and $1+\sigma u^2$ makes it grow with $u$. Since $1+\sigma u^2$ increases with $u$, the source satisfies $\partial_u f^->0$ and therefore lies outside the strict one-sided monotonicity assumption~\eqref{eq:one-sided-monotonicity}; it falls under the relaxed regime of Remark~\ref{rmk:source}, in which the positive one-sided Lipschitz part of $f^-$ is small relative to the monotonicity of the diffusion operator. These data admit a heat-conduction interpretation, in which $u$ is the temperature, $\beta^\mp(u)$ are the temperature-dependent conductivities of the two media, $f^-$ is an internal heat source concentrated near the origin, and $u=0$ prescribes a zero temperature on the outer boundary.

Because the coefficients, source, interface, and boundary data are radially symmetric, the exact solution depends only on $r$ and satisfies the two-point boundary-value problem
\begin{equation*}
-\frac{1}{r}\frac{d}{dr}\!\left(r\,\beta(u)\,\frac{du}{dr}\right)=f(r,u),\quad r\in(0,R),
\qquad u'(0)=0,\quad u(R)=0,
\end{equation*}
where $(\beta,f)$ take their $\Omega^-$ values on $(0,a)$ and their $\Omega^+$ values on $(a,R)$, and $u$ and the radial flux $r\,\beta(u)\,u'$ are continuous at $r=a$. For each coefficient combination we solve this radial problem to high accuracy and use it as the reference solution $u_g$. Writing the equation as a first-order system in $u$ and the radial flux, we integrate it with the eighth-order Runge--Kutta method DOP853 under tight tolerances, and determine the unknown center value $u(0)$ by a shooting procedure that enforces $u(R)=0$. The interface $r=a$ is placed at an exact integration node, so that the flux variable is continued across it without smoothing the derivative kink. The radial reference was verified at three tolerance/grid levels, and the two finest DOP853 solutions were indistinguishable to the displayed precision in both radial relative $L^2$ and $L^\infty$ norms. For the representative case the shooting residual was $|U(R)|=2.3656\times10^{-16}$, and the interface continuity checks gave $|[U]|=|[F]|=0$, confirming that the reference is high-precision. The radial symmetry is used only to construct the reference: the proposed method solves the problem as a full two-dimensional interface problem.

In the primary LRaNN stage, $\Omega^-$ and $\Omega^+$ are each approximated by a $\tanh$ network with $500$ neurons, with hidden weights and biases sampled from $\mathcal N(0,1)$ and $\mathcal N(0,0.1^2)$. The Gauss--Newton iteration is run until the residual norm stagnates. The perturbation-correction stage uses an independent $\sin$ correction network with $2000$ neurons per subdomain, with hidden weights and biases sampled from $\mathcal N(0,(7\pi)^2)$ and $\mathcal N(0,1)$; it uses the default second-order correction, and the perturbation parameter $\epsilon$ is set to the final primary residual scale $\|\mathcal F(u_N)\|$. The errors are measured against the reference $u_g$ on a uniform test set covering $\Omega$.

For the representative case $(k^-,k^+)=(1,10^8)$, the primary LRaNN stage reaches $\erro_{L^2}(u_N)=1.3334\times10^{-2}$ and $\erro_{L^\infty}(u_N)=1.5883\times10^{-2}$, and the perturbation-correction stage reduces them to $\erro_{L^2}(u_h)=1.2820\times10^{-9}$ and $\erro_{L^\infty}(u_h)=2.3379\times10^{-9}$. Fig~\ref{fig:stage1_comparison_exp5_heat} shows the correction and the resulting error: Fig~\ref{fig:stage1_error_exp5_heat} shows the primary error $e_N$ together with the correction $\epsilon u_p$; Fig~\ref{fig:perturbation_solution_exp5_heat} shows the corrected error $e_h$; Fig~\ref{fig:residual_vs_iteration_exp5_heat} shows that the residual norm and $\erro_{L^2}(U)$ stagnate in the primary stage and decrease further in the perturbation-correction stage; and Fig~\ref{fig:interface_error_concise_exp5_heat} shows the primary error $e_N$ and $\epsilon u_p$ along $\Gamma$, where the two curves coincide.

Fig~\ref{fig:numerical_solutions_exp5_heat} shows the corrected solution $u_h$ for four representative combinations and illustrates how the solution distribution varies with the contrast between the two subdomains.

Fig~\ref{fig:contrast_error_exp5_heat} summarizes the relative errors over all tested combinations. The primary $\erro_{L^2}(u_N)$ ranges from about $2.2\times10^{-5}$ to about $6.0\times10^{-2}$, the largest value occurring for the balanced high-magnitude case $(k^-,k^+)=(10^4,10^4)$, where the large coefficients reduce the solution amplitude and amplify the relative error; $\erro_{L^\infty}(u_N)$ follows the same pattern. In every combination the perturbation-correction stage reduces both indicators by about four to seven orders of magnitude, bringing $\erro_{L^2}(u_h)$ into the $4.4\times10^{-10}$--$1.9\times10^{-6}$ range, with the largest corrected error again at $(k^-,k^+)=(10^4,10^4)$.

\begin{figure}[t]
	\centering
	\begin{subfigure}[t]{0.6\textwidth}
		\centering
		\includegraphics[width=\textwidth]{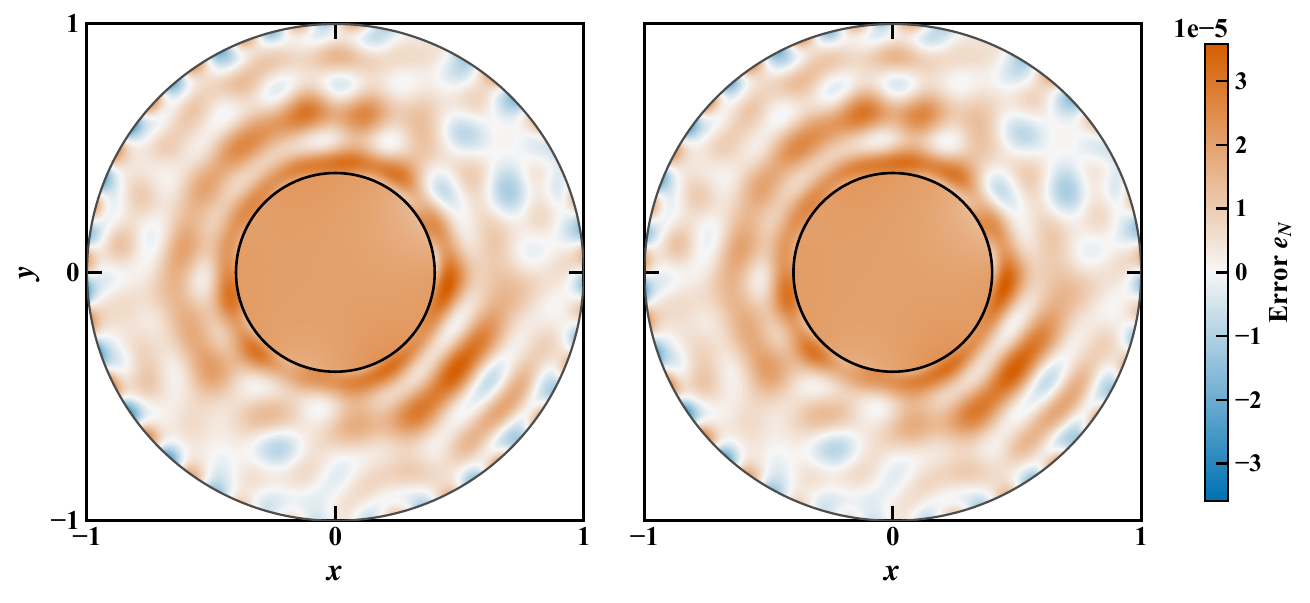}
		\caption{Left: primary error $e_N$. Right: correction $\epsilon u_p$.}
		\label{fig:stage1_error_exp5_heat}
	\end{subfigure}\hfill
	\begin{subfigure}[t]{0.36\textwidth}
		\centering
		\includegraphics[width=\textwidth]{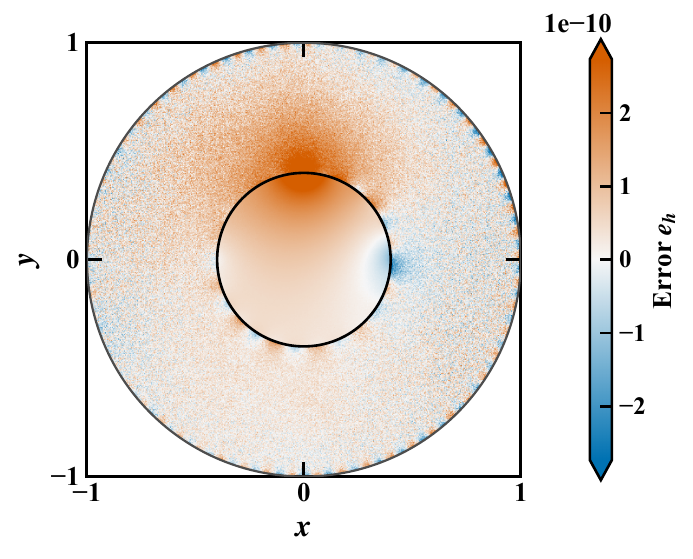}
		\caption{Pointwise error $e_h$.}
		\label{fig:perturbation_solution_exp5_heat}
	\end{subfigure}

	\par\medskip

	\begin{subfigure}[t]{0.52\textwidth}
		\centering
		\includegraphics[width=\textwidth]{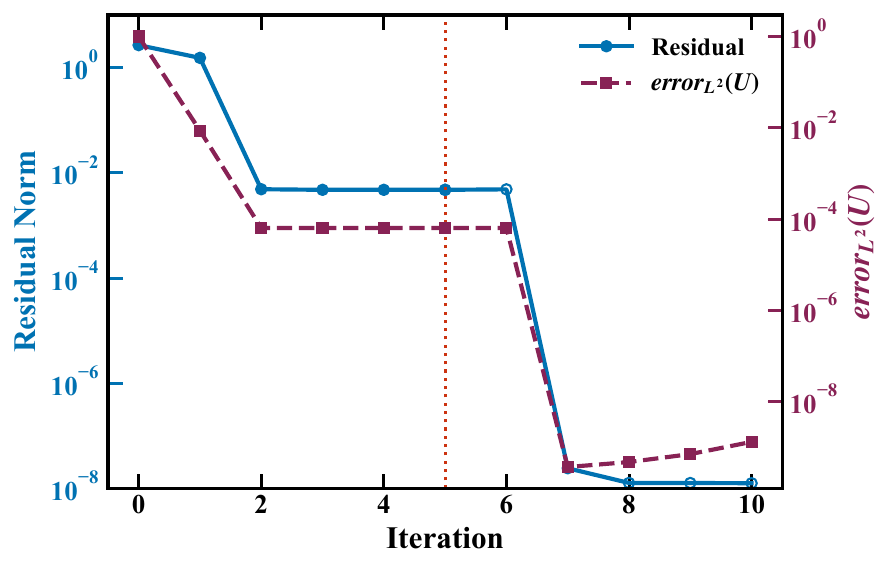}
		\caption{Residual norm and $\erro_{L^2}(U)$ versus iteration.}
		\label{fig:residual_vs_iteration_exp5_heat}
	\end{subfigure}\hfill
	\begin{subfigure}[t]{0.46\textwidth}
		\centering
		\includegraphics[width=\textwidth]{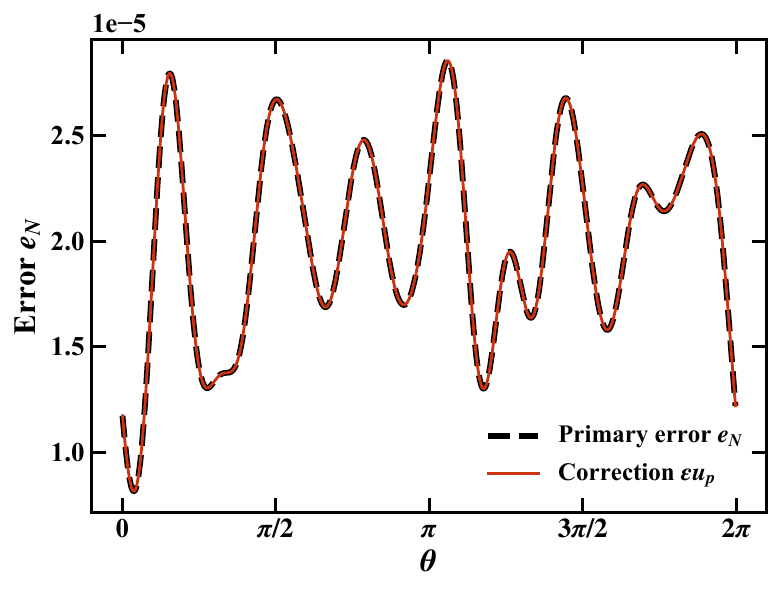}
		\caption{Primary error $e_N$ and $\epsilon u_p$ at interface $\Gamma$. Black dashed: $e_N$. Red solid: $\epsilon u_p$.}
		\label{fig:interface_error_concise_exp5_heat}
	\end{subfigure}
	\caption{Numerical results for Example~5, representative case $(k^-,k^+)=(1,10^8)$.}
	\label{fig:stage1_comparison_exp5_heat}
\end{figure}

\begin{figure}[t]
	\centering
	\begin{subfigure}[t]{0.45\textwidth}
		\includegraphics[width=\textwidth]{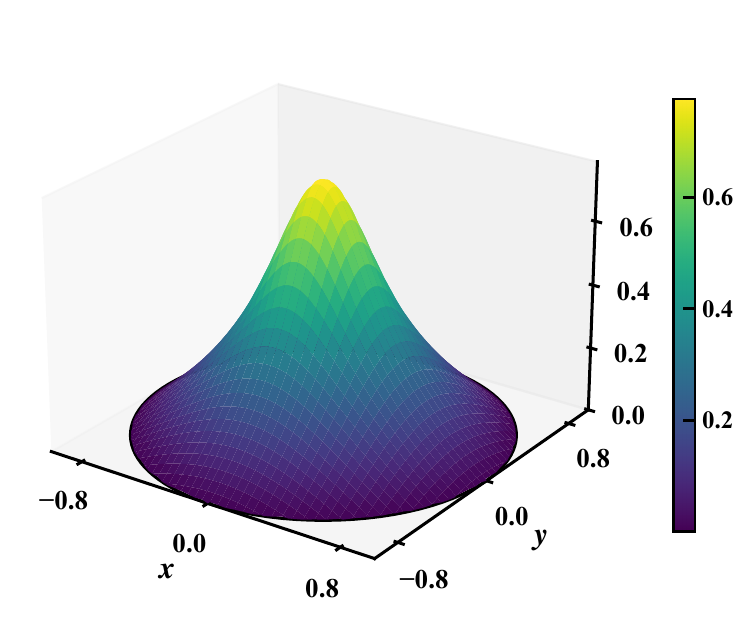}
		\caption{$(k^-,k^+)=(1,1)$.}
		\label{fig:numsol_11_exp5_heat}
	\end{subfigure}
	\hfill
	\begin{subfigure}[t]{0.45\textwidth}
		\includegraphics[width=\textwidth]{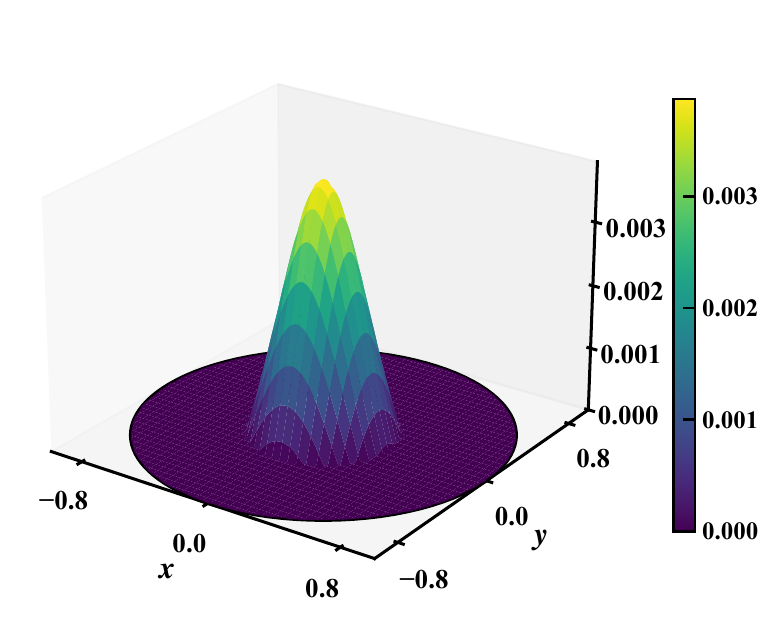}
		\caption{$(k^-,k^+)=(10^2,10^8)$.}
		\label{fig:numsol_1e4_exp5_heat}
	\end{subfigure}

	\par\medskip

	\begin{subfigure}[t]{0.45\textwidth}
	\includegraphics[width=\textwidth]{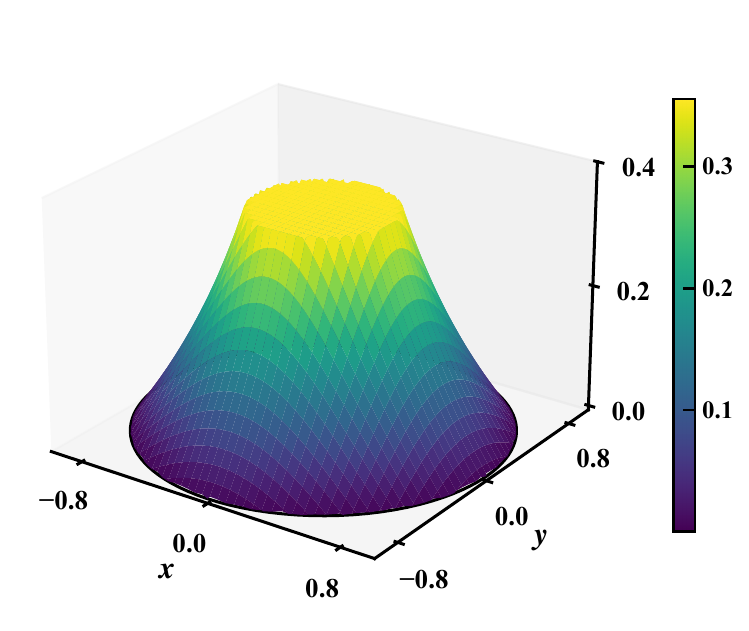}
		\caption{$(k^-,k^+)=(10^8,1)$.}
		\label{fig:numsol_1e8in_exp5_heat}
	\end{subfigure}
	\hfill
	\begin{subfigure}[t]{0.45\textwidth}
		\includegraphics[width=\textwidth]{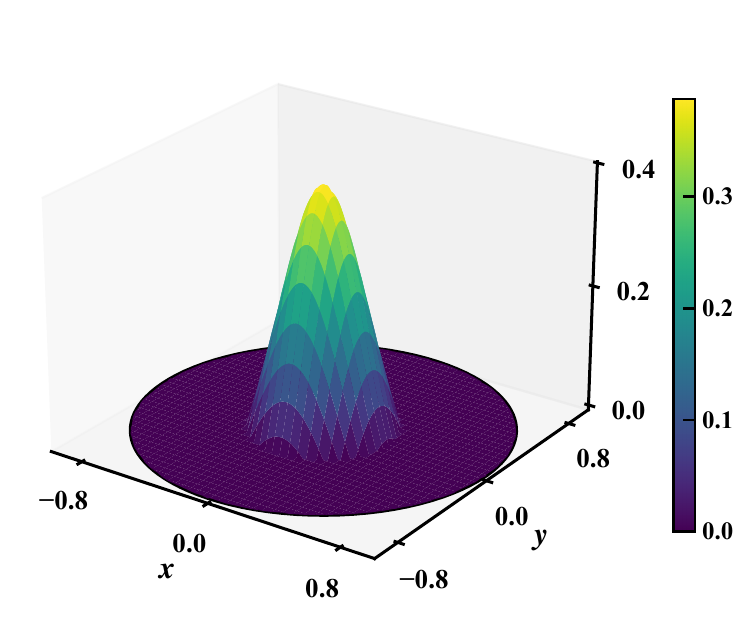}
		\caption{$(k^-,k^+)=(1,10^8)$.}
		\label{fig:numsol_1e8out_exp5_heat}
	\end{subfigure}
	\caption{Example~5. Corrected numerical solution $u_h$ for four representative coefficient combinations.}
	\label{fig:numerical_solutions_exp5_heat}
\end{figure}

\begin{figure}[t]
	\centering	\includegraphics[width=0.85\textwidth]{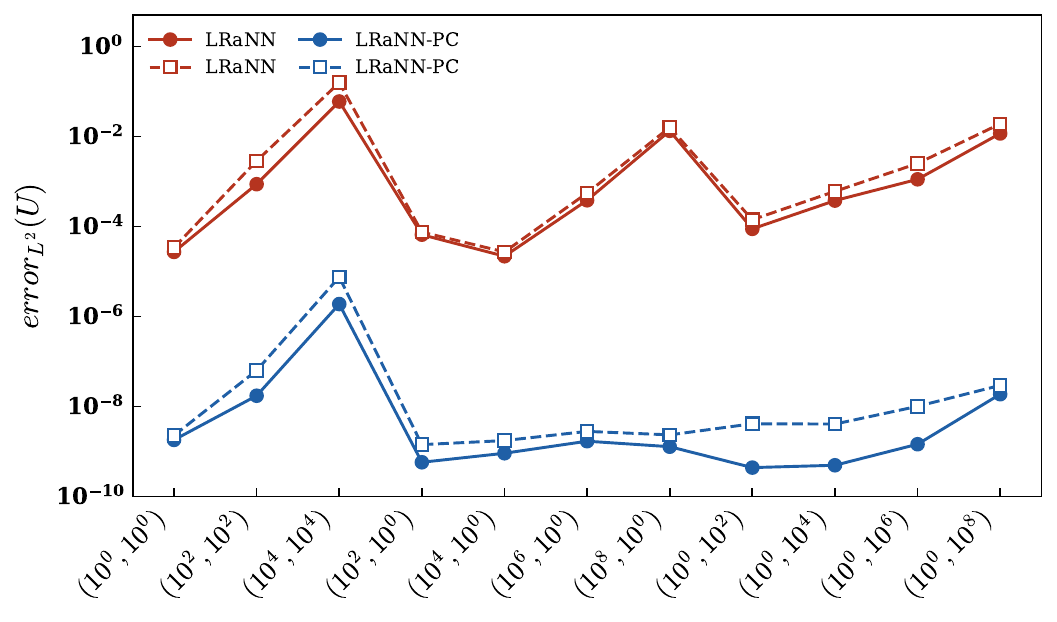}
	\caption{Example~5: $\erro_{L^2}(U)$ of the primary LRaNN stage and the corrected LRaNN-PC stage over the tested coefficient combinations $(k^-,k^+)$.}
	\label{fig:contrast_error_exp5_heat}
\end{figure}
\subsection{Benchmark Comparison with Reported IFE Results}
\label{subsec:comparison_ife}

To further assess the accuracy of the proposed perturbation-correction framework, we compare LRaNN-PC with the IFE results reported in~\cite{adjerid2025immersed}. We consider two benchmark problems from that work: one with a solution-dependent diffusion coefficient and the other with a gradient-dependent diffusion coefficient. The comparison below is based on the error values reported in~\cite{adjerid2025immersed} under the listed IFE settings; it is not intended as a degree-of-freedom or computational-cost comparison, since the IFE method and the present randomized collocation framework use different discretization parameters.

In both benchmarks, the computational domain is $\Omega=(0,1)^2$, and the circular interface
\[
\Gamma=\left\{(x,y)\ \middle|\ (x-0.719)^2+(y-0.713)^2=0.17^2\right\}
\]
divides the domain into
\[
\Omega^-=\left\{(x,y)\ \middle|\ (x-0.719)^2+(y-0.713)^2<0.17^2\right\},
\qquad
\Omega^+=\Omega\setminus\overline{\Omega^-}.
\]
The interface conditions are homogeneous,
\[
\llbracket u\rrbracket=0,\qquad
\llbracket \beta\nabla u\cdot \mathbf n\rrbracket=0,
\qquad x\in\Gamma,
\]
and the boundary data $g$ and source term $f$ are chosen from the exact solution
\begin{equation}
\label{eq:ife_exact_solution}
u(x,y)=
\begin{cases}
\dfrac{1}{10}\cos(x+3y)e^x\left((x-0.719)^2+(y-0.713)^2-0.17^2\right),
& (x,y)\in\Omega^-,
\\[2mm]
\cos(x+3y)e^x\left((x-0.719)^2+(y-0.713)^2-0.17^2\right),
& (x,y)\in\Omega^+.
\end{cases}
\end{equation}
Since the present paper uses relative error indicators, we convert the published absolute errors to $\erro_{L^2}(U)$ and $\erro_{H^1}(U)$ according to \eqref{eq:relative_l2_linf_errors} and \eqref{eq:relative_h1_error}. For the proposed method, $U=u_N$ in the primary LRaNN stage and $U=u_h$ in the LRaNN-PC stage. The settings $p=4$, $N=40$ and $p=2$, $N=40$ are the finest reported in~\cite{adjerid2025immersed} for these two benchmarks; we therefore compare against the highest-accuracy IFE results available in that reference.

\paragraph{Solution-dependent diffusion}
The first benchmark is the quasi-linear elliptic interface problem
\begin{equation}
\label{eq:ife_solution_dependent_problem}
\begin{cases}
-\nabla\cdot\left(\beta(u)\nabla u\right)=f,
& x\in\Omega\setminus\Gamma,
\\
u=g,
& x\in\partial\Omega,
\end{cases}
\end{equation}
where
\begin{equation}
\label{eq:ife_solution_dependent_beta}
\beta(u)=
\begin{cases}
10\left(5+(1+100u^2)^{-1/2}\right),
& x\in\Omega^-,
\\[1mm]
5+(1+u^2)^{-1/2},
& x\in\Omega^+.
\end{cases}
\end{equation}
This problem corresponds to Example~4.6 in~\cite{adjerid2025immersed}. The reported IFE result with $p=4$ and $N=40$ gives $\erro_{L^2}(U)=8.3655\times10^{-10}$ and $\erro_{H^1}(U)=6.8076\times10^{-8}$ after normalization by the exact solution. For the same benchmark, the primary LRaNN approximation gives $\erro_{L^2}(u_N)=1.1286\times10^{-5}$ and $\erro_{H^1}(u_N)=4.6791\times10^{-5}$. After the perturbation-correction stage, the errors decrease to $\erro_{L^2}(u_h)=6.4172\times10^{-12}$ and $\erro_{H^1}(u_h)=8.6701\times10^{-12}$. Under this reported IFE setting, the corrected LRaNN-PC errors are smaller than the converted IFE error values.

\paragraph{Gradient-dependent diffusion}
The second benchmark is a quasi-linear elliptic interface problem with a gradient-dependent diffusion coefficient,
\begin{equation}
\label{eq:ife_gradient_dependent_problem}
\begin{cases}
-\nabla\cdot\left(\beta(u_x,u_y)\nabla u\right)=f,
& x\in\Omega\setminus\Gamma,
\\
u=g,
& x\in\partial\Omega,
\end{cases}
\end{equation}
where
\begin{equation}
\label{eq:ife_gradient_dependent_beta}
\beta(u_x,u_y)=
\begin{cases}
10\left(20+(1+100u_x^2+100u_y^2)^{-1/2}\right),
& x\in\Omega^-,
\\[1mm]
20+(1+u_x^2+u_y^2)^{-1/2},
& x\in\Omega^+.
\end{cases}
\end{equation}
The same exact solution~\eqref{eq:ife_exact_solution} is used. This problem corresponds to Example~4.7 in~\cite{adjerid2025immersed} and is used here as a numerical extension beyond the solution-dependent diffusion model. The reported IFE result with $p=2$ and $N=40$ gives $\erro_{L^2}(U)=1.3737\times10^{-5}$ and $\erro_{H^1}(U)=6.7037\times10^{-4}$ after normalization by the exact solution. For the same benchmark, the primary LRaNN approximation gives $\erro_{L^2}(u_N)=4.7279\times10^{-5}$ and $\erro_{H^1}(u_N)=2.0824\times10^{-4}$. After the perturbation-correction stage, the errors decrease to $\erro_{L^2}(u_h)=9.5628\times10^{-12}$ and $\erro_{H^1}(u_h)=8.2566\times10^{-12}$. Under this reported IFE setting, the corrected LRaNN-PC errors are smaller than the converted IFE error values.

\begin{table}[h]
\centering
\caption{
Benchmark comparison with the IFE results reported in~\cite{adjerid2025immersed}. 
The published absolute IFE errors are converted to the relative indicators $\erro_{L^2}(U)$ and $\erro_{H^1}(U)$ used in this paper.
}
\label{tab:comparison_ife}
\small
\begin{tabular}{llccc}
\toprule
Benchmark & Method & Setting/Stage & $\erro_{L^2}(U)$ & $\erro_{H^1}(U)$ \\
\midrule
Solution-dependent diffusion
& IFE~\cite{adjerid2025immersed} & $p=4,\ N=40$
& $8.3655\times10^{-10}$ & $6.8076\times10^{-8}$ \\
Solution-dependent diffusion
& LRaNN & Stage 1
& $1.1286\times10^{-5}$ & $4.6791\times10^{-5}$ \\
Solution-dependent diffusion
& LRaNN-PC & Stage 2
& $6.4172\times10^{-12}$ & $8.6701\times10^{-12}$ \\
\midrule
Gradient-dependent diffusion
& IFE~\cite{adjerid2025immersed} & $p=2,\ N=40$
& $1.3737\times10^{-5}$ & $6.7037\times10^{-4}$ \\
Gradient-dependent diffusion
& LRaNN & Stage 1
& $4.7279\times10^{-5}$ & $2.0824\times10^{-4}$ \\
Gradient-dependent diffusion
& LRaNN-PC & Stage 2
& $9.5628\times10^{-12}$ & $8.2566\times10^{-12}$ \\
\bottomrule
\end{tabular}
\end{table}

The comparison in Table~\ref{tab:comparison_ife} should be interpreted only as an accuracy comparison with the reported IFE data under the listed settings. It is not a degree-of-freedom, runtime, or computational-cost comparison. The primary LRaNN approximation is not uniformly more accurate than the reported IFE results; for example, in the gradient-dependent benchmark, its $\erro_{L^2}(u_N)$ is larger than the corresponding reported IFE value. After the perturbation-correction stage, however, the corrected errors are reduced to the $10^{-11}$--$10^{-12}$ level in both benchmarks. These results show that the perturbation-correction stage remains effective on the two published benchmark configurations considered here, including the gradient-dependent case used as a numerical extension.
\subsection{Additional extensions}
\label{subsec:extensions}

We close this section with two additional tests that illustrate the applicability of the same perturbation-correction method beyond the steady solution-dependent elliptic setting considered in the main numerical examples. The first test is a moving-interface problem formulated in space--time, and the second test involves a gradient-dependent diffusion coefficient.

\paragraph{Moving-interface problem.}
We consider a quasi-linear parabolic interface problem with a prescribed time-dependent interface $\Gamma(t)$. The two subdomains have different solution-dependent diffusion laws,
\[
\beta^+(u)=1+u^2,\qquad \beta^-(u)=e^u+1,
\]
and the governing equations are
\[
\begin{cases}
u^+_t-\nabla\cdot\left(\beta^+(u^+)\nabla u^+\right)=f^+(x,u^+),
& (x,t)\in\Omega^+(t)\times[0,T],\\[1mm]
u^-_t-\nabla\cdot\left(\beta^-(u^-)\nabla u^-\right)=f^-(x,u^-),
& (x,t)\in\Omega^-(t)\times[0,T],\\[1mm]
\llbracket u\rrbracket=w, & x\in\Gamma(t),\\[1mm]
\llbracket \beta(u)\partial_{\mathbf n}u\rrbracket=v, & x\in\Gamma(t),\\[1mm]
u=g, & (x,t)\in\partial\Omega\times[0,T],\\[1mm]
u^\pm(x,0)=u_0^\pm(x), & x\in\Omega^\pm(0).
\end{cases}
\]
The spatial domain is $\Omega=[-1,1]\times[-1,1]$, and the moving interface is prescribed by
\[
\Gamma(t)=\left\{(x,y)\mid \sqrt{x^2+y^2}=r(t)\right\},
\qquad r(t)=0.5t+0.5.
\]
The source terms, boundary data, and interface data are generated from the exact solution
\[
u(x,y,t)=
\begin{cases}
e^{-t}\sin(\pi x)\sin(\pi y), & (x,y,t)\in\Omega^+(t)\times[0,T],\\[1mm]
t(x^2+y^2), & (x,y,t)\in\Omega^-(t)\times[0,T].
\end{cases}
\]
The initial data $u_0^\pm$ are obtained by evaluating this exact solution at $t=0$ and are imposed through an initial-time residual in the space--time least-squares functional. 

We solve this problem on $t\in[0,0.2]$ using a space--time LRaNN representation. The primary approximation uses $\tanh$ features in $(x,y,t)$, and the perturbation-correction stage uses an independent sinusoidal space--time correction basis. Fig~\ref{fig:stage1_comparison_exp5} reports the results at $t=0.05$, $0.10$, and $0.20$. The first row shows the primary error $e_N$ together with the scaled correction $\epsilon u_p$ at the three time levels. The second row shows the corrected error $e_h$. In all three snapshots, the corrected error is reduced to the $10^{-6}$--$10^{-5}$ level.

\begin{figure}[t]
	\centering
	\begin{subfigure}[t]{0.32\textwidth}
		\includegraphics[width=\textwidth]{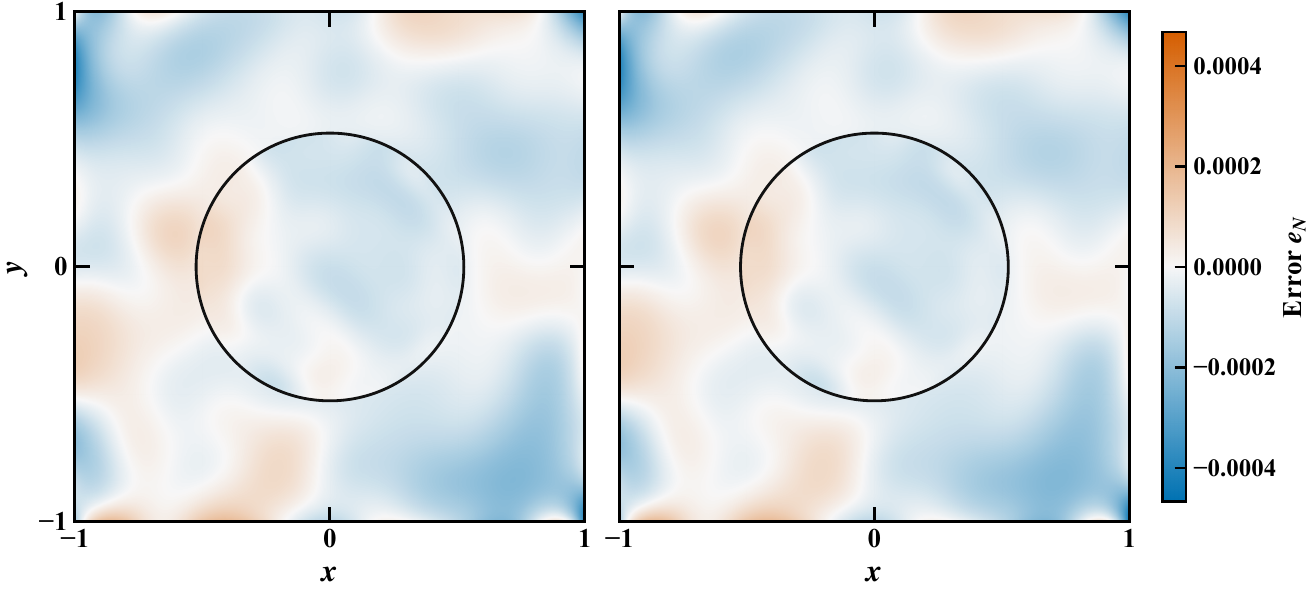}
		\caption{$t=0.05$: Left: primary error $e_N$. Right: correction $\epsilon u_p$.}
	\end{subfigure}
	\hfill
	\begin{subfigure}[t]{0.32\textwidth}
		\includegraphics[width=\textwidth]{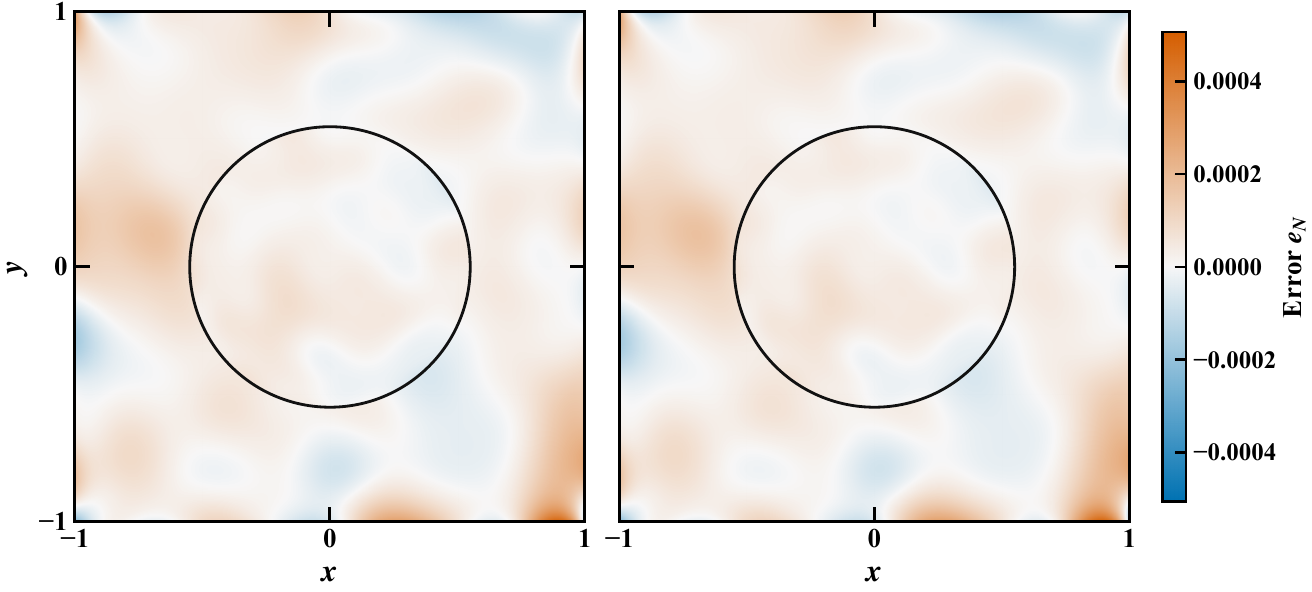} 
		\caption{$t=0.10$: Left: primary error $e_N$. Right: correction $\epsilon u_p$.}
	\end{subfigure}
	\hfill
	\begin{subfigure}[t]{0.32\textwidth}
		\includegraphics[width=\textwidth]{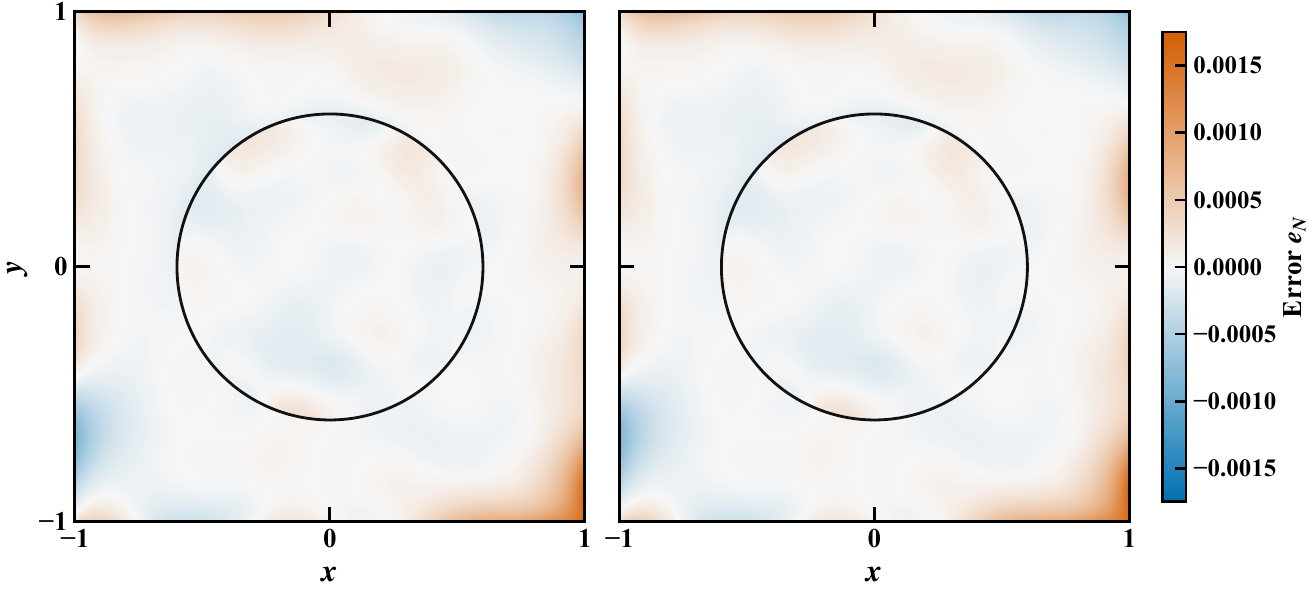} 
		\caption{$t=0.20$: Left: primary error $e_N$. Right: correction $\epsilon u_p$.}
	\end{subfigure}

	\par\medskip

	\begin{subfigure}[t]{0.32\textwidth}
		\includegraphics[width=\textwidth]{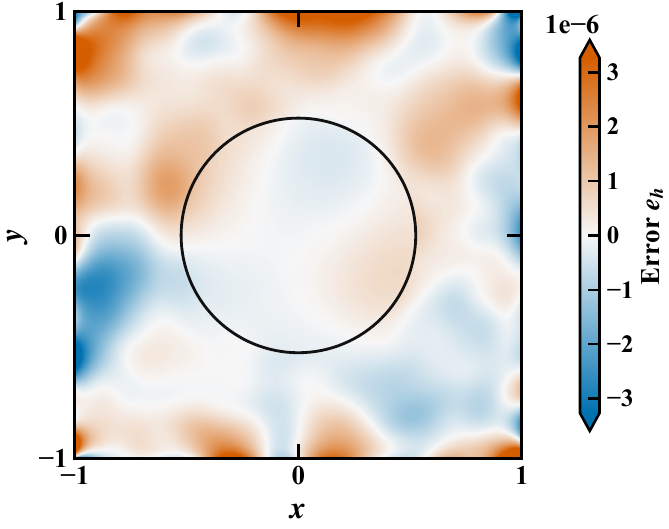} 
		\caption{$t=0.05$: pointwise error $e_h$.}
	\end{subfigure}
	\hfill
	\begin{subfigure}[t]{0.32\textwidth}
		\includegraphics[width=\textwidth]{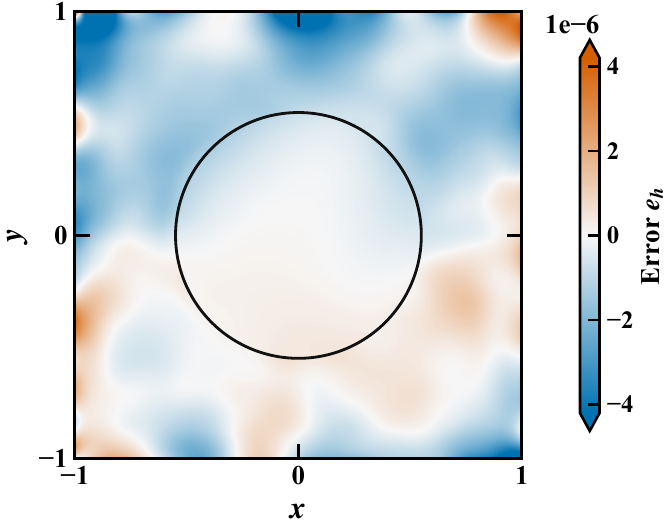} 
		\caption{$t=0.10$: pointwise error $e_h$.}
	\end{subfigure}
	\hfill
	\begin{subfigure}[t]{0.32\textwidth}
		\includegraphics[width=\textwidth]{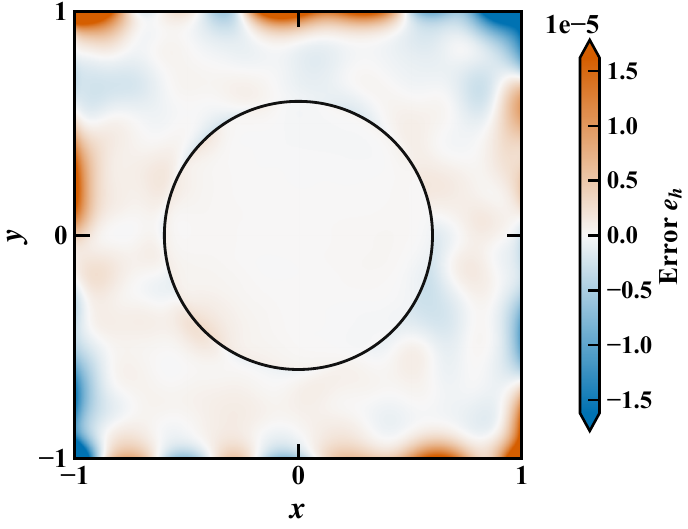} 
		\caption{$t=0.20$: pointwise error $e_h$.}
	\end{subfigure}

	\caption{Moving-interface extension. Top row: primary error $e_N$ and scaled correction $\epsilon u_p$ at three time levels. Bottom row: corrected error $e_h$.}
	\label{fig:stage1_comparison_exp5}
\end{figure}

\paragraph{Gradient-dependent diffusion}
As a second extension, we consider a quasi-linear elliptic interface problem in which the diffusion coefficient depends on the gradient magnitude in one subdomain and on the solution value in the other. The computational domain is $\Omega=[-1,1]\times[-1,1]$, and the circular interface
\[
\Gamma=\{(x,y)\mid x^2+y^2=1/4\}
\]
separates the inner subdomain $\Omega^-=\{x^2+y^2<1/4\}$ from the outer subdomain $\Omega^+=\Omega\setminus\overline{\Omega^-}$. The diffusion coefficients are
\[
\beta^+=1+|\nabla u|^2,\qquad \beta^-=1+u^2.
\]
The source terms and interface data are chosen so that the exact solution is
\[
u(x,y)=
\begin{cases}
0.5\sin(\pi x)\sin(\pi y)+0.25, & (x,y)\in\Omega^+,\\[1mm]
0.25-(x^2+y^2), & (x,y)\in\Omega^-.
\end{cases}
\]
The same two-stage LRaNN-PC procedure is applied. Fig~\ref{fig:stage1_comparison_exp6} shows the primary error $e_N$ and the scaled correction $\epsilon u_p$ as side-by-side surfaces in panel~(a), together with the corrected error field, the convergence history, and the interface trace. The primary maximum pointwise error is about $1.5\times10^{-4}$, while the corrected maximum pointwise error is about $6\times10^{-9}$. The residual norm and $\erro_{L^2}(U)$ also decrease further after the perturbation-correction stage. Along $\Gamma$, the primary error and the scaled correction overlap, and the interface $E_{L^\infty}$ error decreases from about $6\times10^{-4}$ to the $10^{-8}$ level.

\begin{figure}[t]
	\centering
	\begin{subfigure}[t]{0.65\textwidth}
		\centering
		\includegraphics[width=\textwidth]{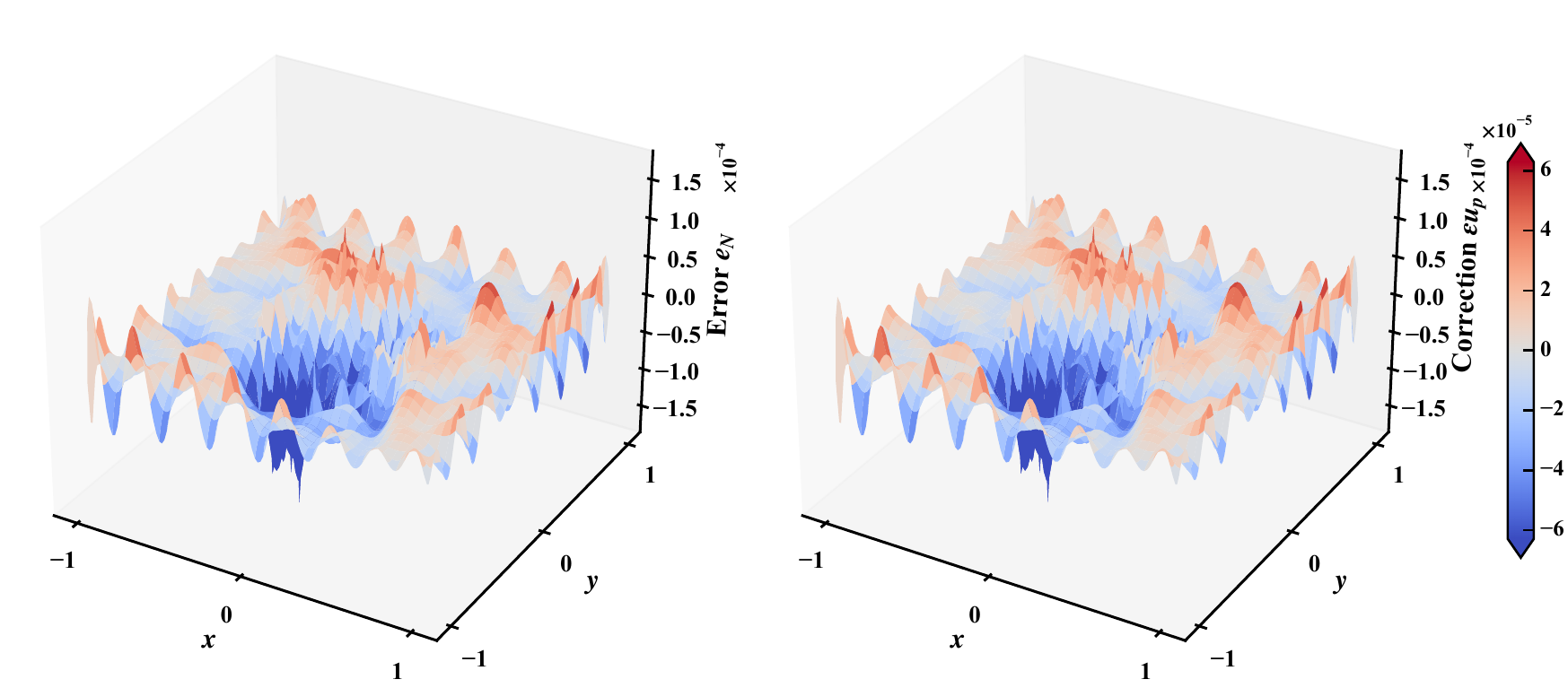}
		\caption{Left: primary error $e_N$. Right: correction $\epsilon u_p$.}
		\label{fig:error_comparison_exp6}
	\end{subfigure}\hfill
	\begin{subfigure}[t]{0.33\textwidth}
		\centering
		\includegraphics[width=\textwidth]{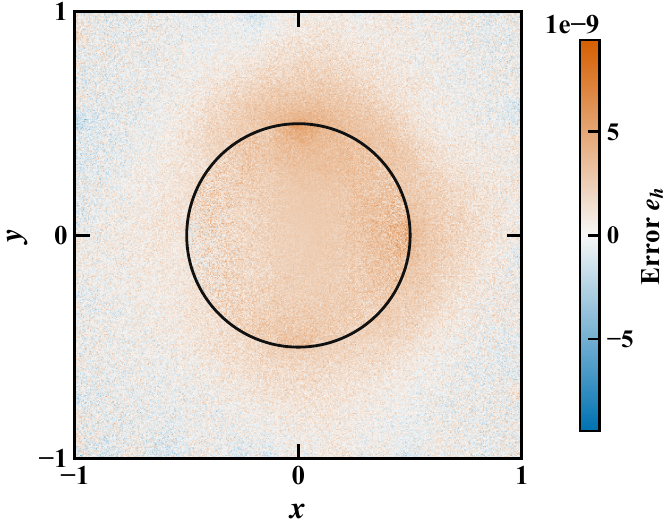} 
		\caption{Pointwise error $e_h$.}
		\label{fig:perturbation_solution_exp6}
	\end{subfigure}

	\par\medskip

	\begin{subfigure}[t]{0.52\textwidth}
		\centering
		\includegraphics[width=\textwidth]{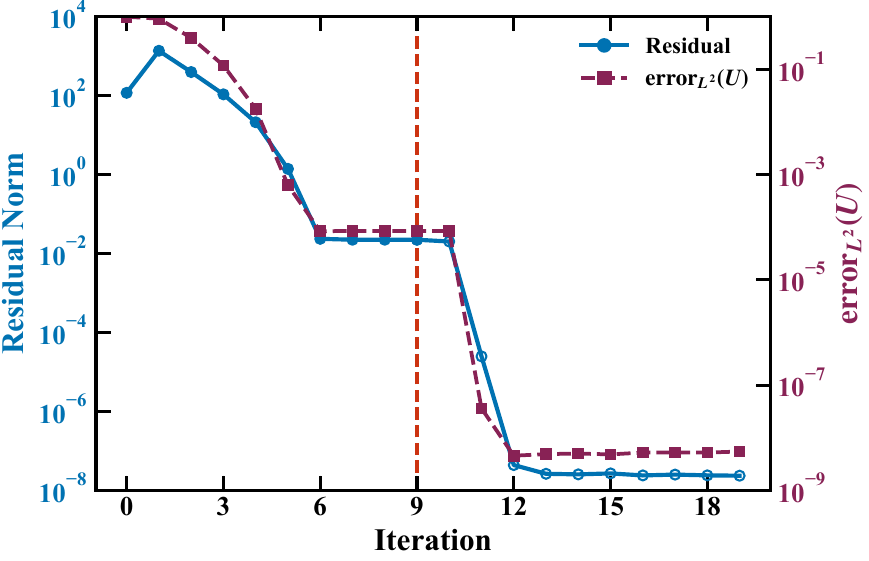} 
		\caption{Residual norm and $\erro_{L^2}(U)$ versus iteration.}
		\label{fig:residual_vs_iteration_exp6}
	\end{subfigure}\hfill
	\begin{subfigure}[t]{0.46\textwidth}
		\centering
		\includegraphics[width=\textwidth]{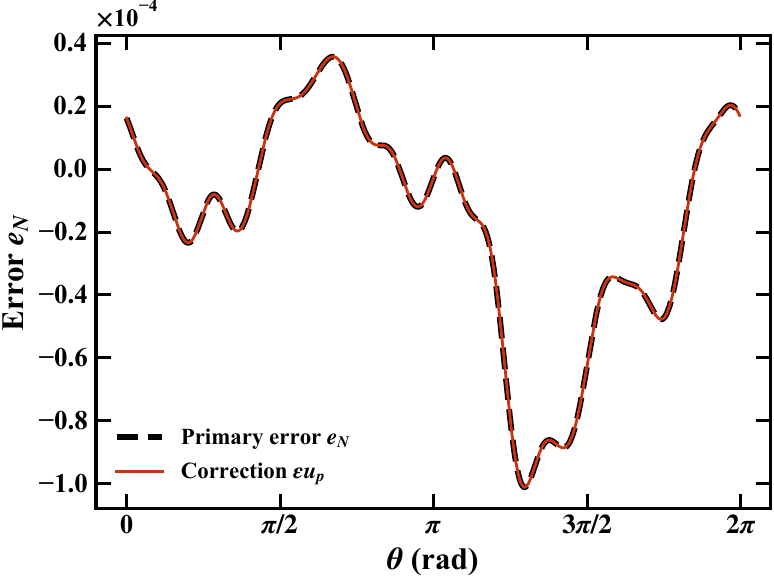} 
		\caption{Primary error $e_N$ and $\epsilon u_p$ along $\Gamma$.}
		\label{fig:interface_error_concise_exp6}
	\end{subfigure}
	\caption{Gradient-dependent diffusion extension.}
	\label{fig:stage1_comparison_exp6}
\end{figure}

\section{Conclusion}
\label{sec:conclusion}
In this work, we have proposed a two-stage perturbation-correction framework based on LRaNNs for solution-dependent quasi-linear elliptic interface problems, and have numerically examined its extension to gradient-dependent and time-dependent interface settings. The primary LRaNN stage uses Gauss--Newton iteration to compute a primary approximation that captures the main features on the interface and within each subdomain. The perturbation-correction stage then computes a residual-driven correction on a new randomized basis to reduce the residual error left by the primary approximation. By explicitly targeting leading residual modes, the method reduces the residual left by the primary approximation in the tested numerical experiments.

The numerical tests cover static and moving interfaces, straight, multiple, and curved interface geometries, high-contrast coefficients, and nonlinear diffusion depending on $u$, on $|\nabla u|$, and on both. They include a physically motivated problem that does not admit a closed-form solution, assessed against a high-accuracy numerical reference, and a benchmark comparison with published IFE results. In these tests, the primary LRaNN stage may stagnate at a nonzero residual level, whereas the perturbation-correction stage reduces the relative $L^2$ error by about four to seven orders of magnitude in the tested examples relative to the primary approximation, consistent with the residual-reduction mechanism discussed in Sec~\ref{sec:error analysis}. The local randomized representation imposes the nonlinear interface jump conditions through the least-squares objective and does not require an interface-fitted mesh.

The residual-controlled analysis applies to the model problem considered in this paper under the stated sufficient assumptions. At the algorithmic level, the correction step can also be formulated for broader nonlinear interface problems when the residual operator is well defined and locally expandable. Solutions with strong local singularities are limited by the approximation error rather than the optimization error and are beyond the scope of the present study (Remark~\ref{rmk:error-sources}). In future work, we plan to extend the perturbation-correction strategy to further nonlinear PDE problems with nonconvex landscapes, and to enrich the trial space for low-regularity solutions.

\appendix
\section{Sensitivity to the random initialization and the algorithmic parameters}
\label{app:sensitivity}

The studies in this appendix use the configuration of Example~1: a primary $\tanh$ network with $m=100$ neurons per subdomain, with hidden weights and biases sampled from $\mathcal N(0,1)$ and $\mathcal N(0,0.1^2)$, and an independent $\sin$ correction network. In each study one parameter group is varied while the others are kept at their default values. We report the primary relative errors $\erro_{L^2}(u_N)$ and $\erro_{L^\infty}(u_N)$ and the corrected relative errors $\erro_{L^2}(u_h)$ and $\erro_{L^\infty}(u_h)$.

\subsection{Initialization of the primary network}
\label{app:init-primary}

Table~\ref{tab:app-primary-init} varies the sampling of the primary hidden weights and biases. The scale of the hidden weights is the dominant factor for the primary stage: $\erro_{L^2}(u_N)$ is smallest near unit weight scale and grows as the scale is reduced or increased, rising to $8.8\times10^{-3}$ when the weight standard deviation is set to $4$. The bias scale and the choice between Gaussian and uniform sampling have a comparatively small effect. After the correction stage, $\erro_{L^2}(u_h)$ returns to the $10^{-12}$ level in every case except the most severely over-scaled one ($\mathbf w_j\sim\mathcal N(0,4^2)$), where the primary approximation is too inaccurate for the correction to recover full precision.

\begin{table}[ht]
	\centering
	\caption{Effect of the primary-network initialization (Example~1). The hidden weights $\mathbf w_j$ and biases $b_j$ of the primary $\tanh$ network are sampled as listed; all other settings are at their default values. The first row is the default configuration.}
	\label{tab:app-primary-init}
	\begin{tabular}{cccccc}
		\toprule
		$\mathbf w_j$ & $b_j$ & $\erro_{L^2}(u_N)$ & $\erro_{L^\infty}(u_N)$ & $\erro_{L^2}(u_h)$ & $\erro_{L^\infty}(u_h)$ \\
		\midrule
		$\mathcal N(0,1)$    & $\mathcal N(0,0.1^2)$ & $2.9301\times 10^{-7}$ & $3.7504\times 10^{-7}$ & $1.5408\times 10^{-12}$ & $1.8248\times 10^{-12}$ \\
		$\mathcal N(0,0.5^2)$ & $\mathcal N(0,0.1^2)$ & $4.0248\times 10^{-6}$ & $8.2858\times 10^{-6}$ & $6.4062\times 10^{-10}$ & $7.7871\times 10^{-10}$ \\
		$\mathcal N(0,2^2)$   & $\mathcal N(0,0.1^2)$ & $7.5342\times 10^{-5}$ & $1.5549\times 10^{-4}$ & $1.4118\times 10^{-9}$  & $2.8356\times 10^{-9}$  \\
		$\mathcal N(0,4^2)$   & $\mathcal N(0,0.1^2)$ & $8.8209\times 10^{-3}$ & $1.8976\times 10^{-2}$ & $4.5236\times 10^{-6}$  & $5.9801\times 10^{-6}$  \\
		$\mathcal N(0,1)$    & $\mathcal N(0,0.5^2)$ & $7.9479\times 10^{-8}$ & $1.5301\times 10^{-7}$ & $1.5362\times 10^{-12}$ & $1.3835\times 10^{-12}$ \\
		$\mathcal N(0,1)$    & $\mathcal N(0,1)$     & $8.1993\times 10^{-8}$ & $1.1305\times 10^{-7}$ & $1.2273\times 10^{-11}$ & $1.2556\times 10^{-11}$ \\
		$U(-1,1)$            & $U(-0.1,0.1)$         & $8.8145\times 10^{-7}$ & $1.4652\times 10^{-6}$ & $7.2706\times 10^{-11}$ & $9.3853\times 10^{-11}$ \\
		$U(-2,2)$            & $U(-0.1,0.1)$         & $8.9742\times 10^{-6}$ & $2.0117\times 10^{-5}$ & $1.2387\times 10^{-12}$ & $1.9669\times 10^{-12}$ \\
		\bottomrule
	\end{tabular}
\end{table}

\subsection{Initialization of the correction network}
\label{app:init-correction}

Table~\ref{tab:app-correction-init} varies the weight magnitude of the $\sin$ correction network, written as $\mathbf w_j\sim\mathcal N(0,(c\pi)^2)$. This magnitude sets the frequency content of the correction basis. At $c=1$ the basis cannot fully represent the high-frequency primary residual and $\erro_{L^2}(u_h)$ stops at $2.9\times10^{-10}$; for $c\ge2.5$ the corrected error reaches the $10^{-12}$ level and changes little with further increase, while the effective rank of the correction Jacobian grows with $c$. The default $c=5$ lies in this saturated regime.

\begin{table}[ht]
	\centering
	\caption{Effect of the correction-network initialization (Example~1). The hidden weights of the $\sin$ correction network are sampled as $\mathbf w_j\sim\mathcal N(0,(c\pi)^2)$ with the listed $c$, and the biases as $b_j\sim\mathcal N(0,1)$. The primary stage is identical for all rows, with $\erro_{L^2}(u_N)=2.9301\times10^{-7}$. Here $r$ is the effective rank of the correction Jacobian and the last column the number of correction iterations. The default is $c=5$.}
	\label{tab:app-correction-init}
	\begin{tabular}{ccccc}
		\toprule
		$c$ & $\erro_{L^2}(u_h)$ & $\erro_{L^\infty}(u_h)$ & $r$ & Iter. \\
		\midrule
		$1$   & $2.8723\times 10^{-10}$ & $4.4754\times 10^{-10}$ & $323$  & $5$  \\
		$2.5$ & $1.5592\times 10^{-12}$ & $1.9343\times 10^{-12}$ & $620$  & $10$ \\
		$5$   & $1.5408\times 10^{-12}$ & $1.8248\times 10^{-12}$ & $1033$ & $4$  \\
		$7.5$ & $1.5408\times 10^{-12}$ & $1.8288\times 10^{-12}$ & $1346$ & $4$  \\
		\bottomrule
	\end{tabular}
\end{table}

\subsection{Penalty weights}
\label{app:penalty}

Table~\ref{tab:app-penalty} varies the boundary and interface penalty weights relative to the interior weights $\omega_{\Omega^\pm}=1$. The weights are chosen so that the boundary and interface residual blocks are comparable in magnitude to the interior residual. The primary stage is most accurate for $\omega\in[10^2,10^3]$ and degrades for much smaller or much larger weights, whereas the corrected error stays at the $10^{-12}$ level across the whole range, including the asymmetric choices. We use $\omega_{\partial\Omega}=\omega_\Gamma=10^2$ as the default.

\begin{table}[ht]
	\centering
	\caption{Effect of the penalty weights (Example~1). The boundary weight $\omega_{\partial\Omega}$ and the interface weight $\omega_\Gamma:=\omega_{\Gamma_d}=\omega_{\Gamma_n}$ are varied; the interior weights are fixed at $\omega_{\Omega^\pm}=1$. The default is $\omega_{\partial\Omega}=\omega_\Gamma=10^2$.}
	\label{tab:app-penalty}
	\begin{tabular}{cccccc}
		\toprule
		$\omega_{\partial\Omega}$ & $\omega_\Gamma$ & $\erro_{L^2}(u_N)$ & $\erro_{L^\infty}(u_N)$ & $\erro_{L^2}(u_h)$ & $\erro_{L^\infty}(u_h)$ \\
		\midrule
		$10^0$ & $10^0$ & $1.3416\times 10^{-5}$ & $1.9402\times 10^{-5}$ & $1.1711\times 10^{-12}$ & $1.0843\times 10^{-12}$ \\
		$10^1$ & $10^1$ & $1.1519\times 10^{-6}$ & $1.8095\times 10^{-6}$ & $1.1471\times 10^{-12}$ & $9.7022\times 10^{-13}$ \\
		$10^2$ & $10^2$ & $2.9301\times 10^{-7}$ & $3.7504\times 10^{-7}$ & $1.5408\times 10^{-12}$ & $1.8248\times 10^{-12}$ \\
		$10^3$ & $10^3$ & $2.3874\times 10^{-7}$ & $1.6865\times 10^{-7}$ & $6.0417\times 10^{-12}$ & $7.5002\times 10^{-12}$ \\
		$10^4$ & $10^4$ & $1.6461\times 10^{-6}$ & $8.9142\times 10^{-7}$ & $5.6941\times 10^{-12}$ & $6.7352\times 10^{-12}$ \\
		$10^3$ & $10^2$ & $3.4628\times 10^{-7}$ & $4.2042\times 10^{-7}$ & $5.8845\times 10^{-12}$ & $6.8798\times 10^{-12}$ \\
		$10^2$ & $10^3$ & $3.0067\times 10^{-7}$ & $4.5064\times 10^{-7}$ & $1.4259\times 10^{-12}$ & $1.4442\times 10^{-12}$ \\
		\bottomrule
	\end{tabular}
\end{table}

\subsection{Gauss--Newton convergence tolerance}
\label{app:gn-tol}

Table~\ref{tab:app-gn-tol} varies the Gauss--Newton convergence tolerance $\delta_0$ in the stopping criterion~\eqref{eq:Netwon_converge}. Reducing $\delta_0$ below $10^{-4}$ does not change the corrected error, but it increases the number of correction iterations from $4$ to $10$. A tolerance of $10^{-4}$ is therefore sufficient.

\begin{table}[ht]
	\centering
	\caption{Effect of the Gauss--Newton convergence tolerance (Example~1). The iteration is stopped when the relative change in the residual norm falls below $\delta_0$, as in~\eqref{eq:Netwon_converge}. The primary stage is identical for all rows. The last column is the number of correction iterations. The default is $\delta_0=10^{-4}$.}
	\label{tab:app-gn-tol}
	\begin{tabular}{cccc}
		\toprule
		$\delta_0$ & $\erro_{L^2}(u_h)$ & $\erro_{L^\infty}(u_h).$ \\
		\midrule
		$10^{-2}$ & $1.5408\times 10^{-12}$ & $1.8248\times 10^{-12}$  \\
		$10^{-3}$ & $1.5408\times 10^{-12}$ & $1.8248\times 10^{-12}$  \\
		$10^{-4}$ & $1.5408\times 10^{-12}$ & $1.8248\times 10^{-12}$ \\
		$10^{-6}$ & $1.5408\times 10^{-12}$ & $1.8246\times 10^{-12}$  \\
		$10^{-8}$ & $1.5408\times 10^{-12}$ & $1.8246\times 10^{-12}$ \\
		\bottomrule
	\end{tabular}
\end{table}

\bibliographystyle{siamplain}
\bibliography{reference}

@book{brenner2013interfacial,
  title={Interfacial transport processes and rheology},
  author={Brenner, Howard},
  year={2013},
  publisher={Elsevier}
}

@Article{material_mechanics,
    title = {A \mbox{P}etrov-\mbox{G}alerkin finite element interface method for interface problems with \mbox{B}loch-periodic boundary conditions and its application in phononic crystals},
    journal = {Journal of Computational Physics},
    volume = {393},
    pages = {117-138},
    year = {2019},
    issn = {0021-9991},
    author = {Liqun Wang and Hui Zheng and Xin Lu and Liwei Shi},
}

@article{kirk2023nonlinear,
  title={Nonlinear electrochemical impedance spectroscopy for lithium-ion battery model parameterization},
  author={Kirk, Toby L and LewisDouglas, Adam and Howey, David and Please, Colin P and Chapman, S Jon},
  journal={Journal of The Electrochemical Society},
  volume={170},
  number={1},
  pages={010514},
  year={2023},
  publisher={IOP Publishing}
}

@article{donato2015homogenization,
  title={Homogenization of diffusion problems with a nonlinear interfacial resistance},
  author={Donato, Patrizia and Le Nguyen, Kim Hang},
  journal={Nonlinear Differential Equations and Applications NoDEA},
  volume={22},
  number={5},
  pages={1345--1380},
  year={2015},
  publisher={Springer}
}

@article{zeng2013efficient,
  title={Efficient conservative numerical schemes for 1\mbox{D} nonlinear spherical diffusion equations with applications in battery modeling},
  author={Zeng, Yi and Albertus, Paul and Klein, Reinhardt and Chaturvedi, Nalin and Kojic, Aleksandar and Bazant, Martin Z and Christensen, Jake},
  journal={Journal of The Electrochemical Society},
  volume={160},
  number={9},
  pages={A1565},
  year={2013},
  publisher={IOP Publishing}
}

@article{teigen2009diffuse,
  title={A diffuse-interface approach for modeling transport, diffusion and adsorption/desorption of material quantities on a deformable interface},
  author={Teigen, Knut Erik and Li, Xiangrong and Lowengrub, John and Wang, Fan and Voigt, Axel},
  journal={Communications in Mathematical Sciences},
  volume={4},
  number={7},
  pages={1009},
  year={2009}
}

@article{IBM1977,
  title={Numerical analysis of blood flow in the heart},
  author={Peskin, Charles S},
  journal={Journal of Computational Physics},
  volume={25},
  number={3},
  pages={220--252},
  year={1977},
  publisher={Elsevier}
}

@article{gfm2003,
  title={Convergence of the ghost fluid method for elliptic equations with interfaces},
  author={Liu, Xu-Dong and Sideris, Thomas},
  journal={Mathematics of Computation},
  volume={72},
  number={244},
  pages={1731--1746},
  year={2003}
}

@article{MIB2014,
  title={\mbox{MIB} \mbox{G}alerkin method for elliptic interface problems},
  author={Xia, Kelin and Zhan, Meng and Wei, Guo-Wei},
  journal={Journal of Computational and Applied Mathematics},
  volume={272},
  pages={195--220},
  year={2014},
  publisher={Elsevier}
}

@article{babuvska1970finite,
  title={The finite element method for elliptic equations with discontinuous coefficients},
  author={Babu{\v{s}}ka, Ivo},
  journal={Computing},
  volume={5},
  number={3},
  pages={207--213},
  year={1970},
  publisher={Springer}
}

@article{kirchhart2016analysis,
  title={Analysis of an \mbox{XFEM} discretization for \mbox{S}tokes interface problems},
  author={Kirchhart, Matthias and Gross, Sven and Reusken, Arnold},
  journal={SIAM Journal on Scientific Computing},
  volume={38},
  number={2},
  pages={A1019--A1043},
  year={2016},
  publisher={SIAM}
}

@article{DG2012,
  title={An unfitted discontinuous \mbox{G}alerkin method applied to elliptic interface problems},
  author={Massjung, Ralf},
  journal={SIAM Journal on Numerical Analysis},
  volume={50},
  number={6},
  pages={3134--3162},
  year={2012},
  publisher={SIAM}
}

@article{lehrenfeld2013analysis,
  title={Analysis of a \mbox{N}itsche \mbox{XFEM-DG} discretization for a class of two-phase mass transport problems},
  author={Lehrenfeld, Christoph and Reusken, Arnold},
  journal={SIAM Journal on Numerical Analysis},
  volume={51},
  number={2},
  pages={958--983},
  year={2013},
  publisher={SIAM}
}

@article{wang2018convergence,
  title={Convergence of finite difference methods for the wave equation in two space dimensions},
  author={Wang, Siyang and Nissen, Anna and Kreiss, Gunilla},
  journal={Mathematics of Computation},
  volume={87},
  number={314},
  pages={2737--2763},
  year={2018}
}

@book{li2006immersed,
  title={The immersed interface method: numerical solutions of \mbox{PDE}s involving interfaces and irregular domains},
  author={Li, Zhilin and Ito, Kazufumi},
  year={2006},
  publisher={SIAM}
}

@article{leveque1994immersed,
  title={The immersed interface method for elliptic equations with discontinuous coefficients and singular sources},
  author={LeVeque, Randall J and Li, Zhilin},
  journal={SIAM Journal on Numerical Analysis},
  volume={31},
  number={4},
  pages={1019--1044},
  year={1994},
  publisher={SIAM}
}

@article{yu2007matched,
  title={Matched interface and boundary (\mbox{MIB}) method for elliptic problems with sharp-edged interfaces},
  author={Yu, Sining and Zhou, Yongcheng and Wei, Guo-Wei},
  journal={Journal of Computational Physics},
  volume={224},
  number={2},
  pages={729--756},
  year={2007},
  publisher={Elsevier}
}

@article{zhou2006high,
  title={High order matched interface and boundary method for elliptic equations with discontinuous coefficients and singular sources},
  author={Zhou, YC and Zhao, Shan and Feig, Michael and Wei, Guo-Wei},
  journal={Journal of Computational Physics},
  volume={213},
  number={1},
  pages={1--30},
  year={2006},
  publisher={Elsevier}
}

@article{fedkiw1999non,
  title={A non-oscillatory \mbox{E}ulerian approach to interfaces in multimaterial flows (the ghost fluid method)},
  author={Fedkiw, Ronald P and Aslam, Tariq and Merriman, Barry and Osher, Stanley},
  journal={Journal of Computational Physics},
  volume={152},
  number={2},
  pages={457--492},
  year={1999},
  publisher={Elsevier}
}

@article{liu2000boundary,
  title={A boundary condition capturing method for \mbox{P}oisson's equation on irregular domains},
  author={Liu, Xu-Dong and Fedkiw, Ronald P and Kang, Myungjoo},
  journal={Journal of Computational Physics},
  volume={160},
  number={1},
  pages={151--178},
  year={2000},
  publisher={Elsevier}
}

@article{raissi2019physics,
  title={Physics-informed neural networks: a deep learning framework for solving forward and inverse problems involving nonlinear partial differential equations},
  author={Raissi, Maziar and Perdikaris, Paris and Karniadakis, George E},
  journal={Journal of Computational Physics},
  volume={378},
  pages={686--707},
  year={2019},
  publisher={Elsevier}
}

@article{sirignano2018dgm,
  title={\mbox{DGM}: a deep learning algorithm for solving partial differential equations},
  author={Sirignano, Justin and Spiliopoulos, Konstantinos},
  journal={Journal of Computational Physics},
  volume={375},
  pages={1339--1364},
  year={2018},
  publisher={Elsevier}
}

@article{yu2018deep,
  title={The deep \mbox{R}itz method: a deep learning-based numerical algorithm for solving variational problems},
  author={Yu, Bing and others},
  journal={Communications in Mathematics and Statistics},
  volume={6},
  number={1},
  pages={1--12},
  year={2018},
  publisher={Springer}
}

@article{wang2024multi,
  title={Multi-stage neural networks: function approximator of machine precision},
  author={Wang, Yongji and Lai, Ching-Yao},
  journal={Journal of Computational Physics},
  volume={504},
  pages={112865},
  year={2024},
  publisher={Elsevier}
}

@article{aldirany2024multi,
  title={Multi-level neural networks for accurate solutions of boundary-value problems},
  author={Aldirany, Ziad and Cottereau, R{\'e}gis and Laforest, Marc and Prudhomme, Serge},
  journal={Computer Methods in Applied Mechanics and Engineering},
  volume={419},
  pages={116666},
  year={2024},
  publisher={Elsevier}
}

@article{huang2006extreme,
  title={Extreme learning machine: theory and applications},
  author={Huang, Guang-Bin and Zhu, Qin-Yu and Siew, Chee-Kheong},
  journal={Neurocomputing},
  volume={70},
  number={1-3},
  pages={489--501},
  year={2006},
  publisher={Elsevier}
}

@article{pao1994learning,
  title={Learning and generalization characteristics of the random vector functional-link net},
  author={Pao, Yoh-Han and Park, Gwang-Hoon and Sobajic, Dejan J},
  journal={Neurocomputing},
  volume={6},
  number={2},
  pages={163--180},
  year={1994},
  publisher={Elsevier}
}

@article{lin2014elm,
  title={Is extreme learning machine feasible? A theoretical assessment (\mbox{P}art \mbox{II})},
  author={Lin, Shaobo and Liu, Xia and Fang, Jian and Xu, Zongben},
  journal={IEEE Transactions on Neural Networks and Learning Systems},
  volume={26},
  number={1},
  pages={21--34},
  year={2014},
  publisher={IEEE}
}

@article{igelnik1995stochastic,
  title={Stochastic choice of basis functions in adaptive function approximation and the functional-link net},
  author={Igelnik, Boris and Pao, Yoh-Han},
  journal={IEEE Transactions on Neural Networks},
  volume={6},
  number={6},
  pages={1320--1329},
  year={1995},
  publisher={IEEE}
}

@article{zhang2025immersed,
  title={An immersed interface neural network for elliptic interface problems},
  author={Zhang, Xinru and He, Qiaolin},
  journal={Journal of Computational and Applied Mathematics},
  volume={459},
  pages={116372},
  year={2025},
  publisher={Elsevier}
}

@article{jagtap2020extended,
  title={Extended physics-informed neural networks (\mbox{XPINN}s): a generalized space-time domain decomposition based deep learning framework for nonlinear partial differential equations},
  author={Jagtap, Ameya D and Karniadakis, George Em},
  journal={Communications in Computational Physics},
  volume={28},
  number={5},
  year={2020},
  publisher={Brown Univ., Providence, RI (United States)}
}

@article{tseng2023cusp,
  title={A cusp-capturing \mbox{PINN} for elliptic interface problems},
  author={Tseng, Yu-Hau and Lin, Te-Sheng and Hu, Wei-Fan and Lai, Ming-Chih},
  journal={Journal of Computational Physics},
  volume={491},
  pages={112359},
  year={2023},
  publisher={Elsevier}
}

@article{zhu2023physics,
  title={Physics-informed neural networks for solving dynamic two-phase interface problems},
  author={Zhu, Xingwen and Hu, Xiaozhe and Sun, Pengtao},
  journal={SIAM Journal on Scientific Computing},
  volume={45},
  number={6},
  pages={A2912--A2944},
  year={2023},
  publisher={SIAM}
}

@article{roy2024adaptive,
  title={Adaptive interface-\mbox{PINN}s (\mbox{AdaI-PINN}s): an efficient physics-informed neural networks framework for interface problems},
  author={Roy, Sumanta and Annavarapu, Chandrasekhar and Roy, Pratanu and Sarma, Antareep Kumar},
  journal={arXiv preprint arXiv:2406.04626},
  year={2024}
}

@article{lai2025hard,
  title={The hard-constraint \mbox{PINN}s for interface optimal control problems},
  author={Lai, Ming-Chih and Song, Yongcun and Yuan, Xiaoming and Yue, Hangrui and Zeng, Tianyou},
  journal={SIAM Journal on Scientific Computing},
  volume={47},
  number={3},
  pages={C601--C629},
  year={2025},
  publisher={SIAM}
}

@article{sarma2024interface,
  title={Interface \mbox{PINN}s (\mbox{I-PINN}s): a physics-informed neural networks framework for interface problems},
  author={Sarma, Antareep Kumar and Roy, Sumanta and Annavarapu, Chandrasekhar and Roy, Pratanu and Jagannathan, Shriram},
  journal={Computer Methods in Applied Mechanics and Engineering},
  volume={429},
  pages={117135},
  year={2024},
  publisher={Elsevier}
}

@article{li2025physics,
  title={Physics-informed neural networks for enhanced interface preservation in lattice \mbox{B}oltzmann multiphase simulations},
  author={Li, Yue and Zhang, Lihong},
  journal={arXiv preprint arXiv:2504.10539},
  year={2025}
}

@book{showalter2013monotone,
  title={Monotone operators in \mbox{B}anach space and nonlinear partial differential equations},
  author={Showalter, Ralph Edwin},
  volume={49},
  year={2013},
  publisher={American Mathematical Soc.}
}

@book{zeidler2013nonlinear,
  title={Nonlinear functional analysis and its applications: \mbox{II/B}: nonlinear monotone operators},
  author={Zeidler, Eberhard},
  year={2013},
  publisher={Springer Science \& Business Media}
}

@article{lee2025nonoverlapping,
  title={A nonoverlapping domain decomposition method for extreme learning machines: elliptic problems},
  author={Lee, Chang-Ock and Lee, Youngkyu and Ryoo, Byungeun},
  journal={Computers \& Mathematics with Applications},
  volume={189},
  pages={109--128},
  year={2025},
  publisher={Elsevier}
}

@article{ren2025physics,
  title={Physics-informed extreme learning machine (\mbox{PIELM}) for \mbox{S}tefan problems},
  author={Ren, Fei and Zhuang, Pei-Zhi and Chen, Xiaohui and Yu, Hai-Sui and Yang, He},
  journal={Computer Methods in Applied Mechanics and Engineering},
  volume={441},
  pages={118015},
  year={2025},
  publisher={Elsevier}
}

@article{dong2021local,
  title={Local extreme learning machines and domain decomposition for solving linear and nonlinear partial differential equations},
  author={Dong, Suchuan and Li, Zongwei},
  journal={Computer Methods in Applied Mechanics and Engineering},
  volume={387},
  pages={114129},
  year={2021},
  publisher={Elsevier}
}

@article{lu2021deepxde,
  title={Deep\mbox{XDE}: a deep learning library for solving differential equations},
  author={Lu, Lu and Meng, Xuhui and Mao, Zhiping and Karniadakis, George Em},
  journal={SIAM Review},
  volume={63},
  number={1},
  pages={208--228},
  year={2021},
  publisher={SIAM}
}

@article{kharazmi2021hp,
  title={hp-\mbox{VPINN}s: variational physics-informed neural networks with domain decomposition},
  author={Kharazmi, Ehsan and Zhang, Zhongqiang and Karniadakis, George Em},
  journal={Computer Methods in Applied Mechanics and Engineering},
  volume={374},
  pages={113547},
  year={2021},
  publisher={Elsevier}
}

@article{chen2022bridging,
  title={Bridging traditional and machine learning-based algorithms for solving \mbox{PDE}s: the random feature method},
  author={Chen, Jingrun and Chi, Xurong and Yang, Zhouwang and others},
  journal={Journal of Machine Learning},
  volume={1},
  number={3},
  pages={268--298},
  year={2022}
}

@article{zeng2025high,
  title={High-precision physics-informed extreme learning machines for evolving interface problems},
  author={Zeng, Shaojie and Liang, Yijie and Zhang, Qinghui},
  journal={Engineering Analysis with Boundary Elements},
  volume={174},
  pages={106171},
  year={2025},
  publisher={Elsevier}
}

@article{ying2004decomposition,
  title={A decomposition theorem for the solutions to the interface problems of quasi-linear elliptic equations},
  author={Ying, Lung An},
  journal={Acta Mathematica Sinica},
  volume={20},
  pages={859--868},
  year={2004},
  publisher={Springer}
}

@article{liu2023meshfree,
  title={Meshfree methods for nonlinear equilibrium radiation diffusion equation with jump coefficient},
  author={Liu, Haowei and Liu, Zhiyong and Xu, Qiuyan and Yang, Jiye},
  journal={Computers \& Mathematics with Applications},
  volume={147},
  pages={239--258},
  year={2023},
  publisher={Elsevier}
}

@article{adjerid2025immersed,
  title={Immersed finite element methods for linear and quasi-linear elliptic interface problems},
  author={Adjerid, Slimane},
  journal={Computers \& Mathematics with Applications},
  volume={197},
  pages={19--42},
  year={2025},
  publisher={Elsevier}
}

@article{mishra2022estimates,
  title={Estimates on the generalization error of physics-informed neural networks for approximating a class of inverse problems for \mbox{PDE}s},
  author={Mishra, Siddhartha and Molinaro, Roberto},
  journal={IMA Journal of Numerical Analysis},
  volume={42},
  number={2},
  pages={981--1022},
  year={2022},
  publisher={Oxford University Press}
}

@article{mishra2023estimates,
  title={Estimates on the generalization error of physics-informed neural networks for approximating \mbox{PDE}s},
  author={Mishra, Siddhartha and Molinaro, Roberto},
  journal={IMA Journal of Numerical Analysis},
  volume={43},
  number={1},
  pages={1--43},
  year={2023},
  publisher={Oxford University Press}
}

@article{li2025local,
  title={Local randomized neural networks with finite difference methods for interface problems},
  author={Li, Yunlong and Wang, Fei},
  journal={Journal of Computational Physics},
  volume={529},
  pages={113847},
  year={2025},
  publisher={Elsevier}
}

@article{wang2021new,
  title={New immersed finite volume element method for elliptic interface problems with non-homogeneous jump conditions},
  author={Wang, Quanxiang and Zhang, Zhiyue and Wang, Liqun},
  journal={Journal of Computational Physics},
  volume={427},
  pages={110075},
  year={2021},
  publisher={Elsevier}
}

@article{fan2026hybrid,
    author = {Fan, Chen and Lang, Siyuan and Toseef, Muhammad and Zhang, Zhiyue},
    title = {A hybrid predictor--corrector decoupled method based on operator learning for solving interface problems},
    journal = {Journal of Nonlinear Science},
    year = {2026},
    volume = {36},
    pages = {25}
}

@article{fan2024decoupling,
  title={Decoupling numerical method based on deep neural network for nonlinear degenerate interface problems},
  author={Fan, Chen and Ali, Muhammad Aamir and Zhang, Zhiyue},
  journal={Computer Physics Communications},
  volume={303},
  pages={109275},
  year={2024},
  publisher={Elsevier}
}

@article{zhao2021semi,
  title={Semi-decoupling hybrid asymptotic and augmented finite volume method for nonlinear singular interface problems},
  author={Zhao, Tengjin and Ito, Kazufumi and Zhang, Zhiyue},
  journal={Journal of Computational and Applied Mathematics},
  volume={396},
  pages={113606},
  year={2021},
  publisher={Elsevier}
}

@article{he2022mesh,
  title={A mesh-free method using piecewise deep neural network for elliptic interface problems},
  author={He, Cuiyu and Hu, Xiaozhe and Mu, Lin},
  journal={Journal of Computational and Applied Mathematics},
  volume={412},
  pages={114358},
  year={2022},
  publisher={Elsevier}
}

@article{hu2022discontinuity,
  title={A discontinuity capturing shallow neural network for elliptic interface problems},
  author={Hu, Wei-Fan and Lin, Te-Sheng and Lai, Ming-Chih},
  journal={Journal of Computational Physics},
  volume={469},
  pages={111576},
  year={2022},
  publisher={Elsevier}
}

@article{chang2023hybrid,
  title={A hybrid neural-network and \mbox{MAC} scheme for \mbox{S}tokes interface problems},
  author={Chang, Che-Chia and Dai, Chen-Yang and Hu, Wei-Fan and Lin, Te-Sheng and Lai, Ming-Chih},
  journal={arXiv preprint arXiv:2306.06333},
  year={2023}
}

@article{li2023local,
  title={Local randomized neural networks methods for interface problems},
  author={Li, Yunlong and Wang, Fei},
  journal={arXiv preprint arXiv:2308.03087},
  year={2023}
}

@article{wu2024solving,
  title={Solving parametric elliptic interface problems via interfaced operator network},
  author={Wu, Sidi and Zhu, Aiqing and Tang, Yifa and Lu, Benzhuo},
  journal={Journal of Computational Physics},
  volume={514},
  pages={113217},
  year={2024},
  publisher={Elsevier}
}

@article{chi2024random,
  title={The random feature method for solving interface problems},
  author={Chi, Xurong and Chen, Jingrun and Yang, Zhouwang},
  journal={Computer Methods in Applied Mechanics and Engineering},
  volume={420},
  pages={116719},
  year={2024},
  publisher={Elsevier}
}

@article{bi2025xi,
  title={\mbox{XI-DeepONet}: an operator learning method for elliptic interface problems},
  author={Bi, Ran and Chen, Jingrun and Deng, Weibing},
  journal={Journal of Computational Physics},
  volume={538},
  pages={114164},
  year={2025},
  publisher={Elsevier}
}

@article{zhai2026physics,
  title={Physics-informed neural networks for solving two-phase flow problems with moving interfaces},
  author={Zhai, Qijia and Sun, Pengtao and Xie, Xiaoping and Zhu, Xingwen and Zhang, Chen-Song},
  journal={arXiv preprint arXiv:2604.00948},
  year={2026}
}

@article{hansbo2002unfitted,
  title={An unfitted finite element method, based on \mbox{N}itsche’s method, for elliptic interface problems},
  author={Hansbo, Anita and Hansbo, Peter},
  journal={Computer Methods in Applied Mechanics and Engineering},
  volume={191},
  number={47-48},
  pages={5537--5552},
  year={2002},
  publisher={Elsevier}
}

@article{burman2015cutfem,
  title={\mbox{CutFEM}: discretizing geometry and partial differential equations},
  author={Burman, Erik and Claus, Susanne and Hansbo, Peter and Larson, Mats G and Massing, Andr{\'e}},
  journal={International Journal for Numerical Methods in Engineering},
  volume={104},
  number={7},
  pages={472--501},
  year={2015},
  publisher={Wiley Online Library}
}

@book{nocedal2006numerical,
  title={Numerical optimization},
  author={Nocedal, Jorge and Wright, Stephen J},
  year={2006},
  publisher={Springer}
}

@book{bjorck2024numerical,
  title={Numerical methods for least squares problems},
  author={Bj{\"o}rck, {\AA}ke},
  year={2024},
  publisher={SIAM}
}

@article{hansen1987truncated,
  title={The truncated \mbox{SVD} as a method for regularization},
  author={Hansen, Per Christian},
  journal={BIT Numerical Mathematics},
  volume={27},
  number={4},
  pages={534--553},
  year={1987},
  publisher={Springer}
}

@article{brenner2003poincare,
  title={Poincar{\'e}--\mbox{F}riedrichs Inequalities for Piecewise $H^1$ Functions},
  author={Brenner, Susanne C},
  journal={SIAM Journal on Numerical Analysis},
  volume={41},
  number={1},
  pages={306--324},
  year={2003},
  publisher={SIAM}
}

@article{li2000gradient,
  title={Gradient estimates for solutions to divergence form elliptic equations with discontinuous coefficients},
  author={Li, Yan Yan and Vogelius, Michael},
  journal={Archive for rational mechanics and analysis},
  volume={153},
  number={2},
  pages={91--151},
  year={2000},
  publisher={Springer}
}

@article{li2003estimates,
  title={Estimates for elliptic systems from composite material},
  author={Li, Yanyan and Nirenberg, Louis},
  journal={Communications on Pure and Applied Mathematics: A Journal Issued by the Courant Institute of Mathematical Sciences},
  volume={56},
  number={7},
  pages={892--925},
  year={2003},
  publisher={Wiley Subscription Services, Inc., A Wiley Company New York}
}

@article{dong2019gradient,
  title={Gradient estimates for divergence form elliptic systems arising from composite material},
  author={Dong, Hongjie and Xu, Longjuan},
  journal={SIAM Journal on Mathematical Analysis},
  volume={51},
  number={3},
  pages={2444--2478},
  year={2019},
  publisher={SIAM}
}

@book{girault2012finite,
  title={Finite element methods for \mbox{N}avier-\mbox{S}tokes equations: theory and algorithms},
  author={Girault, Vivette and Raviart, Pierre-Arnaud},
  year={2012},
  publisher={Springer Science \& Business Media}
}

\end{document}